\documentclass[leqno]{amsart}

\usepackage{amsmath,amsthm,amssymb,amscd}
\usepackage{mathrsfs}
\usepackage{indentfirst}
\usepackage{fullpage}
\usepackage{url}
\usepackage[all]{xy}
\usepackage{bbm}

\newcommand{\X}[1]{\ensuremath{X(#1)}}
\newcommand{\Xw}[2]{\ensuremath{X(#1)_{/#2}}}

\newcommand{\ord}[1]{\ensuremath{\eta_p^{\mathrm{ord}}(#1)}}
\newcommand{\et}[1]{\ensuremath{\eta_p^{\mathrm{\acute{e}t}}(#1)}}
\newcommand{\xw}[2]{\ensuremath{\left(X(#1),\eta^{(p)}(#1),\et{#1}\times\ord{#1}\right)_{/#2}}}
	
\numberwithin{equation}{section}

\newtheorem*{qexp}{Theorem}
\newtheorem{factoring}{Lemma}[section]
\newtheorem{improv}[factoring]{Lemma}
\newtheorem{switch}[factoring]{Remark}
\newtheorem{adelic}[factoring]{Proposition}
\newtheorem{frob}{Lemma}[section]
\newtheorem{ver}[frob]{Lemma}
\newtheorem{Serre--Tate}[frob]{Proposition}

\newtheorem{taylor}{Proposition}[section]
\newtheorem{shift}[taylor]{Lemma}
\newtheorem{key}[taylor]{Proposition}
\newtheorem{support}[taylor]{Corollary}
\newtheorem*{recipe}{Recipe for a choice of $\lambda$}
\newtheorem{back}[taylor]{Lemma}
\newtheorem*{main}{Main Theorem}
\newtheorem{remark}[taylor]{Remark}

\begin{document}
\author{Miljan Brako\v cevi\'c}
\title{Anticyclotomic $p$-adic $L$-function of central critical Rankin--Selberg $L$-value}
\date{\today}
\address{Department of Mathematics, UCLA, Los Angeles, CA 90095-1555, USA} 
\email{miljan@math.ucla.edu}
\thanks{This research work is partially supported by Prof. Haruzo Hida's NSF grant DMS-0753991
through graduate student research fellowship and by UCLA Dissertation Year Fellowship.}

\begin{abstract}
\noindent Let $M$ be an imaginary quadratic field, $\mathbf{f}$ a Hecke eigen-cusp form on $\mathrm{GL}_2(\mathbb{Q})\backslash \mathrm{GL}_2(\mathbb{A})$ and $\hat{\pi}_{\mathbf{f}}$ the unitary base-change to $M$ of automorphic representation $\pi_{\mathbf{f}}$ associated to $\mathbf{f}$. Take a unitary arithmetic Hecke character $\chi$ of $M^{\times} \backslash M^{\times}_{\mathbb{A}}$ inducing the inverse of the central character of $\pi_{\mathbf{f}}$. The celebrated formula of Waldspurger relates the square of $L_{\chi}(\mathbf{f}):=\int_{M^{\times} \backslash M^{\times}_{\mathbb{A}}}  \mathbf{f}(t) \chi(t)d^{\times}t$ to the central critical value $L(\frac{1}{2},\hat{\pi}_{\mathbf{f}}\otimes \chi)$. We construct a new $p$-adic $L$-function that interpolates $L_{\chi}(\mathbf{f})$ over arithmetic $\chi$'s for a cusp form $\mathbf{f}$ in the spirit of the landmark result of Katz where he did that for Eisenstein series. 
\end{abstract}
\maketitle
\section{Introduction}
\label{sec:Introduction}
Let $M$ be an imaginary quadratic field, $R$ its ring of integers and $p$ a fixed prime so that the following assumption holds throughout the paper:
\[ \tag{ord}  p\; \text{splits into product of primes}\; p=\mathfrak{p}\bar{\mathfrak{p}}\; \text{in}\; R.\]

\noindent We fix two embeddings   $\iota_\infty :\bar{\mathbb{Q}} \hookrightarrow \mathbb{C}$ and  $\iota_p : \bar{\mathbb{Q}} \hookrightarrow \mathbb{C}_p$ and write $c$ for both complex conjugation of  $\mathbb{C}$ and $\bar{\mathbb{Q}}$ induced by $\iota_\infty$. Then we can choose an embedding $\sigma : M \hookrightarrow \bar{\mathbb{Q}}$ such that the $p$-adic place induced by $\iota_p \circ \sigma$ is distinct from the one induced by $\iota_p \circ c \circ \sigma$. This choice $\Sigma = \{ \sigma \}$ is called $p$-ordinary CM type and its existence is equivalent to (ord). Fix an embedding $M\hookrightarrow \mathrm{M}_2(\mathbb{Q})$ so that we have ${M^{\times} \backslash M^{\times}_{\mathbb{A}}}  \hookrightarrow G(\mathbb{Q})\backslash G(\mathbb{A})$ for the algebraic group  $G=\mathrm{GL}(2)_{/\mathbb{Q}}$. Let $f$  be a normalized Hecke eigen-cusp form of level $\Gamma_0(N)$, $N\geq 1$, weight $k\geq1$, and nebentypus $\psi$ and let $\mathbf{f}$ be its corresponding adelic form on  $G(\mathbb{Q})\backslash G(\mathbb{A})$ with central character $\boldsymbol{\psi}$ (see Section \ref{sec:hecke} for definition). All reasonable adelic lifts of $f$ are equal up to twists by powers of the conductor one character $|\mathrm{det}(g)|_{\mathbb{A}}$ and $\mathbf{f}^u(g) := \mathbf{f}|\boldsymbol{\psi}(\mathrm{det}(g))|^{-1/2}$ is the unique one which generates a unitary automorphic representation  $\pi_\mathbf{f}$. We take the base-change lift  $\hat{\pi}_{\mathbf{f}}$ to $\mathrm{Res}_{M/\mathbb{Q}}G$. Pick an arithmetic Hecke character $\chi$ of $M^{\times} \backslash M^{\times}_{\mathbb{A}}$ such that $\chi|_{\mathbb{A}^\times} = \boldsymbol{\psi}^{-1}$ and set $\chi^-:=(\chi \circ c)/|\chi|$. The celebrated formula of Waldspurger \cite{Wa} relates the square of $L_{\chi}(\mathbf{f}):=\int_{M^{\times} \backslash M^{\times}_{\mathbb{A}}}  \mathbf{f}(t) \chi(t)d^{\times}t$ to the central critical value $L(\frac{1}{2},\hat{\pi}_{\mathbf{f}}\otimes \chi^-)$ of Rankin--Selberg convolution of $f$ and theta series $\theta(\chi^-)$ associated to $\chi^-$, up to finitely many ambiguous local factors. In the recent paper \cite{HidaNV} Hida computes explicit formula for $L_{\chi}(\mathbf{f})^2$ without any ambiguity covering all arithmetic Hecke characters $\chi$ with $\chi|_{\mathbb{A}^\times} = \boldsymbol{\psi}^{-1}$ producing the central critical value. 

The goal of this paper is to interpolate $p$-adically $L_{\chi}(\mathbf{f})$ over arithmetic $\chi$'s for a cusp form $\mathbf{f}$ in the spirit of the landmark paper of Katz \cite{Ka} where he did that for Eisenstein series. To briefly outline our idea,  let $N=\prod_ll^{\nu(l)}$ be the prime factorization and denote by $N_{ns}=\prod_{l \, \text{non-split}}l^{\nu(l)}$ its ``non-split'' part. Consider the order $R_{N_{ns}p^n}:=\mathbb{Z}+N_{ns}p^nR$ and let $\mathrm{Cl}_M^-(N_{ns}p^n):=\mathrm{Pic}(R_{N_{ns}p^n})$, that is, the group of $R_{N_{ns}p^n}$-projective fractional ideals modulo globally principal ideals. By class field theory, $\mathrm{Cl}_M^-(N_{ns}p^n)$ is the Galois group $\mathrm{Gal}(H_{N_{ns}p^n}/M)$ of the ring class field $H_{N_{ns}p^n}$ of conductor $N_{ns}p^n$. We define anticyclotomic class group modulo $N_{ns}p^\infty$, $\mathrm{Cl}_M^-(N_{ns}p^\infty):=\varprojlim_n\mathrm{Cl}_M^-(N_{ns}p^n)$ for the projection $\pi_{m+n,n}:\mathrm{Cl}_M^-(N_{ns}p^{m+n})\to\mathrm{Cl}_M^-(N_{ns}p^n)$ taking $\mathfrak{a}$ to $\mathfrak{a}R_{N_{ns}p^n}$. By class field theory, the group $\mathrm{Cl}_M^-(N_{ns}p^\infty)$ is isomorphic to the Galois group of the maximal ring class field of conductor $N_{ns}p^\infty$  of $M$, namely $H_{N_{ns}p^\infty}=\bigcup_n H_{N_{ns}p^n}$. To accomplish our goal we construct a $W$-valued $p$-adic measure $\mathrm{d}\mu_f$ on $\mathrm{Cl}_M^-(N_{ns}p^\infty)$ such that its moments interpolate the special values at carefully chosen CM points of Hecke eigen-cusp form $f$ and $p$-adic modular forms obtained by applying Katz's $p$-adic differential operator to $f$. Here $W$ is the ring of Witt vectors with coefficients in an algebraic closure $\bar{\mathbb{F}}_p$ of the finite field of $p$ elements $\mathbb{F}_p$, regarded as $p$-adically closed discrete valuation ring inside $p$-adic completion $\mathbb{C}_p$ of $\bar{\mathbb{Q}}_p$. We set $\mathcal{W}=\iota_p^{-1}(W)$ which is a strict henselization of $\mathbb{Z}_{(p)}=\mathbb{Q}\cap\mathbb{Z}_p$. A subtle construction of CM points on a Shimura variety that underlies the central critical values  $L(\frac{1}{2},\hat{\pi}_{\mathbf{f}}\otimes \chi^-)$ allows us to relate $\int_{\mathrm{Cl}_M^-(N_{ns}p^\infty)}\chi\mathrm{d}\mu_f$ for a very broad class of infinite order arithmetic Hecke characters $\chi$ to the ``square root'' of $L(\frac{1}{2},\hat{\pi}_{\mathbf{f}}\otimes \chi^-)$ by utilizing Hida's computation. These CM points are associated to proper ideal classes of $R_{N_{ns}p^n}$ and carry elliptic curves with complex multiplication by $M$ having ordinary reduction over $\mathcal{W}$ and equipped with suitable level structures also defined over $\mathcal{W}$. Thus, if we denote by $\delta_k$ the differential operator of Maass--Shimura  and by $d$ its $p$-adic analogue invented by Katz, then for a modular form $f$ integral over $\mathcal{W}$, the values of $\delta_k^mf$, $m\geq0$, at our CM points lie in $\mathcal{W}$ and, up to canonical period, coincide with values of $d^mf$ at the same points by the results of Shimura and Katz (\cite{Ka} and \cite{Sh75}).

The Shimura variety $Sh$ in question is Shimura curve classifying elliptic curves up to isogeny constructed by Shimura in \cite{Sh66} and reinterpreted by Deligne in \cite{De71}. 
Subtleties of Hida's computation of explicit Waldspurger formula required careful study of the arithmetic geometry of this Shimura curve. Ultimately, Hecke relation among CM points on $Sh$ holds the key to the distribution relation of the cuspidal measure that we construct. Second, we study the effect of Hecke operators on the Serre--Tate coordinate $t$ in the infinitesimal neighborhoods of CM points. This is one of the novelties of the paper and we precisely determine how certain Hecke isogeny actions on the Igusa tower over an irreducible component of $Sh$ containing a CM point affect the Serre--Tate deformation space of elliptic curve carried by this CM point (see Proposition \ref{Serre--Tate}). This enables us to establish an interpolation formula in great generality.

Let $\varphi:M^{\times} \backslash M^{\times}_{\mathbb{A}}\to\mathbb{C}^\times$ be an anticyclotomic arithmetic Hecke character such that $\varphi(a_\infty)=a_\infty^{m(1-c)}$ for some $m\geq 0$ (i.e. of infinity type $(m,-m)$) and of conductor $N_{ns}p^s$ where $s\geq \mathrm{ord}_p(N)$ is an \emph{arbitrary} integer or of conductor $N_{ns}$. Let $\widehat{\varphi}:M^{\times} \backslash M^{\times}_{\mathbb{A}}\to \mathbb{C}_p^\times$ defined by $\widehat{\varphi}(x)=\varphi(x)x_p^{m(1-c)}$ be its $p$-adic avatar. Then Mazur--Mellin transform of the measure $\mathrm{d}\mu_f$ given by
\[ \mathscr{L}(f,\xi)=\int_{\mathrm{Cl}_M^-(N_{ns}p^\infty)}\xi{d}\mu_f \text{ for } \xi \in \mathrm{Hom}_{cont}(\mathrm{Cl}_M^-(N_{ns}p^\infty),W^\times), \]
is a $p$-adic analytic Iwasawa function on the $p$-adic Lie group $\mathrm{Hom}_{cont}(\mathrm{Cl}_M^-(N_{ns}p^\infty),W^\times)$ and for characters $\varphi$ as above we have the following interpolation property (see Main theorem in Section \ref{sec:mainthm}): 
\[\left(\frac{\mathscr{L}(f,\widehat{\varphi})}{\mathrm{\Omega}_p^{k+2m}}\right)^2=C(\lambda,\varphi,m)\frac{L(\frac{1}{2},\hat{\pi}_{\mathbf{f}}\otimes (\varphi\lambda)^-)}{\left(\mathrm{\Omega}_\infty^{k+2m}\right)^2}\] 
for an explicit constant $C(\lambda,\varphi,m)$, where $\lambda$ is a fixed choice of Hecke character of $M^{\times} \backslash M^{\times}_{\mathbb{A}}$ of infinity type $(k,0)$ such that $\lambda|_{\mathbb{A}^\times} = \boldsymbol{\psi}^{-1}$, and $\mathrm{\Omega}_p\in W^\times$ and $\mathrm{\Omega}_\infty\in\mathbb{C}^\times$ are N\'eron periods of an elliptic curve of CM type $\Sigma$ defined over $\mathcal{W}$.

We emphasize that no condition on the conductor $N$ of cusp form $f$ is ever imposed and that the interpolation formula is optimal in the sense that we don't impose any restrictions on the conductor of arithmetic Hecke character $\varphi\lambda$ other than the optimal ones imposed by the criticality assumption $\lambda|_{\mathbb{A}^\times} = \boldsymbol{\psi}^{-1}$. In particular, conductor at split primes is also allowed (see Remark \ref{remark}).
The main theorem is stated at the end of the paper, after providing all necessary notation and collecting running assumptions.  Fix a decomposition $\mathrm{Cl}_M^-(N_{ns}p^\infty)=\Delta\times\Gamma$ where $\Gamma$ is a torsion free subgroup topologically isomorphic to $\mathbb{Z}_p$ and $\Delta$ is a finite group. An arbitrary character $\xi:\mathrm{Cl}_M^-(N_{ns}p^\infty)\to W^\times$ can be uniquely represented in the form $\xi=\xi_{\Delta}\xi_{\Gamma}$ where $\xi_{\Delta}$ is a character of $\Delta$ trivial on $\Gamma$, and $\xi_{\Gamma}$ is a character of $\Gamma$ trivial on $\Delta$. Denote by $\langle \cdot \rangle:\mathrm{Cl}_M^-(N_{ns}p^\infty)\to \Gamma$ the natural projection, and let $\xi_{\Gamma,t}$ denotes the torsion part of $\xi_{\Gamma}$. Our interpolation formula covers $p$-adic avatars of the form $\xi_{\Delta}\xi_{\Gamma,t}\langle \cdot \rangle^m$ where $\xi_{\Delta}$, $\xi_{\Gamma,t}$ and $m\geq 0$ are arbitrary and moving independently.

This $p$-adic $L$-function exhibits crucial distinction from the  $p$-adic $L$-functions relevant to the central critical value $L(\frac{1}{2},\hat{\pi}_{\mathbf{f}}\otimes \chi)$ constructed by Vatsal in \cite{Va} following the method of Bertolini and Darmon from \cite{BD}, and by Perrin-Riou in \cite{P-R} following the method of Hida from \cite{H85} under slightly different assumptions. First of all, Vatsal's and Perrin-Riou's constructions are limited to weight 2 cusp form $f$ and finite order characters $\chi$ whereas we don't impose any restriction on the weight $k\geq1$ of $f$ and allow infinitely many infinite order characters in the interpolation. Second, both constructions are made under the assumption that the cusp form $f$ is ordinary at $p$, a restriction we don't impose. Most importantly, both constructions are characterized by the weight of cusp form $f$ being strictly greater than the weight of theta series $\theta(\chi)$ associated to Hecke character $\chi$, whereas we treat exactly the opposite case. Not only does this make canonical period in their constructions to be the one constructed by Hida in \cite{H88ajm}, namely
\[ \mathrm{\Omega}_f=\frac{\langle f,f\rangle}{\eta_0}\] 
where $\langle f,f\rangle$ is the Petersson inner product on $\Gamma_0(N)$ and $\eta_0$ is Hida's congruence number associated to $f$, whereas our canonical period $\mathrm{\Omega}_\infty$ is N\'eron period of an elliptic curve of CM type $\Sigma$ and is more natural in the context of Iwasawa theory, but also the canonical Selmer group and hence the Main Conjecture associated to our $p$-adic $L$-function are distinct from the ones in Vatsal's and Perrin-Riou's cases. 

Hida on the other hand constructs in great generality (\cite{H85} and \cite{H88aif}) a $p$-adic $L$-function that interpolates the central critical value
of the Rankin--Selberg convolution of two independent $p$-adic families of modular forms and our $p$-adic $L$-function is essentially the ``square root'' of the restriction of his general $p$-adic $L$-function to one of the variables. The precise relationship is a work in progress and will be published in a forthcoming paper, since Hida  still uses  his canonical period $\Omega_f$ as above, and his $p$-adic $L$-function interpolates the central critical $L$-value itself, rather than its ``square root'' as in our case (the same remark holds for Perrin-Riou's construction following his method). 

In the course of preparations to submit this paper, we became aware of the recent preprint of Bertolini, Darmon and Prasanna (\cite{BDP}) where they construct anticyclotomic $p$-adic $L$-function attached to normalized Hecke newform $f$ and field $M$ interpolating the central critical Rankin--Selberg $L$-value. We would like to emphasize that our work is done independently and under significantly different set of running assumptions making certain aspects of our construction substantially different. The spirit of our work also differs in the sense that their main goal goes beyond a construction of $p$-adic $L$-function as they prove a deep and beautiful theorem that is a $p$-adic analogue of Gross--Zagier formula, whereas our main goal is a cuspidal version of Eisenstein measure in \cite{Ka} whose Mazur--Mellin transform gives rise to a $p$-adic $L$-function interpolating the ``square root'' of the central critical $L$-value. In \cite{BDP} (see Assumption 6.1) the authors assume that the level $N$ of a cusp form $f$ is squarefree and the Heegner hypothesis holds; the conductor $c$ of order to which Heegner points are associated (\cite{BDP} Section 1.5) is fixed and prime to $N$ and the conductor $\mathfrak{f}_\chi$ of central critical Hecke character $\chi$ appearing in interpolation is \emph{a priori} bounded by condition $\mathfrak{f}_\chi |c\mathfrak{N}$, where $\mathfrak{N}$ is a fixed cyclic ideal of norm $N$ of the ring of integers of $M$. Following the notation there, we denote the set of such characters $\Sigma^{(2)}_{cc}(c,\mathfrak{N})$. Note that these conditions exclude the possibility that both $N$ and $\mathfrak{f}_\chi$ are divisible by $p$. On the other hand, we impose no squarefree restriction on the level $N$ of cusp form and no Heegner hypothesis is assumed. In fact, the assumptions above are opposite from assumptions we are running in the paper, namely if $N$ happens to be divisible by $p$, say $p^r\| N$ for \emph{arbitrary} $r\geq 0$, then we work with arithmetic Hecke characters whose conductor at $p$ is $p^s$ for \emph{arbitrary} $s\geq r$, and our CM points are then associated to orders of $M$ of conductor divisible by $p^s$. The difference of our interpolation formula from theirs becomes evident in the case when (apart from the conductor) the nebentypus of $f$ also ramifies at $p$. If we set $\chi_m=\lambda \cdot \varphi \cdot |\cdot|_{M_\mathbb{A}}^m$ for $\lambda$ fixed as above and we let anticyclotomic $\varphi$ vary as above, then $\chi_m$ is central critical character of $\infty$-type $(k+m,-m)$ and the authors would have to assume $\varphi$ to be unramified at $p$ (\cite{BDP} Proposition 5.3). On the other hand, we allow both the cusp form and Hecke character to be unramified at $p$.  Second, apart from being attached to $f$ and $M$, the $p$-adic $L$-function there is dependent on the pair $(c,\mathfrak{N})$ in the sense that $\Sigma^{(2)}_{cc}(c,\mathfrak{N})$ is equipped with topology induced by the topology of uniform convergence on the subgroup of $M_{\mathbb{A}}^{(\infty)}$ of ideles that are prime to $p$ and the $p$-adic $L$-function is defined as $p$\emph{-adically continuous function} on the completion of $\Sigma^{(2)}_{cc}(c,\mathfrak{N})$ with respect to this topology. (\cite{BDP} Section 5.4). On the other hand, our $p$-adic $L$-function is a $p$\emph{-adic analytic Iwasawa function} on the $p$-adic Lie group $\mathrm{Hom}_{cont}(\mathrm{Cl}_M^-(N_{ns}p^\infty),W^\times)$ as the Mazur--Mellin transfrom of a bounded measure. 

During preparations to submit this paper, Kartik Prasanna pointed us to the work of Andrea Mori (\cite{Mo}) in which the author constructs a bounded measure on $\mathbb{Z}_p$ whose power series expansion is Serre--Tate $t$-expansion of $f$ around a fixed CM point (in the sense of our Proposition \ref{taylor}) -- we learned that this idea is originally due to A. Mori and goes back to his Brandeis University thesis in 1989. However consideration in \cite{Mo} is entirely devoted to a single fixed CM point throughout the paper and no moving of it on Shimura curve was ever considered, hence the measure is defined on $\mathbb{Z}_p$ and \emph{a priori} depends on Hecke characters that will ultimately appear in interpolation, the latter being a limitation. From the point of view of our paper, construction in \cite{Mo} is relevant to a special case: when imaginary quadratic field $M$ is of class number 1, a restricted set of everywhere unramified Hecke characters is being interpolated and the level $N$ of an even weight cusp form $f$ is prime to $p$ and squarefree.

A nice feature of our result lies in the fact that the Bloch--Kato conjecture suggests that the central critical value $L(\frac{1}{2},\hat{\pi}_{\mathbf{f}}\otimes \chi)$ is a square in the coefficient field $\mathbb{Q}(f,\chi)$ up to a product of philosophically well understood Tamagawa factors. Our construction yields that the central critical value is equal to $L_{\chi}(\mathbf{f})^2$ up to a product of fudge Euler factors. We plan to compute explicitly Tamagawa factors for $L(\frac{1}{2},\hat{\pi}_{\mathbf{f}}\otimes \chi)$ and relate them to the fudge Euler factors thus establishing the assertion independently of the Bloch--Kato conjecture. 

We make two applications of our construction. In \cite{Br1} we study non-vanishing modulo $p$ of central critical Rankin--Selberg $L$-values with anticyclotomic twists using Hida's method from \cite{HidaDwork} while in \cite{Br2} we compute the Iwasawa $\mu$-invariant of the constructed anticyclotomic $p$-adic $L$-function.  

\subsection*{Acknowledgment} I would like to express my gratitude to Prof. Haruzo Hida for his generous insight and guidance, as well as constant care and support. I am indebted to the anonymous referee for careful reading and providing corrections, valuable comments and suggestions that improved the exposition of this paper. 
\tableofcontents
\newpage
\section{Elliptic curves with complex multiplication}
\label{sec:eccm}

For each $\mathbb{Z}$-lattice $\mathfrak{a}\subset M$ whose $p$-adic completion $\mathfrak{a}_p=\mathfrak{a}\otimes_\mathbb{Z}\mathbb{Z}_p$ is identical to $R\otimes_\mathbb{Z}\mathbb{Z}_p$ we consider a complex torus $\X{\mathfrak{a}}(\mathbb{C})=\mathbb{C}/\mathfrak{a}$. It follows from Shimura-Taniyama theory (\cite{ACM} 12.4) that this complex torus is algebraizable to an elliptic curve having complex multiplication by $M$ with CM type $\Sigma$ and which is defined over a number field. More precisely, if $M'$ is a reflex field of $(M,\Sigma)$, essentially by the main theorem of complex multiplication (\cite{ACM} 18.6) we can find a model \X{\mathfrak{a}} defined over an abelian extension $k$ of $M'$ such that all torsion points of \X{\mathfrak{a}} are rational over an abelian extension of $M'$ (\cite{ACM} 21.1). Adding sufficiently deep level structure we get a model that is unique up to an isomorphism. Serre--Tate's criterion of good reduction (\cite{SeTa}) implies, by further deepening the level structure, that \X{\mathfrak{a}} has good reduction over $W\cap k$. It follows that \X{\mathfrak{a}} is actually defined over the field of fractions $\mathcal{K}$ of $\mathcal{W}$ and extends to an elliptic curve over $\mathcal{W}$ still denoted by \Xw{\mathfrak{a}}{\mathcal{W}}. All endomorphisms of \Xw{\mathfrak{a}}{\mathcal{W}} are defined over $\mathcal{W}$ and its special fiber $\Xw{\mathfrak{a}}{\bar{\mathbb{F}}_p}=\Xw{\mathfrak{a}}{\mathcal{W}}\otimes \bar{\mathbb{F}}_p$ is ordinary by our assumption (ord). 

It is well known that a $\mathbb{Z}$-lattice in $M$ is actually a proper ideal of an order of $M$. Indeed, if $R(\mathfrak{a})=\{\alpha\in R|\alpha\mathfrak{a}\subset\mathfrak{a}\}$ then $R(\mathfrak{a})$ is a $\mathbb{Z}$-order of $M$. Every $\mathbb{Z}$-order $\mathcal{O}$ of $M$ is of the form $\mathcal{O}=\mathbb{Z}+cR$ for a rational integer $c$ called the conductor and the following are equivalent (see Proposition 4.11 and (5.4.2) in \cite{IAT} and Theorem 11.3 of \cite{CRT})
\begin{itemize}
\item[(1)] $\mathfrak{a}$ is $\mathcal{O}$-projective fractional ideal
\item[(2)] $\mathfrak{a}$ is locally principal, i.e. \frenchspacing the localization at each prime is principal
\item[(3)] $\mathfrak{a}$ is a proper $\mathcal{O}$-ideal, i.e. \frenchspacing $\mathcal{O}=R(\mathfrak{a})$.  
\end{itemize}
Thus, one can define class group $\mathrm{Cl}_M^-(\mathcal{O})$ to be the group of $\mathcal{O}$-projective fractional ideals modulo globally principal ideals. It is a finite group called the ring class group of conductor $c$ where $c$ is the conductor of $\mathcal{O}$. As pointed out in the introduction, we restrict our attention to $\mathrm{Cl}_M^-(cp^n)=\mathrm{Pic}(R_{cp^n})$ for $R_{cp^n}=\mathbb{Z}+cp^nR$, where $c$ is a suitable choice of integer prime to $p$. If $\mathrm{Cl}_M(cp^n)$ and $\mathrm{Cl}_{\mathbb{Q}}(cp^n)$ denote ray class groups modulo $cp^n$ of $M$ and $\mathbb{Q}$ respectively, then exact sequence

\[ 0\longrightarrow \mathrm{Cl}_{\mathbb{Q}}(cp^n)\longrightarrow \mathrm{Cl}_M(cp^n)\longrightarrow \mathrm{Cl}_M^-(cp^n)\longrightarrow 0\] yields
\[ |\mathrm{Cl}_M^-(cp^n)|=\frac{2h(M)\varphi_M(cp^n)}{|R^\times|\varphi_\mathbb{Q}(cp^n)}\]
where $h(M)$ is the class number of $M$ and $\varphi_M$ and $\varphi_\mathbb{Q}$ are Euler phi functions of $M$ and $\mathbb{Q}$ respectively. The adelic interpretation of $\mathrm{Cl}_M^-(cp^n)$ is given by
\[ \mathrm{Cl}_M^-(cp^n) = M^\times \left\backslash (M^{(\infty)}_{\mathbb{A}})^{\times}\right/(\mathbb{A}^{(\infty)})^\times \widehat{R}_{cp^n}^{\times} \]
where $\widehat{R}_{cp^n}={R}_{cp^n} \otimes_\mathbb{Z}\widehat{\mathbb{Z}}$.

\section{Algebro-geometric and $p$-adic modular forms}

Mainly to set a notation, we briefly recall basic definitions and facts from algebro-geometric theory of modular forms restricting ourselves to what we are going to need in the sequel.

\subsection{Algebro-geometric modular forms}\label{sec:agmf} 
Let $N$ be a positive integer and $B$ a fixed base $\mathbb{Z}[\frac{1}{N}]$-algebra. The modular curve $\mathfrak{M}(N)$ of level $N$ classifies pairs $(E,i_N)_{/S}$, for a $B$-scheme $S$, formed by
\begin{itemize}
\item An elliptic curve $E$ over $S$, that is, a proper smooth morphism $\pi : E\to S$ whose geometric fibers are connected curves of genus 1, together with a section $\textbf{0}:S\to E$.
\item An embedding of finite flat group schemes $i_N:\mu_N\hookrightarrow E[N]$, called level $\Gamma_1(N)$-structure,  where $E[N]$ is a scheme-theoretic kernel of multiplication by $N$ map -- it is a finite flat abelian group scheme over $S$ of rank $N^2$.
\end{itemize}
In other words, $\mathfrak{M}(N)$ is a coarse moduli scheme, and a fine moduli scheme if $N>3$, of the following functor from the category of $B$-schemes to the category \textit{SETS}
\[ \mathcal{P}(S)=[(E,i_N)_{/S}]_{/\cong}\]
where $[\quad]_{/\cong}$ denotes the set of isomorphism classes of the objects inside the brackets. If $\omega$ is a basis of $\pi_*(\Omega_{E/S})$, that is a nowhere vanishing section of $\Omega_{E/S}$, one can further consider a functor classifying triples $(E,i_N,\omega)_{/S}$:
\[ \mathcal{Q}(S)=[(E,i_N,\omega)_{/S}]_{/\cong}\; .\]
Then $a\in \mathbb{G}_m(S)$ acts on $\mathcal{Q}(S)$ via $(E,i_N,\omega)\mapsto (E,i_N,a\omega)$ and consequently $\mathcal{Q}$ is a $\mathbb{G}_m$-torsor over $\mathcal{P}$. This yields representability of $\mathcal{Q}$ by a $B$-scheme $\mathcal{M}(N)$ affine over $\mathfrak{M}(N)_{/B}$.

Fix a positive integer $k$ and a continuous character $\psi : (\mathbb{Z}/N\mathbb{Z})^\times\rightarrow B^\times$. Denote by $\zeta_N$ the canonical generator of $\mu_N$. A $B$-integral holomorphic modular form of weight $k$, level $\Gamma_0(N)$ and nebentypus $\psi$ is a function of isomorphism classes of $(E,i_N,\omega)_{/A}$, defined over $B$-algebra $A$, satisfying the following conditions:
\begin{itemize}
 \item[(G0)] $f((E,i_N,\omega)_{/A})\in A$ if $(E,i_N,\omega)$ is defined over $A$
 \item[(G1)] If $\varrho:A\rightarrow A'$ is a morphism of $B$-algebras then $f((E,i_N,\omega)_{/A}\otimes_{B} A')=\varrho(f((E,i_N,\omega)_{/A}))$;
 \item[(G2)] $f((E,i_N,a\omega)_{/A})=a^{-k}f((E,i_N,\omega)_{/A})$ for $a\in A^\times=\mathbb{G}_m(A)$;
 \item[(G3)] $f((E, i_N\circ b,\omega)_{/A})=\psi(b)f((E,i_N,\omega)_{/A})$ for $b \in (\mathbb{Z}/N\mathbb{Z})^\times$, where $b$ acts on $i_N$  by the canonical action of $\mathbb{Z}/N\mathbb{Z}$ on the finite flat group scheme $\mu_N$;
 \item[(G4)] For the Tate curve $Tate(q^N)$ over $B\otimes_{\mathbb{Z}}\mathbb{Z}((q))$ viewed as algebraization of formal quotient        
   $\widehat{\mathbb{G}}_m/q^{N\mathbb{Z}}$, its canonical differential $\omega_{Tate}^{can}$ deduced from $\frac{\mathrm{d}t}{t}$  on $\widehat{\mathbb{G}}_m$,  the canonical level $\Gamma_1(N)$-structure $i_{Tate,N}^{can}$ coming from the canonical image of the point $\zeta_N$ from  $\widehat{\mathbb{G}}_m$, and all $\alpha \in \mathrm{Aut}(Tate(q^N)[N]) \cong G(\mathbb{Z}/N\mathbb{Z})$, we have
   \[ f((Tate(q^N),\alpha\circ i_{Tate,N}^{can},\omega_{Tate}^{can}))\in B\otimes_{\mathbb{Z}}\mathbb{Z}[[q]]\; .\] 
\end{itemize}
The space of $B$-integral holomorphic modular forms of weight $k$, level $\Gamma_0(N)$ and nebentypus $\psi$ is a $B$-module of finite type and we denote it by $G_k(N,\psi;B)$.

\subsection{$p$-adic modular forms}
\label{p-adicmf}
Fix a prime number $p$ that does not divide $N$. Let $B$ be an algebra that is complete and separated in its $p$-adic topology; such algebras are called $p$-adic algebras. For an abelian scheme $E_{/S}$ we consider a Barsotti--Tate group $E[p^\infty]=\underrightarrow{\lim}_nE[p^n]$ for finite flat group schemes $E[p^n]$ equipped with closed immersions $E[p^n]\hookrightarrow E[p^m]$ for $m>n$ and the multiplication $[p^{m-n}]:E[p^m]\to E[p^n]$ which is an epimorphism in the category of finite flat group schemes. Considering a morphism of ind-group schemes $i_p:\mu_{p^\infty}\hookrightarrow E[p^\infty]$ we have a functor
\[ \widehat{\mathcal{P}}_N(A)=[(E,i_N,i_p)_{/A}]_{/\cong} \]
defined over the category of $p$-adic $B$-algebras $A$. By a theorem of Deligne--Ribet and Katz, this functor is pro-represented by the formal completion $\widehat{\mathfrak{M}}(Np^\infty)$ of $\mathfrak{M}(N)$ along the ordinary locus of its modulo $p$ fiber. A holomorphic $p$-adic modular form over $B$ is a function of isomorphism classes of $(E,i_N,i_p)_{/A}$, defined over $p$-adic $B$-algebra $A$, satisfying the following conditions:
\begin{itemize}
	\item[(P0)] $f((E,i_N,i_p)_{/A})\in A$ if $(E,i_N,i_p)$ is defined over $A$;
	\item[(P1)] If $\varrho:A\rightarrow A'$ is a $p$-adically continuous morphism of $B$-algebras then $f((E,i_N,i_p)_{/A}\otimes_{B} A')=\varrho(f((E,i_N,i_p)_{/A}))$;
	\item[(P2)]  For the Tate curve $Tate(q^N)$ over $\widehat{B((q))}$, which is a $p$-adic completion of $B((q))$, the canonical $p^\infty$-structure $i_{Tate,p}^{can}$, the canonical level $\Gamma_1(N)$-structure $i_{Tate,N}^{can}$, all $p$-adic units $z\in \mathbb{Z}_p^\times$ and all $\alpha \in \mathrm{Aut}(Tate(q^N)[N]) \cong  G(\mathbb{Z}/N\mathbb{Z})$, we have 
\[ f((Tate(q^N),\alpha \circ i_{Tate,N}^{can},z\circ i_{Tate,p}^{can}))\in B[[q]] \;.\] 
\end{itemize}
We denote the space of $p$-adic holomorphic modular forms over $B$ by $V(N;B)$. 

The fundamental $q$-expansion principle holds for both algebro-geometric and $p$-adic modular forms (\cite{DeRa} Theorem VII.3.9 and \cite{Ka76} Section 5):
\begin{qexp}[$q$-expansion principle]
\begin{itemize}
	\item[1.] The $q$-expansion maps $G_k(N,\psi;B)\to B[[q]]$ and $V(N;B)\to B[[q]]$ are injective for any ($p$-adic) algebra $B$.
	\item[2.] Let $B\subset B'$ be ($p$-adic) algebras. The following commutative diagrams 
\begin{center} 
$\begin{CD}
G_k(N,\psi;B) @>>> B[[q]] \\
@VVV @VVV\\
G_k(N,\psi;B') @>>> B'[[q]]
\end{CD} \qquad \begin{CD}
V(N;B) @>>> B[[q]] \\
@VVV @VVV\\
V(N;B') @>>> B'[[q]]
\end{CD}$ 
\end{center}
\end{itemize}
are Cartesian, that is, the image of $G_k(N,\psi;B)$ in $G_k(N,\psi;B')$ ($V(N;B)$ in $V(N;B')$) is precisely the set of ($p$-adic) modular forms whose $q$-expansions have coefficients in B.
\end{qexp}
Note that $p^\infty$-level structure $i_p:\mu_{p^\infty}\hookrightarrow E[p^\infty]$ over $A$ induces isomorphism of formal groups $\hat{i}_p:\widehat{\mathbb{G}}_m\cong\widehat{E}$ over $A$ called trivialization of $E$, where $\widehat{E}$ is the formal completion of $E$ along its zero-section. More precisely, giving a trivialization $\hat{i}_p$ is equivalent to giving a $p^\infty$-level structure $i_p$ due to equivalence between the
categories of $p$-divisible smooth connected abelian formal groups over $A$ and of
connected $p$-divisible groups over $A$. Using trivialization $\hat{i}_p$ we can push forward the canonical differential $\frac{\mathrm{d}t}{t}$ on $\widehat{\mathbb{G}}_m$ to obtain an invariant differential $\omega_p:=\hat{i}_{p,*}(\frac{\mathrm{d}t}{t})$ on $\widehat{E}$ which then extends to an invariant differential on $E$. Thus for $f\in G_k(N,\psi;B)$ we can define
\[f((E,i_N,i_p)):=f((E,i_N,\omega_p))\]
allowing us to regard an algebro-geometric holomorphic modular form as a $p$-adic one. It follows from $q$-expansion principle that $G_k(N,\psi;B)\hookrightarrow V(N;B)$ is an injection preserving $q$-expansions.

\section{Shimura curves}
\label{sec:shvar}
Recall that in Section \ref{sec:eccm}, for a $\mathbb{Z}$-lattice $\mathfrak{a}\subset M$ whose $p$-adic completion $\mathfrak{a}_p=\mathfrak{a}\otimes_\mathbb{Z}\mathbb{Z}_p$ is identical to $R\otimes_\mathbb{Z}\mathbb{Z}_p$, we constructed an elliptic curve \X{\mathfrak{a}} defined over $\mathcal{W}$. Let $\mathcal{T}(\X{\mathfrak{a}})=\varprojlim_N\X{\mathfrak{a}}[N](\bar{\mathbb{Q}})$ be its Tate module. Strictly speaking, to define the Tate module of an elliptic curve $E_{/A}$ defined over a subring $A$ of $\bar{\mathbb{Q}}$ we take a geometric point $s=\mathop{Spec}(\bar{\mathbb{Q}})\in \mathop{Spec}(A)$ in each connected component of $\mathop{Spec}(A)$ and set $\mathcal{T}(E)=\varprojlim_N E[N](\bar{\mathbb{Q}})$. 

Note that any $\widehat{\mathbb{Z}}$-basis $(w_1,w_2)$ of $\widehat{\mathfrak{a}}=\mathfrak{a}\otimes_{\mathbb{Z}}\widehat{\mathbb{Z}}$ gives rise to a level $N$-structure $\eta_N(\mathfrak{a}):(\mathbb{Z}/N\mathbb{Z})^2\cong \X{\mathfrak{a}}[N]$ given by $\eta_N(\mathfrak{a})(x,y)=\frac{xw_1+xw_2}{N}\in \X{\mathfrak{a}}[N]$. After taking their inverse limit and tensoring with $\mathbb{A}^{(\infty)}$ we get level structure
\[ \eta(\mathfrak{a}) = \varprojlim_N\eta_N(\mathfrak{a}):(\mathbb{A}^{(\infty)})^2\cong \mathcal{T}(\X{\mathfrak{a}})\otimes_{\widehat{\mathbb{Z}}}\mathbb{A}^{(\infty)}=:V(\X{\mathfrak{a}}) \; . \] 
We can remove $p$-part of $\eta(\mathfrak{a})$ and define level structure $\eta^{(p)}(\mathfrak{a})$ that conveys information about all prime-to-$p$ torsion in \X{\mathfrak{a}}:
\[ \eta^{(p)}(\mathfrak{a}):(\mathbb{A}^{(p\infty)})^2\cong
\mathcal{T}(\X{\mathfrak{a}})\otimes_{\widehat{\mathbb{Z}}}\mathbb{A}^{(p\infty)}=:V^{(p)}(\X{\mathfrak{a}}) \; . \]
Since prime-to-$p$ torsion in \Xw{\mathfrak{a}}{\mathcal{W}} is unramified at $p$ and $\X{\mathfrak{a}}[N]$ for $p\nmid N$ is \'etale whence constant over $\mathcal{W}$, the level structure $\eta^{(p)}(\mathfrak{a})$ is still defined over $\mathcal{W}$ (\cite{ACM} 21.1 and \cite{SeTa}).

Given affine algebraic group $G=\mathrm{GL}(2)_{/\mathbb{Q}}$ let $\mathbb{S}=\mathrm{Res}_{\mathbb{C}/\mathbb{R}}\mathbb{G}_m$ and denote by $h_{\mathbf{0}}:\mathbb{S}\to G_{/\mathbb{R}}$ the homomorphism of real algebraic groups sending $a+b\mathbf{i}$ to the matrix $\bigl(\begin{smallmatrix} a & -b \\ b & a \end{smallmatrix} \bigr)$. The symmetric domain $\mathfrak{X}$ for $G(\mathbb{R})$ can be identified with conjugacy class of $h_{\mathbf{0}}$ under $G(\mathbb{R})$ and is isomorphic to the union $\mathfrak{H}\sqcup\mathfrak{H}^{c}$ of complex upper and lower half planes via $g\circ h_{\mathbf{0}}\mapsto g\circ \mathbf{i}$. Here the left actions of $G(\mathbb{R})$ on $\mathfrak{X}$ and $\mathfrak{H}\sqcup\mathfrak{H}^{c}$ are by conjugation and $z\mapsto\frac{az+b}{cz+d}$, for $g= \bigl(\begin{smallmatrix} a & b \\ c & d \end{smallmatrix} \bigr)$, respectively. The pair $(G,\mathfrak{X})$ satisfies Deligne's axioms for having its Shimura curve $Sh$ (\cite{De71} and \cite{De79} 2.1.1). It was initially constructed by Shimura in \cite{Sh66} but nicely reinterpreted by Deligne in \cite{De71} 4.16--4.22 as a moduli of abelian schemes up to isogenies. Namely, Deligne realized $Sh$ as a quasi-projective smooth $\mathbb{Q}$-scheme representing moduli functor $\mathcal{F}^{\mathbb{Q}}$ from the category of abelian $\mathbb{Q}$-schemes to \textit{SETS}:
\[ \mathcal{F}^{\mathbb{Q}}(S)=\{(E,\eta)_{/S}\}_{/\approx} \] 
where $\eta:(\mathbb{A}^{(\infty)})^2\cong \mathcal{T}(E)\otimes_{\widehat{\mathbb{Z}}}\mathbb{A}^{(\infty)}=:V(E)$ is a $\mathbb{Z}$-linear isomorphism and two pairs $(E,\eta)_{/S}$ and $(E',\eta')_{/S}$ are isomorphic up to an isogeny, which we write $(E,\eta)_{/S}\approx (E',\eta')_{/S}$, if there exists an isogeny $\phi:E_{/S}\to E'_{/S}$ such that $\phi\circ\eta=\eta'$.

The pairs $(E,\eta^{(p)})_{/S}$ for a $\mathcal{W}$-scheme $S$ are classified up to isogenies of degree prime to $p$ by a $p$-integral model $Sh_{/\mathcal{W}}^{(p)}$ of $Sh_{/\mathbb{Q}}$ constructed by Kottwitz in \cite{Ko}. In fact, scheme $Sh^{(p)}$ is smooth over $\mathbb{Z}_{(p)}$ and $Sh^{(p)}\otimes_{\mathbb{Z}_{(p)}}\mathbb{Q}=Sh/G(\mathbb{Z}_p).$

Each adele $g\in G(\mathbb{A}^{(\infty)})$ acts on a level structure $\eta$ by $\eta\mapsto\eta\circ g$ inducing $G(\mathbb{A}^{(\infty)})$-action on $Sh$.
A sheaf theoretic coset $\bar{\eta}=\eta K$ for an open compact subgroup $K\subset G(\mathbb{A}^{(\infty)})$ is called a level $K$-structure. The level $K$-structure is defined over base scheme $S$ if for each geometric point $s\in S$ we have $\sigma\circ\bar{\eta}=\bar{\eta}$ for all $\sigma\in\pi_1(S,s)$. The choice of $s$ does not matter in the sense that if the condition is fulfilled for one geometric point in each connected component of $S$, it is valid for all $s\in S$. The quotient $Sh_K=Sh/K$ represents the following quotient functor
\[\mathcal{F}_K(S)=\{(E,\bar{\eta})_{/S}\}_{/\approx} \qquad \] and $Sh=\varprojlim_KSh_K$ when $K$ runs over open compact subgroups of $G(\mathbb{A}^{(\infty)})$.

Let $N$ be a positive integer prime to $p$. Note that if $x=(E,i_N)\in \mathfrak{M}(N)(S)$ for a $\mathcal{W}$-scheme $S$, one can choose $\eta^{(p)}$ so that $\eta^{(p)} \; \mathrm{mod} \; \widehat{\Gamma}_1(N)=i_N$ and get a point $x=(E,\eta^{(p)})\in Sh^{(p)}(S)$ projecting down to $x=(E,i_N)\in \mathfrak{M}(N)(S)$. In this paper, $\widehat{\Gamma}_0(N)=\{\bigl(\begin{smallmatrix} a & b \\ c & d \end{smallmatrix} \bigr)\in G(\widehat{\mathbb{Z}})|\, c\in N\widehat{\mathbb{Z}} \}$ and $\widehat{\Gamma}_1(N)=\{\bigl(\begin{smallmatrix} a & b \\ c & d \end{smallmatrix} \bigr)\in \widehat{\Gamma}_0(N) |\, d-1\in N\widehat{\mathbb{Z}} \}$. By universality, we get a morphism $\mathfrak{M}(N)\to Sh^{(p)}/\widehat{\Gamma}_1(N)$ and the image of $\mathfrak{M}(N)$ gives a geometrically irreducible component of $Sh^{(p)}/\widehat{\Gamma}_1(N)$. The $G(\mathbb{A}^{(\infty)})$-action on $Sh^{(p)}$ given by $\eta^{(p)}\mapsto \eta^{(p)}\circ g^{(p)}$ is geometric preserving the base scheme $\mathop{Spec}(\mathcal{W})$.

The complex points of the Shimura curve $Sh$ are given by the projective limit under the inclusion relation of open compact subgroups of $G(\mathbb{A}^{(\infty)})$ and they have the following expression:
\[ Sh(\mathbb{C})= G(\mathbb{Q})\left\backslash \left( \mathfrak{X}\times G(\mathbb{A}^{(\infty)})\right) \right/ Z(\mathbb{Q})\]
where $Z$ is the center of $G$ and the action is given by $\gamma(z,g)u=(\gamma(z),\gamma gu)$ for $\gamma \in G(\mathbb{Q})$ and $u\in Z(\mathbb{Q})$ (\cite{De79} Proposition 2.1.10 and \cite{Mi} page 324 and Lemma 10.1). 

An important point in Deligne's treatment of $Sh_{/\mathbb{Q}}$ is that instead of functor $\mathcal{F}^{\mathbb{Q}}$ one can consider isomorphic functor $\widetilde{\mathcal{F}}^{\mathbb{Q}}$ from the category of abelian $\mathbb{Q}$-schemes to \textit{SETS}:
\[ \widetilde{\mathcal{F}}^{\mathbb{Q}}(S)= \{(E',\eta)_{/S}|\, \exists E\in \mathcal{F}^{\mathbb{Q}}(S)\, :\, E'_{/S}\approx E_{/S}\, ,\, \eta(\widehat{\mathbb{Z}}^2)=\mathcal{T}(E')\}_{/\cong} \] and $\cong$ is not just induced from an isogeny, but rather an isomorphism of abelian schemes. Imposing the extra condition $\eta(\widehat{\mathbb{Z}}^2)=\mathcal{T}(E')$ is compensated by tightening equivalence from ``isogenies'' to ``isomorphisms''.

The isomorphism of functors $\mathcal{F}^{\mathbb{Q}}\cong\widetilde{\mathcal{F}}^{\mathbb{Q}}$ is realized by finding a unique pair $(E',\eta)$ (up to isomorphism) with $\eta(\widehat{\mathbb{Z}}^2)=\mathcal{T}(E')$ in $S$-isogeny class of a given $(E,\eta)$ in $\mathcal{F}^{\mathbb{Q}}(S)$. Due to its relevance for our construction of CM points on $Sh$, let us briefly recall this standard procedure in construction of Shimura varieties (\cite{PAF} 4.2.1). It suffices to show $\mathcal{F}_K^{\mathbb{Q}}\cong\widetilde{\mathcal{F}}_K^{\mathbb{Q}}$ for every open compact subgroup $K\subset G(\mathbb{A}^{(\infty)})$ such that $K\widehat{\mathbb{Z}}^2=\widehat{\mathbb{Z}}^2$. Let $(E,\eta)\in\mathcal{F}_K^{\mathbb{Q}}(S)$. There are three cases to distinguish:

\underline{Case 1}: We assume that $\eta(\widehat{\mathbb{Z}}^2)$ contains $\mathcal{T}(E)$. Then $C=\eta(\widehat{\mathbb{Z}}^2)/\mathcal{T}(E)$ can be naturally identified with a subgroup of $E$. Let $s\in S$ be an arbitrary geometric point and $\sigma \in \pi_1(S,s)$. Then $\sigma\circ\eta=\eta\circ\tilde{\sigma}$ for some $\tilde{\sigma}\in K$ implying that for $x=\eta(v)\, \mathrm{mod}\, \mathcal{T}(E)$ with $v\in\widehat{\mathbb{Z}}$ we have $\sigma(x)=\sigma(\eta(v))=\eta(\tilde{\sigma}v)$ and $\tilde{\sigma}v\in\widehat{\mathbb{Z}}^2$ since $\widehat{\mathbb{Z}}^2$ is stable under $K$. It follows that subgroup $C$ is stable under $\pi_1(S,s)$ for all geometric points $s\in S$ proving that $C$ is a subgroup scheme of $E$ defined over $S$. Thus the quotient $E'=E/C$ is abelian scheme over $S$ (\cite{ABV} Section 12) and we have $\pi_1(S,s)$-equivariant exact sequence \[0\longrightarrow\mathcal{T}(E)\longrightarrow\mathcal{T}(E')\longrightarrow C\longrightarrow0\] 
yielding $\eta:(\mathbb{A}^{(\infty)})^2\cong V(E)=V(E')$ with the exact identity $\eta(\widehat{\mathbb{Z}}^2)=\mathcal{T}(E')$. The pairs $(E,\eta)_{/S}$ and $(E',\eta)_{/S}$ define the same element of $\mathcal{F}_K^{\mathbb{Q}}(S)$.

\underline{Case 2}: We assume $\eta(\widehat{\mathbb{Z}}^2)$ is contained in $\mathcal{T}(E)$. Then canonical identification $\pi_1(E,\textbf{0})=\mathcal{T}(E)$ (\cite{ABV} Section 18) yields a unique abelian scheme $E'_{/S}$ with $\mathbb{Z}$-linear \'etale isogeny $\pi:E'\to E$ with $\eta(\widehat{\mathbb{Z}}^2)=\mathcal{T}(E')$ making $(E',\eta)_{/S}$ the desired pair.

\underline{Case 3}: The case when $\eta(\widehat{\mathbb{Z}}^2)$ and $\mathcal{T}(E)$ are not related by an inclusion is easily handled by successive application of previous two cases. Namely, we can consider $\mathbb{Z}$-lattices $L'$ with $\widehat{L}'=\eta^{-1}(\mathcal{T}(E))+\widehat{\mathbb{Z}}^2$ and $L''$ with $\widehat{L}''=\eta^{-1}(\mathcal{T}(E))\cap\widehat{\mathbb{Z}}^2$. By case 1, starting from $(E,\eta)_{/S}$ we first create $(E'',\eta'')_{/S}$ in the $S$-isogeny class of $(E,\eta)_{/S}$ with $\mathcal{T}(E'')=\eta(\widehat{L}')$. Then applying case 2, from $(E'',\eta'')_{/S}$ we create $(E',\eta)_{/S}$ with $\mathcal{T}(E')=\eta(\widehat{\mathbb{Z}}^2)$ as desired.

\section{CM points on Shimura curves and differential operators}
In this section we associate to each proper $R_{cp^n}$-ideal $\mathfrak{a}$ prime to $p$ a CM point on appropriate moduli scheme and recall definition of differential operators of Maass--Shimura and Katz and their rationality properties at CM points of modular forms. Note that $n$ and fixed $c$ are arbitrary positive integers and $c$ is prime to $p$.

\subsection{Construction of CM points}
\label{subsec:cm}
 The CM points associated to proper $R_{cp^n}$-ideals $\mathfrak{a}$ prime to $p$ are built in three stages. First, we equip \Xw{R}{\mathcal{W}} with appropriate level structures and a fixed choice of invariant differential. Secondly, for every proper $R_c$-ideal $\mathfrak{A}$ such that $\mathfrak{A}_p=R\otimes_\mathbb{Z}\mathbb{Z}_p$, we then induce from \Xw{R}{\mathcal{W}} corresponding level structures and an invariant differential on \Xw{\mathfrak{A}}{\mathcal{W}}. Finally, if $C\subset \X{\mathfrak{A}}[p^n]$ is a suitable rank $p^n$ subgroup scheme of a finite flat group scheme $\X{\mathfrak{A}}[p^n]$ that is \'etale locally isomorphic to $\mathbb{Z}/p^n\mathbb{Z}$ after faithfully flat extension of scalars, we examine geometric quotient of \X{\mathfrak{A}} by $C$ (\cite{ABV} Section 12) and explain how it gives rise to desired CM point associated to a proper $R_{cp^n}$-ideal $\mathfrak{a}$ prime to $p$.

Any choice of $\widehat{\mathbb{Z}}$-basis $(w_1,w_2)$ of $\widehat{R}=R\otimes_{\mathbb{Z}}\widehat{\mathbb{Z}}$ gives rise to a level structure $\eta^{(p)}(R):(\mathbb{A}^{(p\infty)})^2\cong V^{(p)}(\X{R})$ defined over $\mathcal{W}$ as explained at the beginning of previous section. In particular, fixing a choice of $z_1\in R$ such that $R=\mathbb{Z}+\mathbb{Z}z_1$ we get a $\widehat{\mathbb{Z}}$-basis $(w_1,w_2)$ of $\widehat{R}$.
To define appropriate level structure at $p$, recall that a special fiber $\Xw{R}{\bar{\mathbb{F}}_p}= \Xw{R}{\mathcal{W}}\otimes\bar{\mathbb{F}}_p$ of elliptic curve \Xw{R}{\mathcal{W}} is ordinary under (ord). We have $\X{R}[p^\infty]\cong \mathbb{Q}_p/\mathbb{Z}_p \oplus \mu_{p^\infty}$ over algebraically closed field $\bar{\mathbb{F}}_p$ and in particular, the special fiber \Xw{R}{\bar{\mathbb{F}}_p} has a $p^\infty$-level structure $i:\mu_{p^\infty}\hookrightarrow\X{R}[p^\infty]$ over $\bar{\mathbb{F}}_p$. This level structure arises from the fact that, since \X{R} has ordinary reduction over $\mathcal{W}$, the connected component $\X{R}[p^m]^\circ$ of finite flat group scheme $\X{R}[p^m]$ is isomorphic to $\mu_{p^m}$ and sharing with \Xw{R}{\bar{\mathbb{F}}_p} tangent space $\mathop{Lie}(\X{R})$ at the origin. Thus $\X{R}[p^\infty]^\circ=\bigcup_m\X{R}[p^m]^{\circ}=\underrightarrow{\lim}_m\mu_{p^m}=\mu_{p^\infty}$. Using Serre--Tate deformation theory (\cite{Ka78}), \X{R} sits at the origin of Serre--Tate deformation space and we can lift $i$ to a $p^\infty$-level structure $\ord{R}:\mu_{p^\infty}\hookrightarrow\X{R}[p^\infty]$ defined over $\mathcal{W}$. By Cartier duality, we get \'etale part of level structure at $p$, namely $\et{R}:\mathbb{Q}_p/\mathbb{Z}_p\cong\X{R}[p^\infty]^{\acute{e}t}$ over $\mathcal{W}$, as Cartier dual inverse of the ordinary part $\ord{R}:\mu_{p^\infty}\cong \X{R}[p^\infty]^\circ$. In summary, the $p$-part of level structure can essentially be written as the auto-dual short exact sequence
\[0\longrightarrow \mu_{p^\infty}\longrightarrow \X{R}[p^\infty]\longrightarrow \mathbb{Q}_p/\mathbb{Z}_p\longrightarrow 0\quad \text{over}\; \mathcal{W}\;.\]
In the language of Serre--Tate deformation theory, this is nothing but connected component--\'etale quotient exact sequence of \textit{fppf}-sheaves
\[0\longrightarrow \X{R}[p^\infty]^\circ \longrightarrow \X{R}[p^\infty] \longrightarrow\X{R}[p^\infty]^{\acute{e}t} \longrightarrow 0\quad \text{over}\; \mathcal{W}\;,\]
whose extension class uniquely determines the structure of Barsotti--Tate group $\X{R}[p^\infty]$ (\cite{Ka78}). 
Note that the above exact sequences describing $p$-part of level structure on \Xw{R}{\mathcal{W}} split over $\mathcal{W}$. Indeed, $\X{R}[\mathfrak{p}^\infty]$ is multiplicative (\'etale locally) and $\X{R}[\bar{\mathfrak{p}}^\infty]$ is \'etale over $\mathcal{W}$ and we can identify $\X{R}[\mathfrak{p}^\infty]\cong\X{R}[p^\infty]^\circ$ and $\X{R}[\bar{\mathfrak{p}}^\infty]\cong\X{R}[p^\infty]^{\acute{e}t}$. Since $R\otimes_{\mathbb{Z}}\mathbb{Z}_p$ splits into $R\otimes_{\mathbb{Z}}\mathbb{Z}_p=R_{\bar{\mathfrak{p}}}\oplus  R_{\mathfrak{p}}$ and  $R\otimes_{\mathbb{Z}}\mathbb{Z}_p$ is a subring of $\mathrm{End}(\X{R}[p^\infty])$ we have splitting $\X{R}[p^\infty]=\X{R}[\bar{\mathfrak{p}}^\infty]\oplus\X{R}[\mathfrak{p}^\infty]$
over any ring over which \X{R} and its endomorphism algebra are defined, hence over $\mathcal{W}$. The triple $x(R)=\xw{R}{\mathcal{W}}$ is called a CM point on $Sh$. (The reason for identifying $R\otimes_{\mathbb{Z}}\mathbb{Z}_p$ with $R_{\bar{\mathfrak{p}}}\oplus  R_{\mathfrak{p}}$ in that order is explained in Remark \ref{switch}.) We choose and fix once and for all an invariant differential $\omega(R)$ on \Xw{R}{\mathcal{W}} so that $H^0(\X{R},\Omega_{\Xw{R}{\mathcal{W}}})=\mathcal{W}\omega(R)$. 

Let $\mathfrak{A}$ be a proper $R_c$-ideal whose $p$-adic completion $\mathfrak{A}_p=\mathfrak{A}\otimes_\mathbb{Z}\mathbb{Z}_p$ is identical to $R\otimes_\mathbb{Z}\mathbb{Z}_p$. Regarding $c$ as an element of $\mathbb{A}^\times$, $(cw_1,w_2)$ is a basis of $\widehat{R}_c$ over $\widehat{\mathbb{Z}}$ yielding a level structure $\eta^{(p)}(R_c):(\mathbb{A}^{(p\infty)})^2\cong V^{(p)}(\X{R_c})$. Choosing a complete set of representatives $\{a_1,\ldots,a_{H^-}\}\subset M^{\times}_{\mathbb{A}}$ so that $M^{\times}_{\mathbb{A}}=\bigsqcup_{j=1}^{H^-}M^\times a_j\widehat{R}_c^\times M_\infty^\times$ we have $\widehat{\mathfrak{A}}=\alpha a_j\widehat{R}_c$ for some $\alpha\in M^\times$ and $1\leq j\leq H^-$ and we can define $\eta^{(p)}(\mathfrak{A})=\alpha^{-1} a_j^{-1} \eta^{(p)}(R_c)$ so that we have commutative diagram
\[\begin{CD}
\mathcal{T}^{(p)}(\X{\mathfrak{A}}) @<\alpha a_j<< \mathcal{T}^{(p)}(\X{R_c})\\
@A{\cong}A\eta^{(p)}(\mathfrak{A})A @A{\cong}A\eta^{(p)}(R_c)A\\
\widehat{\mathfrak{A}}^{(p)} @<\alpha a_j<< \widehat{R}_c^{(p)}
\end{CD}\qquad .\]
Since $\mathfrak{A}_p=R\otimes_\mathbb{Z}\mathbb{Z}_p$, \X{R\cap\mathfrak{A}} is an \'etale covering of both \X{\mathfrak{A}} and \X{R} and we get $\ord{\mathfrak{A}}:\mu_{p^\infty}\cong \X{\mathfrak{A}}[p^\infty]^\circ$ and $\et{\mathfrak{A}}:\mathbb{Q}_p/\mathbb{Z}_p\cong\X{\mathfrak{A}}[p^\infty]^{\acute{e}t}$ first by pulling back \ord{R} and \et{R} from  \X{R} to \X{R\cap\mathfrak{A}} and then by push-forward from \X{R\cap\mathfrak{A}} to \X{\mathfrak{A}}. In this way we get a CM point \[x(\mathfrak{A})=\xw{\mathfrak{A}}{\mathcal{W}}\] 
on $Sh$ associated to a proper $R_c$-ideal $\mathfrak{A}$. Similarly, $\omega(R)$ induces a differential $\omega(\mathfrak{A})$ on \X{\mathfrak{A}} first by pulling back $\omega(R)$ from  \X{R} to \X{R\cap\mathfrak{A}} and then by pull-back inverse from \X{R\cap\mathfrak{A}} to \X{\mathfrak{A}}. The projection $\pi_1:\X{R\cap\mathfrak{A}}\twoheadrightarrow\X{\mathfrak{A}}$ is \'etale so the pull-back inverse $(\pi_1^*)^{-1}:\Omega_{\X{R\cap\mathfrak{A}}/\mathcal{W}}\to\Omega_{\X{\mathfrak{A}}/\mathcal{W}}$ is an isomorphism, whence $H^0(\X{\mathfrak{A}},\Omega_{\X{\mathfrak{A}}/\mathcal{W}})=\mathcal{W}\omega(\mathfrak{A})$.

Let $C\subset \X{R_c}[p^n]$ be a rank $p^n$ subgroup scheme of a finite flat group scheme $\X{R_c}[p^n]$ that is \'etale locally isomorphic to $\mathbb{Z}/p^n\mathbb{Z}$ after faithfully flat extension of scalars and such that $C\cap\X{R_c}[\mathfrak{p}^n]=\{0\}$ but also $C$ is different from $\X{R_c}[\bar{\mathfrak{p}}^n]$. Note that $\X{R_c}[p^n]=\X{R_c}[\bar{\mathfrak{p}}^n]\oplus\X{R_c}[\mathfrak{p}^n]=\mathbb{Z}/p^n\mathbb{Z}\oplus \mu_{p^n}$ over $\mathcal{W}$. If $\zeta_{p^n}$ and $\gamma_{p^n}$ are the canonical generators of $\mu_{p^n}$ and $\mathbb{Z}/p^n\mathbb{Z}$, respectively, then we actually consider $C$ to be one of $p^{n-1}(p-1)$ rank $p^n$ finite flat subgroup schemes $C_u=\langle\zeta_{p^n}^{-u}\gamma_{p^n}\rangle$ of $\X{R_c}[p^n]$, for $1\leq u\leq p^n$ and $\mathrm{gcd}(u,p)=1$. We shall examine geometric quotient of \X{R_c} by such finite flat subgroup schemes $C$ (\cite{ABV} Section 12). Since extension $W/\mathbb{Z}_p$ is unramified, $C$ as a finite flat subgruop scheme is well defined over $\mathcal{W}[\mu_{p^n}]$, which stands for finite extension of $\mathcal{W}$ obtained by adjoining a primitive $p^n$-th root of unity $\zeta_{p^n}$ inside $\bar{\mathbb{Q}}$. Thus, the geometric quotient $\X{R_c}/C$ is defined over $\mathcal{W}[\mu_{p^n}]$. If $\mathfrak{a}$ is a lattice so that $\X{R_c}/C=\X{\mathfrak{a}}$ then $\mathfrak{a}/R_c=C$ and $\mathfrak{a}$ is a $\mathbb{Z}$-lattice of $M$ because $C$ is $\mathbb{Z}$-submodule. Since $p^nC=0$ we have $p^nR_c\mathfrak{a}\subset\mathfrak{a}$ which means that $\mathfrak{a}$ is $R_{cp^n}$-ideal. Moreover $\mathfrak{a}$ is not $R_{cp^{n-1}}$-submodule so we conclude that $\mathfrak{a}$ is a proper $R_{cp^n}$-ideal. The quotient map $\pi:\X{R_c}\twoheadrightarrow\X{R_c}/C$ is \'etale over $\mathcal{W}[\mu_{p^n}]$ so we obtain level structures $\eta^{(p)}(\mathfrak{a})=\pi_{*}\eta^{(p)}(R_c)=\pi\circ\eta^{(p)}(R_c)$, $\ord{\mathfrak{a}}=\pi_{*}\ord{R_c}=\pi\circ\ord{R_c}$ and $\et{\mathfrak{a}}=\pi_{*}\et{R_c}=\pi\circ\et{R_c}$, as well as an invariant differential $\omega(\mathfrak{a})=(\pi^*)^{-1}\omega(R_c)$ on $\X{R_c}/C$. Note that $H^0(\X{\mathfrak{a}},\Omega_{\X{\mathfrak{a}}/\mathcal{W}[\mu_{p^n}]})=\mathcal{W}[\mu_{p^n}]\omega(\mathfrak{a})$ as $\pi:\X{R_c}\twoheadrightarrow\X{\mathfrak{a}}$ is \'etale and consequently  $(\pi^*)^{-1}:\Omega_{\X{R_c}/\mathcal{W}[\mu_{p^n}]}\to\Omega_{\X{\mathfrak{a}}/\mathcal{W}[\mu_{p^n}]}$ is an isomorphism. In this way we created $p^{n-1}(p-1)$ CM points
\[x(\mathfrak{a})=\xw{\mathfrak{a}}{\mathcal{W}[\mu_{p^n}]}\]
on $Sh$, equipped with an invariant differential $\omega(\mathfrak{a})$ on \X{\mathfrak{a}}, where these $\mathfrak{a}$'s are representatives of exactly $p^{n-1}(p-1)$ proper $R_{cp^n}$-ideal classes in $\mathrm{Cl}_M^-(cp^n)$ that project to proper ideal class of $\bar{\mathfrak{p}}^{-n}R_c$ in $\mathrm{Cl}_M^-(c)$. 

Let $\{\mathfrak{A}_1,\ldots,\mathfrak{A}_{H^-}\}$ be a complete set of representatives  for $\mathrm{Cl}_M^-(c)$. In the same fashion as above, by considering geometric quotients $\X{\mathfrak{A}_j}/C_u$ of \Xw{\mathfrak{A}_j}{\mathcal{W}} by rank $p^n$ \'etale finite flat subgroups schemes $C_u=\langle\zeta_{p^n}^{-u}\gamma_{p^n}\rangle\subset \X{\mathfrak{A}_j}[p^n]$, where $1\leq u\leq p^n$ and $\mathrm{gcd}(u,p)=1$, we create CM points
\[x(\mathfrak{a}_{j,u})=\xw{\mathfrak{a}_{j,u}}{\mathcal{W}[\mu_{p^n}]}\]
on $Sh$, equipped with an invariant differential $\omega(\mathfrak{a}_{j,u})$ on \X{\mathfrak{a}_{j,u}}, where for every fixed $1\leq j\leq H^-$, these $\mathfrak{a}_{j,u}$ are representatives of exactly $p^{n-1}(p-1)$ proper $R_{cp^n}$-ideal classes in $\mathrm{Cl}_M^-(cp^n)$ that project to proper ideal class of $\bar{\mathfrak{p}}^{-n}\mathfrak{A}_j$ in $\mathrm{Cl}_M^-(c)$. 

Finally, if $\mathfrak{b}$ is any proper $R_{cp^n}$-ideal prime to $p$ then after finding its proper ideal class representative $\mathfrak{a}_{j,u}$ in  $\mathrm{Cl}_M^-(cp^n)$, for a unique $j$ and $u$ as above, there is $\beta\in M^\times$ such that $\mathfrak{b}=\beta \mathfrak{a}_{j,u}$ and $\beta$ induces an isomorphism $\beta:\X{\mathfrak{a}_{j,u}}\to\X{\mathfrak{b}}$. Thus, we can define $\eta^{(p)}(\mathfrak{b})=\beta_{*}(\eta^{(p)}(\mathfrak{a}_{j,u}))=\beta^{-1}\eta^{(p)}(\mathfrak{a}_{j,u})$, $\ord{\mathfrak{b}}=\beta_{*}(\ord{\mathfrak{a}_{j,u}})=\beta^{-1}\ord{\mathfrak{a}_{j,u}}$ and $\et{\mathfrak{b}}=\beta_{*}(\et{\mathfrak{a}_{j,u}})=\beta^{-1}\et{\mathfrak{a}_{j,u}}$, as well as an invariant differential $\omega(\mathfrak{b})=(\beta^*)^{-1}\omega(\mathfrak{a}_{j,u})$ on $\X{\mathfrak{b}}$.

\subsection{Differential operators}\label{diffop}
Recall the definition of Maass--Shimura differential operators on $\mathfrak{H}$ indexed by $k\in\mathbb{Z}$:
\[ \delta_k = \frac{1}{2\pi\mathbf{i}}\left(\frac{\partial}{\partial z}+\frac{k}{z-\bar{z}}\right)\; \text{and}\; \delta_k^r=\delta_{k+2r-2}\ldots\delta_k \]
for a non-negative integer $r$. Their important property is that once applied to modular forms they preserve rationality of a value at a CM point (\cite{AAF} III and \cite{Sh75}). Let $\mathfrak{a}$ be a proper $R_{cp^n}$-ideal prime to $p$. We constructed corresponding CM point
\[x(\mathfrak{a})=\xw{\mathfrak{a}}{\mathcal{W}[\mu_{p^n}]}\]
on $Sh$ and an invariant differential $\omega(\mathfrak{a})$ on \Xw{\mathfrak{a}}{\mathcal{W}[\mu_{p^n}]}.
To precisely state this rationality result of Shimura, note that the complex uniformization $\X{\mathfrak{a}}(\mathbb{C})=\mathbb{C}/\mathfrak{a}$ induces a canonical invariant differential $\omega_\infty(\mathfrak{a})$ in $\Omega_{\X{\mathfrak{a}}/\mathbb{C}}$ by pulling back $\mathrm{d}u$, where $u$ is the standard variable on $\mathbb{C}$. Then one can define a period $\Omega_\infty\in\mathbb{C}^\times$ by $\omega(\mathfrak{a})=\Omega_\infty\omega_\infty(\mathfrak{a})$ (\cite{Ka} Lemma 5.1.45). Note that $\Omega_\infty$ does not depend on $\mathfrak{a}$ since $\omega(\mathfrak{a})$ is induced by $\omega(R)$ on \X{R} by construction. Then for $f\in G_k(N,\psi;\mathcal{W})$ we have
\[ \qquad \frac{\delta_k^rf(x(\mathfrak{a}),\omega_\infty(\mathfrak{a}))}{\Omega_\infty^{k+2r}}=\delta_k^rf(x(\mathfrak{a}),\omega(\mathfrak{a}))\in\mathcal{W}[\mu_{p^n}]\]
(\cite{Ka} Theorem 2.4.5).

Katz gave a purely algebro-geometric definition of Maass--Shimura differential operator (\cite{Ka} Chapter II) by interpreting it in terms of Gauss--Manin connection of the universal abelian variety with real multiplication over $\mathfrak{M}$. This allowed him to extend the operator $\delta_* $ to algebro-geometric and $p$-adic modular forms; we denote the latter extension of $\delta^r_*$ by $d^r:V(N;W)\to V(N;W)$. As explained at the end of Section \ref{p-adicmf}, the ordinary part of level structure at $p$, $\ord{\mathfrak{a}}:\mu_{p^\infty}\cong \X{\mathfrak{a}}[\mathfrak{p}^\infty]$ induces trivialization $\widehat{\mathbb{G}}_m\cong \widehat{\X{\mathfrak{a}}}$ for the $p$-adic formal completion $\widehat{\X{\mathfrak{a}}}_{/W[\mu_{p^n}]}$ of \X{\mathfrak{a}} along its zero-section. We obtain an invariant differential $\omega_p(\mathfrak{a})$ on $\widehat{\X{\mathfrak{a}}}_{/W[\mu_{p^n}]}$ by pushing forward $\frac{\mathrm{d}t}{t}$ on $\widehat{\mathbb{G}}_m$, which then extends to an invariant differential on \Xw{\mathfrak{a}}{W[\mu_{p^n}]} also denoted by $\omega_p(\mathfrak{a})$. Then one can define a period $\Omega_p\in W^\times$, independent of $\mathfrak{a}$, by $\omega(\mathfrak{a})=\Omega_p \omega_p(\mathfrak{a})$ (\cite{Ka} Lemma 5.1.47). The fact that will be of instrumental use for us is
\begin{equation}\label{K-S} \frac{(d^rf)(x(\mathfrak{a}),\omega_p(\mathfrak{a}))}{\Omega_p^{k+2r}}=(d^rf)(x(\mathfrak{a}),\omega(\mathfrak{a}))=(\delta_k^rf)(x(\mathfrak{a}),\omega(\mathfrak{a}))\in\mathcal{W}[\mu_{p^n}]
\end{equation}
(\cite{Ka} Theorem 2.6.7).  

\section{Hecke relation among CM points on Shimura curves}
\label{sec:hecke}
Let $S_k(\Gamma_0(N),\psi)$ denote the space of holomorphic cusp forms $f$ of level $\Gamma_0(N)$ and nebentypus $\psi$ with $f(\gamma(z))=\psi(\gamma)f(z)j(\gamma,z)$ for $\gamma\in \Gamma_0(N)$, where $j(\bigl(\begin{smallmatrix} a & b \\ c & d \end{smallmatrix} \bigr),z)=cz+d$ for $z\in\mathfrak{H}$ and $\bigl(\begin{smallmatrix} a & b \\ c & d \end{smallmatrix} \bigr)\in G(\mathbb{R})$. We may regard Dirichlet character $\psi$ as a character of $\widehat{\Gamma}_0(N)$ via $\bigl(\begin{smallmatrix} a & b \\ c & d \end{smallmatrix} \bigr)\mapsto \psi(d)$. Then by strong approximation theorem $G(\mathbb{A})=G(\mathbb{Q})\widehat{\Gamma}_0(N)\mathrm{GL}_2^+(\mathbb{R})$ ($\mathrm{GL}_2^+(\mathbb{R})=\{ g \in G(\mathbb{R})|\mathrm{det}(g)>0\}$) and we can lift $f$ to $\mathbf{f}:G(\mathbb{Q})\backslash G(\mathbb{A})\to \mathbb{C}$ by $\mathbf{f}(\alpha u g_\infty)=f(g_\infty(\mathbf{i}))\psi(u)j(g_\infty,\mathbf{i})^{-k}$ for $\alpha\in G(\mathbb{Q})$, $u\in \widehat{\Gamma}_0(N)$ and $g_\infty \in \mathrm{GL}_2^+(\mathbb{R})$. Note that $\mathbf{f}(\alpha g u)=\psi(u)\mathbf{f}(g)$ for $\alpha\in G(\mathbb{Q})$ and $u\in \widehat{\Gamma}_0(N)$, and that $\mathbf{f}$ is so-called arithmetic lift of $f$ (as opposed to automorphic lift involving the determinant factor which we omitted). If we denote by $\mathcal{S}_k(\widehat{\Gamma}_0(N),\psi)$ the space of adelic cusp forms $\mathbf{f}$ obtained from $f\in S_k(\Gamma_0(N),\psi)$ in described way then $S_k(\Gamma_0(N),\psi)\cong\mathcal{S}_k(\widehat{\Gamma}_0(N),\psi)$ via $f\leftrightarrow\mathbf{f}$. Note that the center $Z(\mathbb{A})$ acts on $\mathcal{S}_k(\Gamma_0(N),\psi)$ via $\mathbf{f}|\zeta(g)=\mathbf{f}(\zeta g)$ and that $\mathbf{f}|\zeta_\infty=\zeta_\infty^{-k}\mathbf{f}$ for $\zeta\in Z(\mathbb{A})$. Thus $\mathcal{S}_k(\Gamma_0(N),\psi)$ decomposes into the direct sum of eigenspaces for this action and on each eigenspace $Z(\mathbb{A})$ acts by a Hecke character whose restriction to $\widehat{\Gamma}_0(N)\cap Z(\mathbb{A})$ is $\psi$ and which sends $\zeta_\infty$ to $\zeta_\infty^{-k}$. If we lift $\psi$ to $\mathbb{A}^\times$ in standard way and set $\boldsymbol{\psi}:=\psi |\cdot|_{\mathbb{A}}^{-k}$, let $\mathcal{S}_k(N,\boldsymbol{\psi})$ denote the $\boldsymbol{\psi}$-eigenspace. Then $S_k(\Gamma_0(N),\psi)\cong \mathcal{S}_k(N,\boldsymbol{\psi})$ via $f\leftrightarrow\mathbf{f}$.

Following closely Section 3.1 of \cite{HidaNV} we adelize Maass--Shimura $m$-th derivative $\delta_k^mf$, $m\geq 0$, to a function $\mathbf{f}_m$ on 
$G(\mathbb{A})$ in the following way. Regarding $X=\frac{1}{2}\bigl(\begin{smallmatrix} 1 & \mathbf{i} \\ \mathbf{i} & -1 \end{smallmatrix} \bigr)\in\mathfrak{sl}_2(\mathbb{C})$ -- a Lie algebra of $\mathrm{SL}_2(\mathbb{C})$, as an invariant differential operator $X_{g_\infty}$ on $\mathrm{SL}_2(\mathbb{C})$ for the variable matrix $g_\infty\in G(\mathbb{R})$ (here identifying $G(\mathbb{R})$ with $\mathrm{SL}_2(\mathbb{R})\times\mathbb{R}^\times$ by the natural isogeny), we set
\[ \mathbf{f}_m(g)=(-4\pi)^{-m}|\mathrm{det}(g)|_\mathbb{A}^{-m}X^m_{g_\infty}\mathbf{f}(g) \]
where $g_\infty$ is infinite part of $g\in G(\mathbb{A})$. Then $\mathbf{f}_m(g_\infty)=(\delta_k^mf)(g_\infty(\mathbf{i}))j(g_\infty,\mathbf{i})^{-k-2m}$, and 
when $\mathrm{det}(g_{\infty})=1$ we have
\begin{equation}\label{der}
\mathbf{f}_m(g)=|\mathrm{det}(g^{(\infty)})|_\mathbb{A}^{-m} \delta_k^m\mathbf{f}(g)
\end{equation}
where $\delta_k^m\mathbf{f}:G(\mathbb{Q})\backslash G(\mathbb{A})\to \mathbb{C}$ is the arithmetic lift of $\delta_k^mf$ as above, given by $\delta_k^m\mathbf{f}(\alpha u g_\infty)=\delta_k^mf(g_\infty(\mathbf{i}))\psi(u)j(g_\infty,\mathbf{i})^{-k-2m}$ for $\alpha\in G(\mathbb{Q})$, $u\in \widehat{\Gamma}_0(N)$ and $g_\infty \in \mathrm{GL}_2^+(\mathbb{R})$ (\cite{HidaNV} Definition 3.3 and Lemma 3.1). Here $g^{(\infty)}$ is finite part of $g\in G(\mathbb{A})$. The central character of $\mathbf{f}_m$ is given by $\boldsymbol{\psi}_m(x)=\boldsymbol{\psi}(x)|x|_{\mathbb{A}}^{-2m}$ and $\mathbf{f}_m(gu)=\boldsymbol{\psi}_m(u)\mathbf{f}_m(g)$ when $u\in \widehat{\Gamma}_0(N)$.

Let $f_0\in S_k(\Gamma_0(N),\psi)$ be a normalized Hecke newform of conductor $N_0$, nebentypus $\psi$ and let  $\mathbf{f}_0\in\mathcal{S}_k(N,\boldsymbol{\psi})$ be the corresponding adelic form with central character $\boldsymbol{\psi}$. Let $f$ be a suitable normalized Hecke eigen-cusp form that will be explicitly made out of $f_0$ in Section \ref{sec:mainthm} such that its arithmetic lift $\mathbf{f}$ is in the automorphic representation $\pi_{\mathbf{f}_0}$ generated by the unitarization $\mathbf{f}_0^u$.

Fix a choice of $z_1\in R$ such that $R=\mathbb{Z}+\mathbb{Z}z_1$ and define $\rho:M\hookrightarrow \mathrm{M}_2(\mathbb{Q})$ by a regular representation
\[\rho(\alpha)\begin{pmatrix} z_1 \\ 1 \end{pmatrix}=\begin{pmatrix} \alpha z_1 \\ \alpha \end{pmatrix}\; . \]
After tensoring with $\mathbb{A}$ we get $\rho:{M^{\times} \backslash M^{\times}_{\mathbb{A}}}  \hookrightarrow G(\mathbb{Q})\backslash G(\mathbb{A})$. We fix $g_1\in G(\mathbb{A})$ such that $g_{1,\infty}(\mathbf{i})=z_1$ and $\mathrm{det}(g_{1,\infty})=1$ while the finite places of $g_1$ will be specified shortly.
\begin{factoring}\label{factoring}
Let $\chi_m:{M^{\times} \backslash M^{\times}_{\mathbb{A}}}\to\mathbb{C}^\times$ be a Hecke character with $\chi_m|_{\mathbb{A}^\times} = \boldsymbol{\psi}_m^{-1}$ and $\chi_m(a_\infty)=a_\infty^{k+2m}$. Then $a\mapsto  \mathbf{f}_m(\rho(a)g_1) \chi_m(a)$ factors through $\mathrm{Cl}_M^-= M^\times \left\backslash M^{\times}_{\mathbb{A}}\right/(\mathbb{A}^{(\infty)})^\times M^\times_\infty$ (the anticyclotomic idele class group).
\end{factoring} 
\begin{proof}
This is Lemma 3.7 of \cite{HidaNV}.
\end{proof}

We denote by $\mathrm{Cl}_M=M^\times \left\backslash M^{\times}_{\mathbb{A}}\right/ M^\times_\infty$ the idele class group and set
\[L_{\chi_m}(\mathbf{f}_m):=\int_{\mathrm{Cl}_M}  \mathbf{f}_m(\rho(a)g_1) \chi_m(a)d^{\times}a  \]
for a choice of Haar measure $d^{\times}a$ on $M^{\times}_{\mathbb{A}}/ M^\times_\infty$ normalized so that $\int_{\widehat{R}^\times} d^{\times}a=1$. This is the normalization used by Hida in Section 2.1 of \cite{HidaNV}. We also take a fundamental domain $\Phi\subset M^{\times}_{\mathbb{A}}/ M^\times_\infty$ of $\mathrm{Cl}_M$ to get a measure on $\mathrm{Cl}_M$ still denoted by $d^{\times}a$. Thus $\int_{\widehat{R}^\times/R^\times} d^{\times}a=1/|R^\times|$ and choosing a complete set of representatives $\{b_1,\ldots,b_{h(M)}\}\subset M^{\times}_{\mathbb{A}}$ so that $M^{\times}_{\mathbb{A}}=\bigsqcup_{j=1}^{h(M)}M^\times b_j(\mathbb{A}^{(\infty)})^\times M_\infty^\times$ we conclude $\mathop{vol}(\mathrm{Cl}_M)=\int_{\mathrm{Cl}_M} d^{\times}a=h(M)/|R^\times|$.

In the main Theorem 4.1 of \cite{HidaNV} Hida computed $L_{\chi_m}(\mathbf{f}_m)^2$ by the Rankin--Selberg convolution method starting from principle used by Waldspurger in \cite{Wa} that the orthogonal similitude group $\mathrm{GO}_D$ of the norm for a quaternion algebra $D$, in this case $D:=\mathrm{M}_2(\mathbb{Q})$, is nearly the same as $D^\times \times D^\times$, and Shimizu's theta lift for this orthogonal group realizes the Jacquet-Langlands correspondence. However the heart of the matter in his computation is a delicate choice of a Schwartz-Bruhat function on $D_{\mathbb{A}}$ attaining this optimality of theta correspondence (Sections 1.4 and 1.7 of \cite{HidaNV}). This careful choice is motivated by the explicit computation of the $q$-expansion of the theta lift of $\mathbf{f}$ to $\mathrm{GO}_D(\mathbb{A})$ via ``partial Fourier transform'' of the Siegel-Weil theta series that was performed in Hida's proof of anticyclotomic Main Conjecture for CM fields in \cite{HidaCoates}. It is followed by a choice of $g_1$ at finite places, yet another subtle maneuver playing the role in the splitting of the quaternionic theta series into a product of theta series of $M$, which in turn comes from splitting the quadratic space $(D,\mathrm{det})=(M,N_{M/\mathbb{Q}})\oplus(M,-N_{M/\mathbb{Q}})$ for the norm form $N_{M/\mathbb{Q}}$ (\cite{HidaNV} Section 2.2). 

Let $N_0=\prod_ll^{\nu(l)}$ be the prime factorization and let $N_{ns}=\prod_{l \, \text{non-split}}l^{\nu(l)}$ be its ``non-split'' part. 
Let $\mathfrak{C}$ denote the conductor of $\chi_m$ as above. For the rest of the paper we assume that $\chi_m$ is unramified outside $N_0$ and $p$ and that its conductor at non-split primes $l|N_{ns}$ is precisely $l^{\nu(l)}$, which we write $\mathfrak{C}_{N_{ns}}=N_{ns}$. As far as ramification at $p$ is concerned, we distinguish the following two cases:
\begin{enumerate}
\item[1)] $\mathfrak{C}_p=p^s$ for some $s \geq \mathrm{max}(1,\mathrm{ord}_p(N_0))$.
\item[2)] $\mathfrak{C}_p=1$ 
\end{enumerate}
We refer to the first one as $p$-ramified case and to the second one as $p$-unramified case. 
To specify $g_1$ at finite places we introduce some notation following closely Section 4 of \cite{HidaNV}. 
We denote by $d(M)\in \mathbb{Z}$ ($d(M)<0$) the discriminant of $M$ and set $d_0(M)=|d(M)|/4$ if $4|d(M)$ while $d_0(M)=|d(M)|$ otherwise. 
We divide the set of prime factors of $N(\mathfrak{C})d_0(M)N_0$ into disjoint union $A\sqcup C$ as follows. If we are in the $p$-ramified case we set $A=\{p\}$, otherwise we set $A=\emptyset$. Set $C=C_0\sqcup C_1$ where $C_1$ is the set of prime factors of $d_0(M)$ and $C_0=C_i\sqcup C_{sp}\sqcup C_r$ so that $C_i$ consists of primes inert in $M$, $C_r=\{2\}$ if $\mathrm{ord}_2(d(M))=2$ with $\nu(2)>2$ and $C_r=\emptyset$ otherwise. Then $C_{sp}$ consists of primes split in $M$ that are not already placed in $A$. Thus, we have three possibilities for prime $p$: in the $p$-ramified case it is placed in set $A$, whereas in the $p$-unramified case it is placed in set $C_{sp}$ when $p|N_0$ or is completely out of this consideration when $p\nmid N_0$.

We already made a choice of a prime $\bar{\mathfrak{p}}$ over $p$ in $M$; in the $p$-ramified case we set $\mathcal{A}=\{\bar{\mathfrak{p}}\}$, and we choose a prime $\bar{\mathfrak{l}}$ over each $l\in C_{sp}$, denoting the set of all these choices by $\mathcal{C}_{sp}=\{\bar{\mathfrak{l}}|l\in C_{sp}\}$. These choices allow us to identify $M_l=M_{\bar{\mathfrak{l}}}\times M_{\mathfrak{l}}=\mathbb{Q}_l\times\mathbb{Q}_l$ for all $l\in A\sqcup C_{sp}$. (Note that these identifications follow reversed notation from the ones in \cite{HidaNV} due to a reason explained in Remark \ref{switch} at prime $p$ -- we proceed similarly at other split primes to keep our notation uniform.) If $\iota_l$ and $c\circ\iota_l$ are projections of $M_l$ to $M_{\mathfrak{l}}$ and $M_{\bar{\mathfrak{l}}}$, respectively, we can write $\iota_l(\alpha)=\alpha$ and $c\circ\iota_l(\alpha)=\bar{\alpha}$. For  $l\in A\sqcup C_{sp}$ we specify $g_{1,l}$  by first choosing $h_{1,l}\in G(\mathbb{Z}_l)$ so that $h_{1,l}^{-1}\rho(\alpha)h_{1,l}=\bigl(\begin{smallmatrix} \bar{\alpha}  & 0 \\ 0 & \alpha \end{smallmatrix} \bigr)$; for example $h_{1,l}=\bigl(\begin{smallmatrix}  \bar{z}_1 & z_1 \\ 1 & 1 \end{smallmatrix} \bigr)$ will do, and then setting:
\[g_{1,p}=h_{1,p}\bigl(\begin{smallmatrix} p^s & 1 \\ 0 & 1 \end{smallmatrix} \bigr) \text{ if an only if } p\in A \text{ (i.e. the $p$-ramified case only),}\]
\[g_{1,l}=h_{1,l}\bigl(\begin{smallmatrix} l^{\nu(l)} & 0 \\ 0 & 1 \end{smallmatrix} \bigr)\text{ for } l\in C_{sp} \; .\]
Unless $l=2$ is inert in $M$, we set
\[g_{1,l}=\bigl(\begin{smallmatrix} l^{\nu(l)} & 0 \\ 0 & 1 \end{smallmatrix} \bigr)\; \text{for}\; l\in C_i\sqcup C_r \sqcup C_1 \; .\]
If exceptionally, $2\in C$ and 2 is inert in $M$, the appropriate choice of $g_{1,l}$ for $l=2$ is given in Lemma 2.5 of \cite{HidaNV}. We chose $g_{1,\infty}\in G(\mathbb{R})$ so that $g_{1,\infty}(\mathbf{i})=z_1$ and $\mathrm{det}(g_{1,\infty})=1$. 
We set $g_{1,l}$ to be the identity matrix in $G(\mathbb{Z}_l)$ for $l\not\in A\sqcup C \sqcup\{\infty\}$ (see the proof of Proposition 2.2 in \cite{HidaNV}).

From now on until the proof of Main Theorem at the end of paper we actually assume that we are in the $p$-ramified case. This is technically more demanding case and the $p$-unramified case will follow as side product of its consideration. To summarize, we are working under assumption
\begin{equation}\label{dominate}
\mathfrak{C}_{N_{ns}p}= N_{ns}p^s \text{ for some } s\geq \mathrm{max}(1,\mathrm{ord}_p(N_0)) \; .
\end{equation}
Now we are able to make a slight improvement of Lemma \ref{factoring} above:
\begin{improv}\label{improv}
Let $\chi_m:{M^{\times} \backslash M^{\times}_{\mathbb{A}}}\to\mathbb{C}^\times$ be a Hecke character such that its conductor $\mathfrak{C}$ satisfies $\mathfrak{C}_{N_{ns}p}= N_{ns}p^s$ and $\chi_m|_{\mathbb{A}^\times} = \boldsymbol{\psi}_m^{-1}$ and $\chi_m(a_\infty)=a_\infty^{k+2m}$. Then $a\mapsto \mathbf{f}_m(\rho(a)g_1) \chi_m(a)$ factors through $\mathrm{Cl}_M^-(N_{ns}p^s)= M^\times \left\backslash M^{\times}_{\mathbb{A}}\right/(\mathbb{A}^{(\infty)})^\times\widehat{R}_{N_{ns}p^s}^\times M^\times_\infty$ .
\end{improv} 
\begin{proof}
We know from Lemma \ref{factoring} that $a\mapsto  \mathbf{f}_m(\rho(a)g_1) \chi_m(a)$ factors through $\mathrm{Cl}_M^-= M^\times \left\backslash M^{\times}_{\mathbb{A}}\right/(\mathbb{A}^{(\infty)})^\times M^\times_\infty$. Note that $\widehat{R}_{p^s}^\times=\{\alpha\in\widehat{R}^\times|\,\alpha\equiv\bar{\alpha}\,\mathrm{mod}\, p^s\}$ or in other words $\widehat{R}_{p^s}^\times=(\mathbb{A}^{(\infty)})^\times U_{p^s}$ for $U_{p^s}=\{\alpha\in\widehat{R}^\times|\,\alpha\equiv 1 \, \mathrm{mod}\, p^s\}$.
We have 
\[g_{1,p}^{-1}\rho(\alpha)g_{1,p}=\bigl(\begin{smallmatrix} p^s & 1 \\ 0 & 1 \end{smallmatrix} \bigr)^{-1}\bigl(\begin{smallmatrix} \bar{\alpha} & 0 \\ 0 & \alpha  \end{smallmatrix} \bigr)\bigl(\begin{smallmatrix} p^s & 1 \\ 0 & 1 \end{smallmatrix} \bigr)=\bigl(\begin{smallmatrix} \bar{\alpha} & \frac{\bar{\alpha}-\alpha}{p^s} \\ 0 & \alpha \end{smallmatrix} \bigr)\in \widehat{\Gamma}_1(p^s)\;\text{for}\;\alpha\in U_{p^s} \;.\]
It follows that $\mathbf{f}_m(\rho(\alpha)g_1)=\mathbf{f}_m(g_1(g_1^{-1}\rho(\alpha)g_1))=\mathbf{f}_m(g_1)$ while $\chi_m(\alpha)=1$ due to the fact that conductor of $\chi_m$ at $p$ is precisely $p^s$. Thus $\mathbf{f}_m(\rho(\alpha)g_1)\chi_m(\alpha)=\mathbf{f}_m(g_1)$.

Note that when $l|N_{ns}$ is ramified or inert one can suppose that $M_l=\mathbb{Q}_l[\sqrt{d_0}]$ with $R\otimes_{\mathbb{Z}}\mathbb{Z}_l= \mathbb{Z}_l[\sqrt{d_0}]$ being $l$-adic integer ring of $M_l$. Taking $z_1=1/\sqrt{d_0}$ we can realize $\rho(x+y\sqrt{d_0})=\bigl(\begin{smallmatrix} x & y \\ d_0y & x \end{smallmatrix} \bigr)$. Hence for $\alpha= x+y\sqrt{d_0}\in (R_{l^{\nu(l)}}\otimes_{\mathbb{Z}}\mathbb{Z}_l)^\times$ we have $l^{-\nu(l)}y\in \mathbb{Z}_l$ and consequently
\[g_{1,l}^{-1}\rho(x+y\sqrt{d_0})g_{1,l}=\bigl(\begin{smallmatrix} l^{\nu(l)} & 0 \\ 0 & 1 \end{smallmatrix} \bigr)^{-1}\bigl(\begin{smallmatrix} x & y \\ d_0y & x \end{smallmatrix} \bigr)\bigl(\begin{smallmatrix} l^{\nu(l)} & 0 \\ 0 & 1 \end{smallmatrix} \bigr)=\bigl(\begin{smallmatrix} x & l^{-\nu(l)}y \\ l^{\nu(l)}d_0y & x \end{smallmatrix} \bigr)\in \widehat{\Gamma}_0(l^{\nu(l)}) \; \text{for}\; l|N_{ns}\;.\]
It follows that $\mathbf{f}_m(\rho(\alpha)g_1)=\mathbf{f}_m(g_1(g_1^{-1}\rho(\alpha)g_1))=\mathbf{f}_m(g_1)\boldsymbol{\psi}_m(x)$ if $\alpha= x+y\sqrt{d_0}\equiv x\,\mod\, l^{\nu(l)}$. On the other hand, using that the conductor of $\chi_m$ at $l$ is precisely $l^{\nu(l)}$ we have $\chi_m(\alpha)=\chi_m(x)=\boldsymbol{\psi}_m^{-1}(x)$ due to the assumption $\chi_m|_{\mathbb{A}^\times} = \boldsymbol{\psi}_m^{-1}$. Thus we conclude $\mathbf{f}_m(\rho(\alpha)g_1)\chi_m(\alpha)=\mathbf{f}_m(g_1)\boldsymbol{\psi}_m(x)\chi_m(\alpha)=\mathbf{f}_m(g_1)$ as desired.
\end{proof}
Using Lemma \ref{improv} we immediately conclude that
\begin{equation} \label{int-sum}
L_{\chi_m}(\mathbf{f}_m) =\frac{\mathop{vol}(\mathrm{Cl}_M)}{|\mathrm{Cl}_M^-(N_{ns}p^s)|}\sum_{j=1}^{|\mathrm{Cl}_M^-(N_{ns}p^s)|}\chi_m(a_j)\mathbf{f}_m(\rho(a_j)g_1)
= \frac{\varphi_{\mathbb{Q}}(N_{ns}p^s)}{2\varphi_M(N_{ns}p^s)}\sum_{j=1}^{h^-}\chi_m(a_j)\mathbf{f}_m(\rho(a_j)g_1) 
\end{equation}
where we set $h^-=h^-(N_{ns}p^s):=|\mathrm{Cl}_M^-(N_{ns}p^s)|$ and $\{a_1,\ldots,a_{h^-}\}\subset M^{\times}_{\mathbb{A}}$ is a complete representative set so that $M^{\times}_{\mathbb{A}}=\bigsqcup_{j=1}^{h^-}M^\times a_j\widehat{R}_{N_{ns}p^s}^\times M_\infty^\times$. 

Our next goal is to give an algebro-geometric interpretation of $\mathbf{f}_m(\rho(a_j)g_1)$, for $j=1,\ldots,h^-$, by constructing CM points on Shimura curve $Sh$ such that values of $d^mf$ at these CM points coincide with $\mathbf{f}_m(\rho(a_j)g_1)$ after dividing both former and latter with suitable CM periods, that is, in the sense of (\ref{K-S}).

To this end, recall that $Sh(\mathbb{C})= G(\mathbb{Q})\left\backslash \left( \mathfrak{X}\times G(\mathbb{A}^{(\infty)})\right) \right/ Z(\mathbb{Q})$ and write $[z,g]\in Sh(\mathbb{C})$ for the image of $(z,g)\in \mathfrak{X}\times G(\mathbb{A}^{(\infty)})$. We restrict our attention to points $x=[z_1,g]$ where $z_1$ is already chosen generator of $R$ used to define regular representation $\rho=\rho_{z_1}:M^\times \hookrightarrow G(\mathbb{Q})$ by $\begin{pmatrix} \alpha z_1 \\ \alpha \end{pmatrix}=\rho_{z_1}(\alpha)\begin{pmatrix} z_1 \\ 1 \end{pmatrix}$.
Tensoring with $\mathbb{A}^{(\infty)}$ we may regard $\rho_{z_1}$ as a representation $\hat{\rho}_{z_1}:(M^{(\infty)}_{\mathbb{A}})^{\times}\hookrightarrow G(\mathbb{A}^{(\infty)})$ and further conjugating by $g$ we get $\hat{\rho}_{x}:(M^{(\infty)}_{\mathbb{A}})^{\times}\hookrightarrow G(\mathbb{A}^{(\infty)})$ given by $\hat{\rho}_{x}=g^{-1}\hat{\rho}_{z_1}(\alpha)g$. To each point $(E,\eta)\in Sh$ we can associate a lattice $\widehat{L}=\eta^{-1}(\mathcal{T}(E))\subset (\mathbb{A}^{(\infty)})^2$ and the level structure $\eta$ is determined by the choice of the basis $w=(w_1,w_2)$ of $\widehat{L}$ over $\widehat{\mathbb{Z}}$. In the view of basis $w$, the $G(\mathbb{A}^{(\infty)})$-action on $Sh$ given by $(E,\eta)\mapsto(E,\eta\circ g)$ is a matrix multiplication $w^\intercal\mapsto g^{-1}w^\intercal$ because $(\eta\circ g)^{-1}(\mathcal{T}(E))=g^{-1}\eta^{-1}(\mathcal{T}(E))=g^{-1}\widehat{L}$, where $\intercal$ stands for a transpose. 
\begin{switch} \label{switch} { \em
The action of matrix $g^{-1}$ records change of the basis vectors themselves, rather than coordinates with respect to the basis, as this is more natural in the modular point of view. Having this on mind and desiring to view modular forms in adelic, algebro-geometric and $p$-adic phrasing in coherent way, it becomes more convenient for us to use identifications $R\otimes_{\mathbb{Z}}\mathbb{Z}_p=R_{\bar{\mathfrak{p}}}\oplus  R_{\mathfrak{p}}$ and  $\X{R}[p^n]=\X{R}[\bar{\mathfrak{p}}^n]\oplus\X{R}[\mathfrak{p}^n]=\mathbb{Z}/p^n\mathbb{Z}\oplus\mu_{p^n}$, $n\geq 1$, in constructing level structures for our CM points due to the definition of nebentypus. }
\end{switch}
The fiber $E_x$ at $x\in Sh(\mathbb{C})$ of the universal abelian scheme over $Sh_{/\mathbb{Q}}$ has complex multiplication by an order of $M$, that is, under the action of $\widehat{R}$ via $\hat{\rho}_{x}$, $g^{-1}\widehat{R}\cap\mathbb{Q}^2$ is identified with a fractional ideal of an order of $M$ prime to $p$. The level structure $\eta_x=\eta_{z_1}\circ g$ identifies $\mathcal{T}(E_x)$ with $g^{-1}\widehat{R}$ where $\eta_{z_1}$ is a level structure arising from $\widehat{\mathbb{Z}}$-basis $(z_1,1)$ of $\widehat{R}$ so that $[z_1,1]\in  Sh(\mathbb{C})$ corresponds to $(\X{R},\eta_{z_1})$ while $x=[z_1,g]$ corresponds to $(E_x,\eta_x)=(\X{R},\eta_{z_1}\circ g)$ in the moduli interpretation of $Sh_{/\mathbb{Q}}$.

In particular, for the choice of $h_1\in G(\widehat{\mathbb{Z}})$ we made, one may assume that the point $[z_1,h_1]$ corresponds to CM point $x(R)=\xw{R}{\mathcal{W}}$ constructed in Section \ref{subsec:cm} since action of matrices in $G(\widehat{\mathbb{Z}})$ preserves $\widehat{R}$, that is, their only impact on the level structure $\eta_{z_1}$ is change of basis of $\widehat{R}$. It follows that the point  $[z_1,h_1\circ\bigl(\begin{smallmatrix} p^s & 1 \\ 0 & 1 \end{smallmatrix} \bigr)]$ corresponds to $\left(X(R),\eta^{(p)}(R),(\et{R}\times\ord{R})\circ\bigl(\begin{smallmatrix} p^s & 1 \\ 0 & 1 \end{smallmatrix} \bigr)\right)_{/\mathcal{W}}$ on $Sh_{/\mathbb{Q}}$ where we regard $\bigl(\begin{smallmatrix} p^s & 1 \\ 0 & 1 \end{smallmatrix} \bigr)$ as an element of $G(\mathbb{A}^{(\infty)})$ being trivial at places outside $p$. 

Recall that $\X{R}[p^s]=\X{R}[\bar{\mathfrak{p}}^s]\oplus\X{R}[\mathfrak{p}^s]=\mathbb{Z}/p^s\mathbb{Z}\oplus\mu_{p^s}$ over $\mathcal{W}$. Let $\zeta_{p^s}$ and $\gamma_{p^s}$ be the canonical generators of $\mu_{p^s}$ and $\mathbb{Z}/p^s\mathbb{Z}$, respectively, and consider rank $p^s$ finite flat subgroup scheme $C=\langle\zeta_{p^s}^{-1}\gamma_{p^s}\rangle$ of $\X{R}[p^s]$ defined over $\mathcal{W}[\mu_{p^s}]$. 
By repeating the argument of Section \ref{subsec:cm}, we get that $\X{R}/C=\X{\mathfrak{a}}$ where $\mathfrak{a}$ is a proper $R_{p^s}$-ideal such that $\mathfrak{a}R=\bar{\mathfrak{p}}^{-s}R$. The quotient map $\pi:\X{R}\twoheadrightarrow\X{R}/C$ is \'etale and we have a commutative diagram with exact rows
\[ \begin{CD}
  @.       @.             C                       @.   @.  \\
 @.       @.             @VVV                  @.   @.  \\
0 @>>>  \mu_{p^\infty} @>>> \X{R}[p^\infty] @>>>  \mathbb{Q}_p/\mathbb{Z}_p @>>> 0  \\
  @.      @|           @VV\pi V             @VVp^sV  \\
0 @>>>  \mu_{p^\infty} @>>> \X{\mathfrak{a}}[p^\infty] @>>>  \mathbb{Q}_p/\mathbb{Z}_p @>>> 0 
\end{CD} \]
and a resulting CM point
\[x(\mathfrak{a})=\xw{\mathfrak{a}}{\mathcal{W}[\mu_{p^s}]}\]
on $Sh$, equipped with an invariant differential $\omega(\mathfrak{a})$ on \X{\mathfrak{a}}. 

If $(w_1,w_2)$ is the $\widehat{\mathbb{Z}}$-basis of $\widehat{R}$ giving rise to the level structure $\eta(R)=\left(\eta^{(p)}(R),\et{R}\times\ord{R}\right)$ appearing in $x(R)$ then we can identify $\gamma_{p^s}(\mathbb{C})$ with $\frac{w_1}{p^s}$ and $\zeta_{p^s}(\mathbb{C})$ with $\frac{w_2}{p^s}$ thus identifying $C(\mathbb{C})$ with $\left(\widehat{\mathbb{Z}}\frac{w_1-w_2}{p^s}+\widehat{\mathbb{Z}}w_2 \right)\big/\left(\widehat{\mathbb{Z}}w_1+\widehat{\mathbb{Z}}w_2\right)$. It follows that making quotient $\pi:x(R)\to x(\mathfrak{a})$ is tantamount to moving $x(R)$ on $Sh$ by $G(\mathbb{A}^{(\infty)})$-action of $\bigl(\begin{smallmatrix} p^s & 1 \\ 0 & 1 \end{smallmatrix} \bigr)$. Indeed, the  $\widehat{\mathbb{Z}}$-basis $(w_1,w_2)$ of $\widehat{R}$ is being sent to $(\frac{w_1-w_2}{p^s},w_2)^\intercal=\bigl(\begin{smallmatrix} p^s & 1 \\ 0 & 1 \end{smallmatrix} \bigr)^{-1}(w_1,w_2)^\intercal$. More generally, making quotient of $\X{R}$ by a rank $p^s$ finite flat subgroup scheme $C_u=\langle\zeta_{p^s}^{-u}\gamma_{p^s}\rangle$ of \X{R}, for $1\leq u \leq p^s-1$ such that $\mathrm{gcd}(u,p)=1$, is tantamount to moving $x(R)$ on $Sh$ by $G(\mathbb{A}^{(\infty)})$-action of $\bigl(\begin{smallmatrix} p^s & u \\ 0 & 1 \end{smallmatrix} \bigr)$. This fact is nothing but dictionary between algebro-geometric and adelic phrasing of Hecke relation among CM points on Shimura curve $Sh$. What underlies such interpretation is Deligne's treatment of $Sh$ explained at the end of Section \ref{sec:shvar}.

Note that in the context of the above argument cases 1 and 2 in the Deligne's treatment of $Sh_{/\mathbb{Q}}$ amount to the same one because \'etale quotient of \Xw{R}{\mathcal{W}} by a rank $p^s$ finite flat subgroup scheme $C$ is at the same time \'etale covering of an elliptic curve isogenous to \Xw{R}{\mathcal{W}} via multiplication-by-$p^s$ map and thus both cases yield the same point on $Sh_{/\mathbb{Q}}$.

To summarize, $[z_1,h_1\bigl(\begin{smallmatrix} p^s & 1 \\ 0 & 1 \end{smallmatrix} \bigr)]$ corresponds to  $\left(X(R),\eta^{(p)}(R),(\et{R}\times\ord{R})\circ\bigl(\begin{smallmatrix} p^s & 1 \\ 0 & 1 \end{smallmatrix} \bigr)\right)$ on $Sh$ which in turn can be viewed as \xw{\mathfrak{a}}{\mathcal{W}[\mu_{p^s}]} for a certain proper  $R_{p^s}$-ideal $\mathfrak{a}$ such that $\mathfrak{a}R=\bar{\mathfrak{p}}^{-s}R$. 

Had we allowed a choice of subgroup scheme $C=\X{R}[\bar{\mathfrak{p}}^s]$ in the consideration above, we would obtain $\X{R}/C=\X{\bar{\mathfrak{p}}^{-s}R}$ whose corresponding CM point $x(\bar{\mathfrak{p}}^{-s}R)$ on $Sh$ is related to $x(R)$ via $G(\mathbb{A}^{(\infty)})$-action of $\bigl(\begin{smallmatrix} p^s & 0 \\ 0 & 1 \end{smallmatrix} \bigr)$. On the other hand, for $C=\X{R}[\mathfrak{p}^s]$ we have $\X{R}/C=\X{\mathfrak{p}^{-s}R}$ whose corresponding CM point $x(\mathfrak{p}^{-s}R)$ on $Sh$ is related to $x(R)$ via $G(\mathbb{A}^{(\infty)})$-action of $\bigl(\begin{smallmatrix}  1 & 0 \\ 0 & p^s \end{smallmatrix} \bigr)$. These two facts are known as Shimura's reciprocity law (\cite{ACM} 26.8 and \cite{minv} Section 3.2) and we may conventionally refer to maps that send $x(R)\mapsto x(\bar{\mathfrak{p}}^{-s}R)$ and $x(R)\mapsto x(\mathfrak{p}^{-s}R)$ on Shimura curve $Sh$, as Verschiebung and Frobenius maps, respectively,  even though we consider them when the characteristic of base ring is not $p$. This concludes aforementioned dictionary between algebro-geometric and adelic phrasing of Hecke relation among CM points on Shimura curve $Sh$.

The very same reasoning combined with Deligne's treatment of $Sh_{/\mathbb{Q}}$ implies that  $G(\mathbb{A}^{(\infty)})$-action of $\bigl(\begin{smallmatrix} l^{\nu(l)} & 0 \\ 0 & 1 \end{smallmatrix} \bigr)$, for split $l\in C_{sp}$, on \xw{\mathfrak{a}}{\mathcal{W}[\mu_{p^s}]} will move this point on $Sh$ to another \xw{\mathfrak{b}}{\mathcal{W}[\mu_{p^s}]} for a certain proper $R_{p^s}$-ideal $\mathfrak{b}$ whose class in $\mathrm{Cl}_M^-(p^s)$ is possibly different  from the one of $\mathfrak{a}$. The fact that action of $\bigl(\begin{smallmatrix} l^{\nu(l)} & 0 \\ 0 & 1 \end{smallmatrix} \bigr)$, for a split $l\in C_{sp}$, on \[\xw{\mathfrak{a}}{\mathcal{W}[\mu_{p^s}]} \]
will not impact the conductor of associated lattice is encoded in Deligne's treatment in the sense that diagonal matrix $\bigl(\begin{smallmatrix} l^{\nu(l)} & 0 \\ 0 & 1 \end{smallmatrix} \bigr)$ commutes with $\rho(R)$ since we realized $\rho(\alpha)=\bigl(\begin{smallmatrix} \bar{\alpha} & 0 \\ 0 & \alpha \end{smallmatrix} \bigr)$ for $\alpha\in R\otimes_{\mathbb{Z}}\mathbb{Z}_l=R_{\bar{\mathfrak{l}}}\oplus R_{\mathfrak{l}}=\mathbb{Z}_l\oplus\mathbb{Z}_l$, and consequently $\widehat{\mathbb{Z}}$-lattice $\bigl(\begin{smallmatrix} l^{\nu(l)} & 0 \\ 0 & 1 \end{smallmatrix} \bigr)^{-1}\widehat{R}$ is still stable under $\widehat{R}$, so that $\bigl(\begin{smallmatrix} l^{\nu(l)} & 0 \\ 0 & 1 \end{smallmatrix} \bigr)^{-1}\widehat{R}\cap\mathbb{Q}^2$ is a $\mathbb{Z}$-lattice still stable under $R$. 

However, for a ramified or inert prime $l|N_{ns}$, the matrix $\bigl(\begin{smallmatrix} l^{\nu(l)} & 0 \\ 0 & 1 \end{smallmatrix} \bigr)$ does not commute with $\rho(R)$ since we realized $\rho(x+y\sqrt{d_0})=\bigl(\begin{smallmatrix} x & y \\ d_0y & x \end{smallmatrix} \bigr)$ for $x+y\sqrt{d_0}\in (R\otimes_{\mathbb{Z}}\mathbb{Z}_l)^\times$ and the action of $\bigl(\begin{smallmatrix} l^{\nu(l)} & 0 \\ 0 & 1 \end{smallmatrix} \bigr)$ changes the conductor of associated lattice in 
\[\xw{\mathfrak{b}}{\mathcal{W}[\mu_{p^s}]}\; ,\] 
namely, $\bigl(\begin{smallmatrix} l^{\nu(l)} & 0 \\ 0 & 1 \end{smallmatrix} \bigr)^{-1}\widehat{R}\cap\mathbb{Q}^2$ is a lattice that is a proper $R_{l^{\nu(l)}}$-ideal of $R$, hence this action will move 
\[\xw{\mathfrak{b}}{\mathcal{W}[\mu_{p^s}]}\]
to some \xw{\mathfrak{c}}{\mathcal{W}[\mu_{p^s}]} for certain proper $R_{l^{\nu(l)}p^s}$-ideal $\mathfrak{c}$. This means that ultimately $[z_1,g_1^{(\infty)}]$ corresponds to point \xw{\mathfrak{d}}{\mathcal{W}[\mu_{p^s}]} on $Sh$ for certain proper $R_{N_{ns}p^s}$-ideal $\mathfrak{d}$. That being said, we can choose $\sigma_1\in \mathrm{Gal}(W[\mu_{p^s}]/W)$ such that $\X{\mathfrak{d}}^{\sigma_1}=\X{R_{N_{ns}p^s}}$ and replace the embedding $\iota_p : \bar{\mathbb{Q}} \hookrightarrow \mathbb{C}_p$ we fixed in the introduction, with $\iota_p\circ\sigma_1$. This allows us to assume that $[z_1,g_1^{(\infty)}]$ corresponds to point \xw{R_{N_{ns}p^s}}{\mathcal{W}[\mu_{p^s}]} on $Sh$. Indeed, replacing $\iota_p$ with $\iota_p\circ\sigma_1$ is tantamount to replacing \X{R}, $z_1$ and $C_1=\langle\zeta_{p^s}^{-1}\gamma_{p^s}\rangle$ with \X{R^{\sigma_1}}, $\sigma_1(z_1)$ and $C_1^{\sigma_1}=\langle\sigma_1(\zeta_{p^s})^{-1}\sigma_1(\gamma_{p^s})\rangle$, respectively, in the above argument. 

Set $\eta=(\eta^{(p)},\eta_p^{\mathrm{\acute{e}t}}\times \eta_p^{\mathrm{ord}})$ and let $\mathfrak{a}_1,\ldots,\mathfrak{a}_{h^-}$ be the complete set of representatives of proper $R_{N_{ns}p^s}$-ideal classes in $\mathrm{Cl}_M^-(N_{ns}p^s)$ so that $\widehat{\mathfrak{a}}_j=a_j\widehat{R}_{N_{ns}p^s}$ for $j=1,\ldots,h^-$. Note that our definition of $\rho:(M^{(\infty)}_{\mathbb{A}})^{\times}\hookrightarrow G(\mathbb{A}^{(\infty)})$ can be conveyed in the following commutative diagram for $\alpha\in(M^{(\infty)}_{\mathbb{A}})^{\times}$: 
\[\begin{CD}
\mathcal{T}(\X{\alpha R}) @<\alpha<< \mathcal{T}(\X{R})\\
@A{\cong}A\eta(\alpha R)A @A{\cong}A\eta(R)A\\
\alpha\widehat{R} @<\rho(\alpha)<< \widehat{R}
\end{CD}\]
so that $\eta(\alpha R)\circ\rho(\alpha)=\eta(R)$ and consequently $x(\alpha R)=x(R)\circ \rho(\alpha)^{-1}$. In particular, when $\alpha\in R^\times$, $\rho(\alpha)$ fixes point $x(R)$ on $Sh$. Similarly, commutative diagrams
\[\begin{CD}
\mathcal{T}(\X{\mathfrak{a}_j}) @<a_j<< \mathcal{T}(\X{R_{N_{ns}p^s}})\\
@A{\cong}A\eta(\mathfrak{a}_j)A @A{\cong}A\eta(R_{N_{ns}p^s})A\\
\widehat{\mathfrak{a}}_j @<\rho(a_j)<< \widehat{R}_{N_{ns}p^s}
\end{CD}\]
tell us that  $\eta(\mathfrak{a}_j)\circ\rho(a_j)=\eta(R_{N_{ns}p^s})$ whence $x(\mathfrak{a}_j)=x(R_{N_{ns}p^s})\circ \rho(a_j^{-1})$.
We may choose representatives $\mathfrak{a}_1,\ldots,\mathfrak{a}_{h^-}$ such that $a_{j,l}=1$ at finite set of primes $l|N_{ns}p$ so that $\rho(a_j^{-1})$ commutes with $g_1$. 

In conclusion, the points $[z_1,\rho(a_j^{-1})g_1^{(\infty)}]$ in $Sh(\mathbb{C})$ correspond to 
\[x(\mathfrak{a}_j)=\xw{\mathfrak{a}_j}{\mathcal{W}[\mu_{p^s}]} \]
on $Sh_{/\mathbb{Q}}$.

Thus we have
{\allowdisplaybreaks
\begin{align} \nonumber
\left(\frac{\varphi_{\mathbb{Q}}(N_{ns}p^s)}{2\varphi_M(N_{ns}p^s)}\right)^{-1}L_{\chi_m}(\mathbf{f}_m)
&\stackrel{(\ref{int-sum})}{=} \sum_{j=1}^{h^-}\chi_m(a_j^{-1})\mathbf{f}_m(\rho(a_j^{-1})g_1) \\
&\stackrel{(\ref{der})}{=} |\mathrm{det}(g_1^{(\infty)})|_\mathbb{A}^{-m} \sum_{j=1}^{h^-}\chi_m(a_j^{-1})|\mathrm{det}(\rho(a_j^{-1}))|_{\mathbb{A}}^{-m} \delta_k^m\mathbf{f}(\rho(a_j^{-1})g_1) \nonumber \\
&=j(g_{1,\infty},\mathbf{i})^{-k-2m}|\mathrm{det}(g_1^{(\infty)})|_\mathbb{A}^{-m}\sum_{j=1}^{h^-}\chi_m(a_j^{-1})\left(|\mathrm{det}(\rho(a_j^{-1}))|_{\mathbb{A}}^{-m}(\delta_k^mf)[z_1,\rho(a_j^{-1})g_1^{(\infty)}])\right) \nonumber \\
&=j(g_{1,\infty},\mathbf{i})^{-k-2m}|\mathrm{det}(g_1^{(\infty)})|_\mathbb{A}^{-m}\sum_{j=1}^{h^-}\chi_m(a_j^{-1})\left(|a_j^{-1}|_{M_\mathbb{A}}^{-m} (\delta_k^mf)(x(\mathfrak{a}_j),\omega_\infty(\mathfrak{a}_j))\right) \nonumber \\
&=j(g_{1,\infty},\mathbf{i})^{-k-2m}|\mathrm{det}(g_1^{(\infty)})|_\mathbb{A}^{-m}\sum_{j=1}^{h^-}\left(\chi_m(a_j^{-1})|a_j|_{M_\mathbb{A}}^m\right) (\delta_k^mf)(x(\mathfrak{a}_j),\omega_\infty(\mathfrak{a}_j)) \label{adelic-sum}
\end{align} }

Let $\mathfrak{A}$ be a proper $R_{N_{ns}}$-ideal. Note that $\X{\mathfrak{A}}[p^n]=\X{\mathfrak{A}}[\bar{\mathfrak{p}}^n]\oplus\X{\mathfrak{A}}[\mathfrak{p}^n]=\mathbb{Z}/p^n\mathbb{Z}\oplus\mu_{p^n}$ over $\mathcal{W}$ for all $n\geq 1$. Clearly, in the above argument of determining Hecke orbit of $x(R)$ on $Sh$ one can replace \X{R} by \X{\mathfrak{A}} and fixed $s$ by arbitrary $n\geq 1$ and the argument is still valid. We summarize this for future reference in the following
\begin{adelic}\label{adelic}
Let $C\subset \X{\mathfrak{A}}[p^n]$ be a rank $p^n$ finite flat subgroup scheme of $\X{\mathfrak{A}}[p^n]$ \'etale  over $\mathcal{W}[\mu_{p^n}]$ that is \'etale locally isomorphic to $\mathbb{Z}/p^n\mathbb{Z}$ after faithfully flat extension of scalars, $n\geq1$. Then
\begin{itemize}
\item[(Frob)] If $C=\X{\mathfrak{A}}[\mathfrak{p}^n]=\mu_{p^n}$ then $x(\mathfrak{A})/C=x(\mathfrak{A})\circ \bigl(\begin{smallmatrix} 1 & 0 \\ 0 & p^n \end{smallmatrix} \bigr)=x(\mathfrak{p}^{-n}\mathfrak{A})$
\item[(Ver)] 
\begin{enumerate}
\item If $C=\X{\mathfrak{A}}[\bar{\mathfrak{p}}^n]$ then $x(\mathfrak{A})/C=x(\mathfrak{A})\circ \bigl(\begin{smallmatrix} p^n & 0 \\ 0 & 1 \end{smallmatrix} \bigr)=x(\bar{\mathfrak{p}}^{-n}\mathfrak{A})$
\item If $C=\langle\zeta_{p^n}^{-u}\gamma_{p^n}\rangle$ for some $1\leq u \leq p^n-1$ such that $\mathrm{gcd}(u,p)=1$, then $x(\mathfrak{A})/C=x(\mathfrak{A})\circ \bigl(\begin{smallmatrix} p^n & u \\ 0 & 1 \end{smallmatrix} \bigr)=x(\tilde{\mathfrak{a}}_u)$ for a proper $R_{N_{ns}p^n}$-ideal $\tilde{\mathfrak{a}}_u$ such that $\tilde{\mathfrak{a}}_uR_{N_{ns}}=\bar{\mathfrak{p}}^{-n}\mathfrak{A}R_{N_{ns}}$.
\end{enumerate}
\end{itemize}
\end{adelic}
 
\section{Serre--Tate deformation space}
We recall some basic facts from deformation theory of elliptic curves over complete local $W$-algebras whose residue field is $\bar{\mathbb{F}}_p$. Denote by $CL_{/W}$ the category of such algebras. We follow Katz's exposition \cite{Ka78}.
Fix a proper $R_{N_{ns}}$-ideal $\mathfrak{A}$ prime to $p$ and consider $x(\mathfrak{A})\in Sh$. The Serre--Tate deformation space $\widehat{S}$ represents the functor $\widehat{\mathcal{P}}:CL_{/W}\to SETS$ given by
\[\widehat{\mathcal{P}}(A)=\{E_{/A}|E\otimes_A\bar{\mathbb{F}}_p=\Xw{\mathfrak{A}}{\bar{\mathbb{F}}_p}\}_{/\cong}\;.\]
If $\widehat{Sh}^{(p)}_{x/W}$ is formal completion of $Sh^{(p)}_{/\mathcal{W}}$ along $x=(\X{\mathfrak{A}},\eta^{(p)}(\mathfrak{A}))\in Sh^{(p)}(\bar{\mathbb{F}}_p)$, then by universality of Shimura curve $Sh^{(p)}$ we have $\widehat{Sh}^{(p)}_{x/W}\cong\widehat{S}_{/W}$. Indeed, $\widehat{Sh}^{(p)}_{x}$ classifies $(E,\eta^{(p)}_E)_{/A}$ with $(E,\eta^{(p)}_E)\otimes_{A}\bar{\mathbb{F}}_p=(\X{\mathfrak{A}},\eta^{(p)}(\mathfrak{A}))$ and since $E[N]$ for $N$ prime to $p$ is \'etale over $\mathop{Spec}(A)$ the level structure $\eta^{(p)}(\mathfrak{A})$ at the special fiber extends  uniquely to $\eta^{(p)}_E$ on $E_{/A}$. On the other hand, Serre--Tate deformation theory yields canonical isomorphism $\widehat{S}_{/W}\cong\widehat{\mathbb{G}}_{m/W}$. Indeed, $E_{/A}\in\widehat{\mathcal{P}}(A)$ is determined by extension class of connected component--\'etale quotient exact sequence of Barsotti--Tate groups
\begin{equation}\label{cceq}
0\longrightarrow E[p^\infty]^\circ \stackrel{i_E}{\longrightarrow} E[p^\infty] \stackrel{i^*_E}{\longrightarrow} E[p^\infty]^{\acute{e}t} \longrightarrow 0
\end{equation}
and by a theorem of Serre and Tate (Theorem 2.3 in \cite{minv}) such an extension class over $A$ is classified by
\[\mathrm{Hom}( E[p^\infty]^{\acute{e}t}_{/A},E[p^\infty]^\circ_{/A})\cong\mathrm{Hom}({\mathbb{Q}_p/\mathbb{Z}_p}_{/A},{\mu_{p^\infty}}_{/A})= \varprojlim_n\mu_{p^n}(A)=\widehat{\mathbb{G}}_m(A)\;.\]
In order to make this identification, we used fixed $\ord{\mathfrak{A}}:\mu_{p^\infty}\cong \X{\mathfrak{A}}[p^\infty]^\circ$ and its Cartier dual inverse $\et{\mathfrak{A}}:\mathbb{Q}_p/\mathbb{Z}_p\cong\X{\mathfrak{A}}[p^\infty]^{\acute{e}t}$. 

Note that in the preceding argument we could think of deformation of  \xw{\mathfrak{A}}{\bar{\mathbb{F}}_p} since for every complete local algebra $A\in CL_{/W}$ and every deformation $E_{/A}$ of \X{\mathfrak{A}}, $E[p^\infty]^{\acute{e}t}$ is \'etale over $\mathop{Spec}(A)$ and the deformation is insensitive of $p$-part $\eta_p=\eta^{\mathrm{\acute{e}t}}\times \eta^{\mathrm{ord}}$ of the full level structure $\eta=\eta^{(p)}\times\eta_p$. We briefly explain this point. Let $(\mathbb{E},\eta_{\mathbb{E}})$ be the universal pair over $Sh^{(p)}$. The Igusa tower over $Sh^{(p)}_{/\bar{\mathbb{F}}_p}$ is defined by
\[Ig:=\varprojlim_n\mathrm{Isom}_{\mathrm{gp-sch}}(\mu_{p^n/Sh^{(p)}},\mathbb{E}[p^n]^\circ_{/Sh^{(p)}})\]
If $V^{(p)}$ is geometrically irreducible component of $Sh^{(p)}$ containing point $x(\mathfrak{A})$ then we write $Ig_{/V^{(p)}}$ for the pull-back of the Igusa tower to  $V^{(p)}_{/\bar{\mathbb{F}}_p}$. The scheme $Ig$ is \'etale faithfully flat over the ordinary locus of $V^{(p)}_{/\bar{\mathbb{F}}_p}$ and the $p$-ordinary level structure $\ord{\mathfrak{A}}$ defines a point $x$ on $Ig_{/V^{(p)}}$. Moreover the theorem of Serre and Tate canonically identifies Serre--Tate deformation space $\widehat{S}_{/W}\cong\widehat{\mathbb{G}}_{m/W}$ to the formal completion $\widehat{Ig}$ at $x$ via $\ord{\mathfrak{A}}$.  

Let $t$ be the canonical coordinate of Serre--Tate deformation space $\widehat{S}$ so that \[\widehat{S}\cong\widehat{\mathbb{G}}_m=\mathop{Spf}(\varprojlim_n W[t,t^{-1}]/(t-1)^n)=Spf(W[[T]])\quad (T=t-1)\; .\]
Note that since we used fixed $\ord{\mathfrak{A}}:\mu_{p^\infty}\cong \X{\mathfrak{A}}[p^\infty]^\circ$ and $\et{\mathfrak{A}}:\mathbb{Q}_p/\mathbb{Z}_p\cong\X{\mathfrak{A}}[p^\infty]^{\acute{e}t}$ to identify $\widehat{S}\cong\widehat{\mathbb{G}}_m$, Serre--Tate coordinate $t=t_{\mathfrak{A}}$ depends on the point $x(\mathfrak{A})\in Sh$. We adopt convention to write $t$ for $t_{\mathfrak{A}}$ where $\mathfrak{A}$ is the proper $R_{N_{ns}}$-ideal fixed throughout this section and that considering Serre--Tate coordinate with respect to origin other than $x(\mathfrak{A})$ will be emphasized by adding appropriate subscript to $t$.

Let $(\boldsymbol{\mathcal{E}},\boldsymbol{\eta})$ be universal deformation of \X{\mathfrak{A}} over $\widehat{S}\cong\widehat{\mathbb{G}}_m$. For each $p$-adic modular form $f\in V(N;W)$ we call the expansion
\[f(t):=f(\boldsymbol{\mathcal{E}},\boldsymbol{\eta})\in W[[T]]\quad (T=t-1)\]
a $t$-expansion of $f$ with respect to Serre--Tate coordinate around $x(\mathfrak{A})$. We have the following $t$-expansion principle
\[\tag{t-exp} \text{The }t\text{-expansion: }f\mapsto f(t)\in W[[T]] \text{ determines }f\text{ uniquely.}\]
A key fact is that by its construction (\cite{Ka78} 4.3.1) Katz $p$-adic differential operator is $\widehat{S}$-invariant and canonical Serre--Tate coordinate $t$ is normalized so that 
\begin{equation}\label{katzdiff}
d=t\frac{\mathrm{d}}{\mathrm{d}t} \; .
\end{equation}
Note that the $t$-expansion of $f$ with respect to Serre--Tate coordinate $t$ around point $x(\mathfrak{A})$ can be computed as Taylor expansion of $f$ with respect to variable $T$ by applying $\frac{\mathrm{d}}{\mathrm{d}T}$ and evaluating result at $x(\mathfrak{A})$ (see (4.6) of \cite{minv}).

For any deformation \xw{E}{A} of \xw{\mathfrak{A}}{\bar{\mathbb{F}}_p} where $A\in CL_{/W}$, we can compute the assigned value of Serre--Tate coordinate $t(E_{/A})\in\widehat{\mathbb{G}}_m(A)$ in the following way (Sections 2.1-2.3 of \cite{minv}). Since any object in $CL_{/W}$ is a projective limit of artinian objects, we first assume that $A$ is an artinian object in  $CL_{/W}$. Using $\ord{E}:\mu_{p^\infty}\cong E[p^\infty]^\circ$ and $\et{E}:\mathbb{Q}_p/\mathbb{Z}_p\cong E[p^\infty]^{\acute{e}t}$, the connected component--\'etale quotient exact sequence of Barsotti--Tate groups (\ref{cceq}) becomes
\[ 0\longrightarrow \mu_{p^\infty} \stackrel{i_E}{\longrightarrow} E[p^\infty] \stackrel{i^*_E}{\longrightarrow} \mathbb{Q}_p/\mathbb{Z}_p\longrightarrow 0 \;. \]
Note that $E^\circ(A)$ is killed by $p^{n_0}$ for sufficiently large $n_0$ due to Drinfeld's theorem in deformation theory (Theorem 2.1 in \cite{minv}). Starting from $y\in E(\bar{\mathbb{F}}_p)$ we can always lift it to $\tilde{y}\in E(A)$ because of smoothness of $E_{/A}$ and the lift $\tilde{y}$ is determined modulo $\mathrm{Ker}(E(A)\rightarrow E(\bar{\mathbb{F}}_p))= E^\circ$ which is a subgroup of $E[p^n]$ if $n\geq n_0$. Thus $p^n\tilde{y}\in E^\circ(A)$ is uniquely determined by $y\in E(\bar{\mathbb{F}}_p)$ and if $y\in E[p^n]$ then $p^n\tilde{y}\in E^\circ(A)$ yielding a homomorphism $``{p^n}":E[p^n](\bar{\mathbb{F}}_p)\rightarrow E^\circ(A)$, via $y\mapsto p^n\tilde{y}$, called Drinfeld lift of multiplication by $p^n$ (\cite{Ka78} Lemma 1.1.2). Taking $1\in\mathbb{Z}_p\cong\mathcal{T}(\mathbb{Q}_p/\mathbb{Z}_p)$ and viewing it as $1=\varprojlim_n\frac{1}{p^n}$ for $\frac{1}{p^n}\in\mathbb{Q}_p/\mathbb{Z}_p[p^n]$, the value $``{p^n}" i_E^{*-1}(\frac{1}{p^n})\in \mu_{p^n}(A)$ becomes stationary if $n\geq n_0$ and we have
\[t(E_{/A})=\varprojlim_n ``{p^n}"i_E^{*-1}\left(\frac{1}{p^n}\right)\in \varprojlim_n \mu_{p^n}(A)=\widehat{\mathbb{G}}_m(A)\;.\]
If an object $A \in CL_{/W}$ is not artinian and $E$ is a deformation over $A$, writing $A=\varprojlim_B B$ for artinian $B$, we define $t(E_{/A})= \varprojlim_B t(E\times_A B_{/B})$.
Note that $t(E_{/A})=q_{E/A}(1,1)$ where $q_{E/A}(\cdot,\cdot):\mathcal{T}E[p^\infty]^{\acute{e}t}\times \mathcal{T}E[p^\infty]^{\acute{e}t}\rightarrow \widehat{\mathbb{G}}_m(A)$ is a bilinear form in Section 2 of \cite{Ka78} that corresponds to deformation $E_{/A}$ in establishing representability of deformation functor $\widehat{\mathcal{P}}$ by formal torus $\mathrm{Hom}_{\mathbb{Z}_p}(\mathcal{T}E[p^\infty]^{\acute{e}t}\times \mathcal{T}E[p^\infty]^{\acute{e}t},\widehat{\mathbb{G}}_m)$.

By definition, $t(\X{\mathfrak{A}})=1$ because the connected component--\'etale quotient exact sequence of $X(\mathfrak{A})[p^\infty]$ splits over $\mathcal{W}$ by complex multiplication and the existence of section $\mathscr{S}$
\[\xymatrix@1{ 0 \ar[r] & \mu_{p^\infty} \ar[r]^-{i_\mathfrak{A}} & X(\mathfrak{A})[p^\infty] \ar[r]^-{i^*_{\mathfrak{A}}} & \mathbb{Q}_p/\mathbb{Z}_p \ar@/^1pc/[l]^-{\mathscr{S}} \ar[r] & 0
}\]
allows us to verify
\[t(X(\mathfrak{A}))=\varprojlim_n ``{p^n}"i_{\mathfrak{A}}^{*-1}\left(\frac{1}{p^n}\right)=\varprojlim_n p^n\mathscr{S}\left(\frac{1}{p^n}\right)= \varprojlim_n 1= 1 \;. \]
We also have $t(\boldsymbol{\mathcal{E}})=t$ by definition and the pair $(E,\eta)_{/A}$ is the fiber of universal $(\boldsymbol{\mathcal{E}},\boldsymbol{\eta})$ at a point $t(E)\in\widehat{\mathbb{G}}_m(A)$. 

For every $s\geq 1$ there is a bijective correspondence:
\[\{ \text{rank } p^s \text{ \'etale subgroup schemes } \mathcal{C} \text{ of }  \boldsymbol{\mathcal{E}}[p^s] \text{ over }  \widehat{W[t,t^{-1}][t^{1/p^s}, \mu_{p^s}]} \}  \]
\begin{equation} 
\label{etale} \updownarrow 
\end{equation}
\[ \{ \text{rank } p^s \text{ \'etale subgroup schemes } C \text{ of } X(\mathfrak{A})[p^s] \text{ over }  \mathcal{W}[\mu_{p^s}] \} \nonumber
\]
given by
\[ \mathcal{C} \longleftrightarrow C=\mathcal{C}\times_{\boldsymbol{\mathcal{E}}[p^s]} X(\mathfrak{A})[p^\infty] \; .
\]
Recall that $\X{\mathfrak{A}}[p^s]=\X{\mathfrak{A}}[\bar{\mathfrak{p}}^s]\oplus\X{\mathfrak{A}}[\mathfrak{p}^s]=\mathbb{Z}/p^s\mathbb{Z}\oplus\mu_{p^s}$ over $\mathcal{W}$. The effect of Frobenius map on the Serre--Tate coordinate is given by the following
\begin{frob} If $\mathcal{C}\subset \boldsymbol{\mathcal{E}}[p^s]$ is rank $p^s$ finite flat subgroup scheme \'etale over $\widehat{W[t,t^{-1}][t^{1/p^s}, \mu_{p^s}]}$ corresponding to $C=\mu_{p^s}$ in (\ref{etale}), then $t(\boldsymbol{\mathcal{E}}/\mathcal{C})=t^{p^s}$.
\end{frob}
\begin{proof}
We have a commutative diagram with exact rows
\[ \begin{CD}
  @.       @.             \mathcal{C}                       @.   @.  \\
  @.       @.             @VVV                  @.   @.  \\
0 @>>>  \mu_{p^\infty} @>i_{\mathcal{E}}>> \mathcal{E}[p^\infty] @>i^*_{\mathcal{E}}>>  \mathbb{Q}_p/\mathbb{Z}_p @>>> 0  \\
  @.         @VVp^sV         @VV\pi V             @|  \\
0 @>>>  \mu_{p^\infty} @>i_{\mathcal{E}/\mathcal{C}}>> \mathcal{E}/\mathcal{C}[p^\infty] @>i^*_{\mathcal{E}/\mathcal{C}}>>  \mathbb{Q}_p/\mathbb{Z}_p @>>> 0 
\end{CD} \qquad . \]
from which we read
\[t=\varprojlim_n ``{p^n}"i_{\mathcal{E}}^{*-1}\left(\frac{1}{p^n}\right)\]
\[t(\mathcal{E}/\mathcal{C})= \varprojlim_n ``{p^n}"i_{\mathcal{E}/\mathcal{C}}^{*-1}\left(\frac{1}{p^n}\right)\]
and $t(\mathcal{E}/\mathcal{C})=t^{p^s}$
\end{proof}
On the other hand, the effect of Verschiebung map is given by the following
\begin{ver} If $\mathcal{C}\subset \boldsymbol{\mathcal{E}}[p^s]$ is rank $p^s$ finite flat subgroup scheme \'etale over $\widehat{W[t,t^{-1}][t^{1/p^s}, \mu_{p^s}]}$ corresponding to cyclic $C$ with $C\cap\mu_{p^s}=\{0\}$ in (\ref{etale}), then $t(\boldsymbol{\mathcal{E}}/\mathcal{C})^{p^s}=t$. More precisely, if $C=\langle \zeta_{p^s}^{-u}\gamma_{p^s} \rangle$, then $t(\boldsymbol{\mathcal{E}}/\mathcal{C})=\zeta_{p^s}^ut^{1/p^s}$.
\end{ver}
\begin{proof} 
We have a commutative diagram with exact rows
\[ \begin{CD}
  @.       @.             \mathcal{C}                       @.   @.  \\
  @.       @.             @VVV                  @.   @.  \\
0 @>>>  \mu_{p^\infty} @>i_{\mathcal{E}}>> \mathcal{E}[p^\infty] @>i^*_{\mathcal{E}}>>  \mathbb{Q}_p/\mathbb{Z}_p @>>> 0  \\
  @.        @|          @VV\pi V         @VVp^sV      \\
0 @>>>  \mu_{p^\infty} @>i_{\mathcal{E}/\mathcal{C}}>> \mathcal{E}/\mathcal{C}[p^\infty] @>i^*_{\mathcal{E}/\mathcal{C}}>>  \mathbb{Q}_p/\mathbb{Z}_p @>>> 0 
\end{CD} \qquad . \]
Then
\[t(\mathcal{E}/\mathcal{C})= \varprojlim_n ``{p^n}"i_{\mathcal{E}/\mathcal{C}}^{*-1}\left(\frac{1}{p^n}\right) = \varprojlim_n ``{p^n}"i_{\mathcal{E}}^{*-1}\left(\frac{1}{p^{n+s}}\right) \]
from which we read
\[t(\mathcal{E}/\mathcal{C})^{p^s}=\left(\varprojlim_n ``{p^n}"i_{\mathcal{E}}^{*-1}\left(\frac{1}{p^{n+s}}\right)\right)^{p^s}=\varprojlim_n ``{p^{n+s}}"i_{\mathcal{E}}^{*-1}\left(\frac{1}{p^{n+s}}\right)=t\]
We choose a $p^s$-th root $t^{1/p^s}$ so that $t(\mathcal{E}/\mathcal{C})=t^{1/p^s}$ when $\mathcal{C}$ corresponds to $X(\mathfrak{A})[\bar{\mathfrak{p}}^s]$ in (\ref{etale}).

If $\mathcal{C}$ corresponds to $C=\langle \zeta_{p^n}^{-u}\gamma_{p^n} \rangle$ in (\ref{etale}), then specializing at the fiber at $t=1$, we can read off which exact $p^s$-th root of $t$ is $t(\mathcal{E}/\mathcal{C})$. The above commutative diagram at the fiber at $t=1$ is
\[\xymatrixcolsep{3pc} \xymatrix{ & & C \ar[d] & \\
0 \ar[r] & \mu_{p^\infty} \ar[r]^-{i_{\mathfrak{A}}} \ar@{=}[d] & X(\mathfrak{A})[p^\infty] \ar[r]^-{i^*_{\mathfrak{A}}} \ar[d]^-\pi & \mathbb{Q}_p/\mathbb{Z}_p \ar[d]^-{p^s} \ar@/^1pc/[l]^-{\mathscr{S}} \ar[r] & 0 \quad .\\
0 \ar[r] & \mu_{p^\infty} \ar[r]^-{i_{X(\mathfrak{A})/C}} & X(\mathfrak{A})/C[p^\infty] \ar[r]^-{i^*_{X(\mathfrak{A})/C}} & \mathbb{Q}_p/\mathbb{Z}_p  \ar[r] & 0}
\]
By Shimura's reciprocity law, the Verschiebung map moves the origin $x(\mathfrak{A})$ of the Serre--Tate deformation space to $x(\bar{\mathfrak{p}}^{-s}\mathfrak{A})$ and we conclude
{\allowdisplaybreaks
\begin{align} \nonumber
t_{\bar{\mathfrak{p}}^{-s}\mathfrak{A}}(X(\mathfrak{A})/C)&=\varprojlim_n ``{p^n}"i_{X(\mathfrak{A})/C}^{*-1}\left(\frac{1}{p^n}\right) 
=\varprojlim_n ``{p^n}"\pi\left(i_{\mathfrak{A}}^{*-1}\left(\frac{1}{p^{n+s}}\right)\right) \nonumber \\
&= \varprojlim_n {p^n}\pi\left(\mathscr{S}\left(\frac{1}{p^{n+s}}\right)\right) 
= \varprojlim_n \pi(\gamma_{p^s})= \varprojlim_n \zeta^u_{p^s}  \nonumber \\
&= \zeta^u_{p^s} \nonumber
\end{align} }
This concludes the proof of the lemma.
\end{proof} 
Specializing the last two lemmas at the fiber at $t=1$ and combining them with Proposition \ref{adelic} we obtain 
\begin{Serre--Tate}\label{Serre--Tate}
Let $C\subset \X{\mathfrak{A}}[p^n]$ be a rank $p^n$ finite flat subgroup scheme of $\X{\mathfrak{A}}[p^n]$ \'etale over $\mathcal{W}[\mu_{p^n}]$ that is \'etale locally isomorphic to $\mathbb{Z}/p^n\mathbb{Z}$ after faithfully flat extension of scalars, $n\geq 1$. Then
\begin{itemize}
\item[(Frob)] $t_{\mathfrak{p}^{-n}\mathfrak{A}}=t_{\mathfrak{A}}^{p^n}$
\item[(Ver)] 
\begin{enumerate}
\item $t_{\bar{\mathfrak{p}}^{-n}\mathfrak{A}}=t_{\mathfrak{A}}^{1/p^n}$
\item If $C=\langle\zeta_{p^n}^{-u}\gamma_{p^n}\rangle$ for some $1\leq u \leq p^n-1$ such that $\mathrm{gcd}(u,p)=1$, then $t_{\bar{\mathfrak{p}}^{-n}\mathfrak{A}}(\X{\tilde{\mathfrak{a}}_u})=\zeta_{p^n}^u$ where $\tilde{\mathfrak{a}}_u$ is a proper $R_{N_{ns}p^n}$-ideal  such that $x(\mathfrak{A})/C=x(\mathfrak{A})\circ \bigl(\begin{smallmatrix} p^n & u \\ 0 & 1 \end{smallmatrix} \bigr)=x(\tilde{\mathfrak{a}}_u)$ and $\tilde{\mathfrak{a}}_uR_{N_{ns}}=\bar{\mathfrak{p}}^{-n}\mathfrak{A}R_{N_{ns}}$.
\end{enumerate}
\end{itemize}
\end{Serre--Tate}
The point here is that by Shimura's reciprocity law, Frobenius and Verschiebung maps move the origin $x(\mathfrak{A})$ of the Serre--Tate deformation space to $x(\mathfrak{p}^{-n}\mathfrak{A})$ and $x(\bar{\mathfrak{p}}^{-n}\mathfrak{A})$, respectively.

\section{Measure associated to a Hecke eigen-cusp form}

We recall notation from Section \ref{sec:hecke}. Let $f_0\in S_k(\Gamma_0(N_0),\psi)$ be a normalized Hecke newform of conductor $N_0$, nebentypus $\psi$ and let  $\mathbf{f}_0$ be the corresponding adelic form with central character $\boldsymbol{\psi}$. Let $f$ be a suitable normalized Hecke eigen-cusp form that will be explicitly made out of $f_0$ in Section \ref{sec:mainthm} such that its arithmetic lift $\mathbf{f}$ is in the automorphic representation $\pi_{\mathbf{f}_0}$ generated by the unitarization $\mathbf{f}_0^u$.
Following Katz, we construct a $W$-valued measure $\mathrm{d}\mu_f$ on $\mathrm{Cl}_M^-(N_{ns}p^\infty)$ interpolating values of $f$ at CM points we constructed in Section \ref{subsec:cm} and whose behavior under Hecke action is studied in Section \ref{sec:hecke}. 

\subsection{Interlude on $p$-adic measures on $\mathbb{Z}_p$} We recall some basic facts from the theory of $p$-adic measures on $\mathbb{Z}_p$ (\cite{LFE} Section 3.5). For a $p$-adic measure $\mu$ on $\mathbb{Z}_p$ having values in $W$ we can define corresponding formal power series $\Phi_\mu(t)$ by
\[\Phi_\mu(t)=\sum_{n=0}^\infty \left(\int_{\mathbb{Z}_p}\binom{x}{n}\mathrm{d}\mu(x)\right)T^n \quad (T=t-1)\;.\]
Essentially by Mahler's theorem, the measure $\mu$ is determined by the formal power series $\Phi_\mu(t)$ and $\mu\mapsto\Phi_\mu$ induces an isomorphism $\mathscr{M}(\mathbb{Z}_p;W)\cong W[[T]]$ between space $\mathscr{M}(\mathbb{Z}_p;W)$ of all $W$-valued measures on $\mathbb{Z}_p$ and ring of formal power series $W[[T]]$. It is then easy to verify that
\[
\int_{\mathbb{Z}_p}x^m\mathrm{d}\mu=\left(t\frac{\mathrm{d}}{\mathrm{d}t}\right)^m\Phi_\mu|_{t=1}\quad \text{for all}\, m\geq0\,.
\]
The space of measures $\mathscr{M}(\mathbb{Z}_p;W)$ is naturally a module over the ring $\mathscr{C}(\mathbb{Z}_p;W)$ of continuous $W$-valued functions on $\mathbb{Z}_p$ in the following way: for $\phi\in\mathscr{C}(\mathbb{Z}_p;W)$ and $\mu \in \mathscr{M}(\mathbb{Z}_p;W)$
\[ \int_{\mathbb{Z}_p}\varphi\mathrm{d}(\phi\mu)=\int_{\mathbb{Z}_p}\varphi(x)\phi(x)\mathrm{d}\mu(x)\quad \text{for all}\, \varphi\in\mathscr{C}(\mathbb{Z}_p;W)\,. \]
In particular, if $\phi\in\mathscr{C}(\mathbb{Z}_p;W)$ is a locally constant function factoring through $\mathbb{Z}_p/p^n\mathbb{Z}_p$, we have 
\begin{equation}\label{average}
\Phi_{\phi\mu}=[\phi]\Phi_\mu \text{ where }[\phi]\Phi_\mu(t)=p^{-n}\sum_{b\in\mathbb{Z}/p^n\mathbb{Z}}\phi(b)\sum_{\zeta\in\mu_{p^n}}\zeta^{-b}\Phi_\mu(\zeta t)
\end{equation}
and
\begin{equation}\label{twmoment}
\int_{\mathbb{Z}_p}\phi(x)x^m\mathrm{d}\mu(x)=\left(t\frac{\mathrm{d}}{\mathrm{d}t}\right)^m([\phi]\Phi_\mu)|_{t=1}\;.
\end{equation}

\subsection{Construction of measure}\label{subsec:measure}

Let $\mathfrak{A}$ be a proper $R_{N_{ns}}$-ideal prime to $p$. Classical modular forms are defined over a number field, so we may assume that $f$ is defined over a localization $\mathcal{V}$ of a ring of integers of a certain number field. We take a finite extension of $W$ generated by $\mathcal{V}$ and, abusing the symbol, we keep denoting it $W$. Then $(d^mf)(x(\mathfrak{A}),\omega_p(\mathfrak{A}))\in W$ by rationality result of Katz. We define a $W$-valued measure $\mathrm{d}\mu_{f,\mathfrak{A}}$ on $\mathbb{Z}_p$ by
\[\int_{\mathbb{Z}_p}\binom{x}{n}\mathrm{d}\mu_{f,\mathfrak{A}}(x)=\binom{d}{n}f(x(\mathfrak{A}),\omega_p(\mathfrak{A})) \quad \text{for all}\, n\geq0\,. \]
Then clearly
\[ \int_{\mathbb{Z}_p}x^m\mathrm{d}\mu = (d^mf)(x(\mathfrak{A}),\omega_p(\mathfrak{A})) \quad \text{for all}\, m\geq0\,. \]

We make an observation in the form of the following
\begin{taylor}\label{taylor}
$\Phi_{\mu_{f,\mathfrak{A}}}(t)$ is the $t$-expansion of $f$ with respect to Serre--Tate coordinate $t$ around point $x(\mathfrak{A})$.
\end{taylor}
\begin{proof}
One can easily verify elementary identity $\binom{t\frac{\mathrm{d}}{\mathrm{d}t}}{n}=\frac{t^n}{n!}\frac{\mathrm{d}^n}{\mathrm{d}t^n}$ (\cite{LFE} Lemma 3.4.1 on page 80). The proposition then follows from (\ref{katzdiff}) and the fact that $t(\X{\mathfrak{A}})=1$. Indeed, if $T=t-1$, we have
\[f(t)=\sum_{n=0}^\infty \frac{1}{n!}\frac{\mathrm{d}^nf}{\mathrm{d}t^n}(1)T^n = \sum_{n=0}^\infty \binom{t\frac{\mathrm{d}}{\mathrm{d}t}}{n}f(1)T^n= \sum_{n=0}^\infty \binom{d}{n}f(x(\mathfrak{A}),\omega_p(\mathfrak{A}))T^n= \sum_{n=0}^\infty \left( \int_{\mathbb{Z}_p}\binom{x}{n}\mathrm{d}\mu_{f,\mathfrak{A}}\right)T^n = \Phi_{\mu_{f,\mathfrak{A}}}(t).\]
\end{proof}

For $\alpha\in \mathrm{GL}_2^+(\mathbb{R})$ we define
\[f|_k\alpha=\mathrm{det}(\alpha)^{k/2}f(\alpha(z))(cz+d)^{-k}\]
Note that operator $|_k$ depends on the weight of the form, but since the weight will always be clear from the context we shall write it as $|$.

Note that for our purpose it suffices to work over Igusa tower $Ig_{N^{(p)}}$ over $\mathfrak{M}(\Gamma_1(N^{(p)}))$, where $N^{(p)}$ is prime-to-$p$ level. Thus, in characteristic 0, the modular form is on $\mathfrak{M}(\Gamma_1(N^{(p)}p^r))$ for some $r\geq 0$.
We are now ready to prove the following

\begin{shift}\label{shift}
For all $n\geq 1$ and every $1\leq u \leq p^n-1$ such that $\mathrm{gcd}(u,p)=1$, we have:
\begin{enumerate}
\item $\Phi_{\mu_{f,\mathfrak{A}}}(\zeta_{p^n}^ut)$ is $t$-expansion of $f|\bigl(\begin{smallmatrix} 1 & -up^{-n} \\ 0 & 1 \end{smallmatrix} \bigr)$ with respect to Serre--Tate coordinate $t$ around $x(\mathfrak{A})$.
\item $\Phi_{\mu_{f,\mathfrak{A}}}(\zeta_{p^n}^ut)|_{t=1}=f(x(\mathfrak{a}_u),\omega_p(\mathfrak{a}_u))$ where $x(\mathfrak{a}_u)= x(\mathfrak{A})\circ \bigl(\begin{smallmatrix} 1 & up^{-n} \\ 0 & 1 \end{smallmatrix} \bigr)$ on $Sh/\widehat{\Gamma}_1(N^{(p)}p^r)$ and $\mathfrak{a}_u$ is a proper $R_{N_{ns}p^n}$-ideal such that $\mathfrak{a}_uR_{N_{ns}}=\mathfrak{A}R_{N_{ns}}$. 
\end{enumerate}
\end{shift}
\begin{proof}
Each point $x$ in the infinitesimal neighborhood of a CM point $x(\mathfrak{A})$ in the Igusa tower $Ig_{N^{(p)}}$ over $\mathfrak{M}(\Gamma_1(N^{(p)}))$ is determined by two data: the Serre--Tate coordinate $t(x)=t_{\mathfrak{A}}$ and the prime-to-$p$ level structure $\eta^{(p)}(\mathfrak{A})$ depending on the center point $x(\mathfrak{A})$. The action of $g\in G(\mathbb{A}^{(\infty)})$ has two aspects, one is moving the Serre--Tate coordinate $t$, and another is action on the level structure given by $\eta\mapsto\eta\circ g$. Note that if $g_p=1$, i.e. $g \in G(\mathbb{A}^{(p\infty)})$, then $g$ preserves $t$ (as the coordinate of $\widehat{\mathbb{G}}_m$) but possibly moves $x(\mathfrak{A})$ to a different point on $Sh$.

Let $\widetilde{p}$ be global $\widetilde{p}=(p,\ldots,p)\in \mathbb{Q}^\times \subset \mathbb{A}^\times$ and write $\widetilde{p}_p$ for local $\widetilde{p}_p=(1,\ldots,1,\stackrel{p}{p},1,\ldots,1)\in \mathbb{Q}_p^\times \subset \mathbb{A}^\times$. 
We view $u\in \mathbb{Z}_p^\times$ as an element of $\mathbb{A}^\times$ being nontrivial at $p$ and having zeros outside $p$ (i.e. $u_l=0$ if $l\not = p$).  Consider the action on the Serre--Tate deformation space $\widehat{S}_{x(\mathfrak{A})}$ of $x(\mathfrak{A})$, which is the infinitesimal neighborhood of  $x(\mathfrak{A})$ in the Igusa tower $Ig_{N^{(p)}}$, of the following product 
\begin{equation} \label{eq:maneuver}
\bigl(\begin{smallmatrix} 1 & 0 \\ 0 & \widetilde{p}^n \end{smallmatrix} \bigr) \bigl(\begin{smallmatrix}  \widetilde{p}^n & u \\ 0 & 1 \end{smallmatrix} \bigr)= \bigl(\begin{smallmatrix} \widetilde{p}^n & 0 \\ 0 & \widetilde{p}^n \end{smallmatrix} \bigr)\bigl(\begin{smallmatrix} 1 & up^{-n} \\ 0 & 1 \end{smallmatrix} \bigr)\,.
\end{equation}
The effect on the Serre--Tate coordinate $t$ can be read off from the left hand side and is precisely $t\mapsto \zeta_{p^n}^ut$ by Proposition \ref{Serre--Tate}. Note that the central element $\widetilde{p}^n\in Z(\mathbb{Q})$ on the right hand side fixes $Sh$ point by point, and the unipotent $\bigl(\begin{smallmatrix} 1 & up^{-n} \\ 0 & 1 \end{smallmatrix} \bigr)$  fixing $\eta_p^{\mathrm{ord}}(\mathfrak{A})$ preserves $\widehat{S}_{x(\mathfrak{A})}$ as a whole space, even though it moves points within it.  That being said, the passage of Serre--Tate coordinate from $t$ to $\zeta_{p^n}^ut$ is solely caused by $\bigl(\begin{smallmatrix} 1 & up^{-n} \\ 0 & 1 \end{smallmatrix} \bigr)$. 

Let $x=(X,\eta^{(p)}, \eta^{\mathrm{\acute{e}t}}\times \eta^{\mathrm{ord}})$ be a general test object that gives rise to a point in the ordinary locus of $Sh$, and let $\omega_p(x)$ be the invariant differential induced by $\eta^{\mathrm{ord}}$ as in Section \ref{diffop}. It follows that $\Phi_{\mu_{f,\mathfrak{A}}}(\zeta_{p^n}^ut)$ is $t$-expansion around $x(\mathfrak{A})$ of the $p$-adic modular form given by
 \[f': x\mapsto f((x,\omega_p(x))\circ \bigl(\begin{smallmatrix} 1 & up^{-n} \\ 0 & 1 \end{smallmatrix} \bigr))\]
In order to complete the proof of (1), we are left to prove a subtle point that this $p$-adic modular form is nothing but $f|\bigl(\begin{smallmatrix} 1 & up^{-n} \\ 0 & 1 \end{smallmatrix} \bigr)^{-1}$. Let $\widetilde{u}$ be global $\widetilde{u}=(u,\ldots,u)\in \mathbb{Q}^\times \subset \mathbb{A}^\times$. If we set $\alpha=\bigl(\begin{smallmatrix} 1 & \widetilde{u}\widetilde{p}^{-n} \\ 0 & 1 \end{smallmatrix}\bigr)\in G(\mathbb{Q})\subset G(\mathbb{A})$ then the crux of the proof is the following identity that holds for all $m\geq 0$:
\begin{equation} \label{subtle}
\delta_k^mf((x,\omega_\infty(x))\circ \alpha_p) = \delta_k^m (f|\alpha_\infty^{-1})(x ,\omega_\infty(x))\;.
\end{equation}
To verify it, recall that if $x\in Sh$ corresponds to $[z,g^{(\infty)}]$ in $Sh(\mathbb{C})$, for some $z\in\mathfrak{X}$ and $g^{(\infty)}\in  G(\mathbb{A}^{(\infty)})$, then by definition
\[\delta_k^mf(x ,\omega_\infty(x))=\delta_k^mf([z,g^{(\infty)}])=\delta_k^m\mathbf{f}(g)j(g_\infty, \mathbf{i})^{k+2m} \]
where $g_\infty \in G(\mathbb{R})$ is such that $g_\infty(\mathbf{i})=z$ and $g=g^{(\infty)}g_\infty$.
Since in characteristic 0 the modular form $f$ is on $\mathfrak{M}(\Gamma_1(N^{(p)}p^r))$, we only need to check identity over $Sh/\widehat{\Gamma}_1(N^{(p)}p^r)$; so without loss of generality we may assume that $g^{(\infty)}=1$.
{\allowdisplaybreaks \begin{align} \nonumber
\delta_k^mf((x ,\omega_\infty(x))\circ \alpha_p)&=\delta_k^m f([z,1]\circ \alpha_p) = \delta_k^m f([z,\alpha_p]) \nonumber \\
&= \delta_k^m\mathbf{f}( \alpha_p g_\infty ) j(g_\infty, \mathbf{i} )^{k+2m} \nonumber \\
&= \delta_k^m\mathbf{f}\left(\alpha (\alpha^{(p\infty)})^{-1} \alpha_\infty^{-1}g_\infty \right) j(g_\infty, \mathbf{i} )^{k+2m}  \nonumber \\
&=  \delta_k^m\mathbf{f} \left(\alpha_\infty^{-1}g_\infty \right) j(g_\infty, \mathbf{i} )^{k+2m} \nonumber \\
&=  \delta_k^m f (\alpha_\infty^{-1}g_\infty(\mathbf{i}))j(\alpha_\infty^{-1}g_\infty,\mathbf{i})^{-k-2m} j(g_\infty, \mathbf{i} )^{k+2m} \nonumber \\
&=  \delta_k^m f (\alpha_\infty^{-1}(z))j(\alpha_\infty^{-1},z)^{-k-2m}  \nonumber \\
&=  (\delta_k^m f)|\alpha_\infty^{-1}(z) =  \delta_k^m (f|\alpha_\infty^{-1})([z,1]) \nonumber \\
&=  \delta_k^m (f|\alpha_\infty^{-1})(x ,\omega_\infty(x)) \nonumber 
\end{align} }
as desired. Here we used a general fact $(\delta_k^mf)|\alpha_\infty^{-1}=\delta_k^m(f|\alpha_\infty^{-1})$ and an obvious fact that for this particular $\alpha$ we have $(\alpha^{(p\infty)})\in\widehat{\Gamma}_1(N^{(p)}p^r)$.

By the Katz--Shimura rationality result (\ref{K-S}) we have
\[ \frac{d^mf'(x,\omega_p(x))}{\Omega_p^{k+2m}}=\frac{d^mf((x,\omega_p(x))\circ\alpha_p)}{\Omega_p^{k+2m}} = \frac{\delta_k^mf((x ,\omega_\infty(x))\circ \alpha_p)}{\Omega_\infty^{k+2m}}=\frac{\delta_k^m (f|\alpha_\infty^{-1})(x ,\omega_\infty(x))}{\Omega_\infty^{k+2m}} = \frac{d^m (f|\alpha_\infty^{-1})(x ,\omega_p(x))}{\Omega_p^{k+2m}} \]
which yields
\[ d^mf'(x,\omega_p(x))= d^m (f|\alpha_\infty^{-1}) \text{ for all } m\geq 0\;.\]
This implies that $f'$ and $f|\alpha_\infty^{-1}$ have the same $t$-expansion (see statement (4.6) in \cite{minv}) and by the $t$-expansion principle we conclude $f'=f|\alpha_\infty^{-1}$ as desired. 

To prove (2), we first note that the action of $\bigl(\begin{smallmatrix} \widetilde{p}^n & u \\ 0 & 1\end{smallmatrix} \bigr)$  and 
$\bigl(\begin{smallmatrix} \widetilde{p}_p^n & u \\ 0 & 1 \end{smallmatrix} \bigr)$ are identical on $\mathfrak{M}(\Gamma_1(N^{(p)}p^r))$ due to coset identity
\[\bigl(\begin{smallmatrix} \widetilde{p}^n & u \\ 0  & 1\end{smallmatrix} \bigr)\widehat{\Gamma}_1(N^{(p)}p^r) = \bigl(\begin{smallmatrix} \widetilde{p}_p^n & u \\ 0 & 1 \end{smallmatrix} \bigr) \widehat{\Gamma}_1(N^{(p)}p^r) \]
that holds for any $r\geq 0$.

By using Proposition \ref{adelic}, we can track down the effect of the action of the left hand side of (\ref{eq:maneuver}) on $x(\mathfrak{A})$. Regarding the matrix $\bigl(\begin{smallmatrix} 1 & 0 \\ 0 & \widetilde{p}^n \end{smallmatrix} \bigr)$, note that $\bigl(\begin{smallmatrix} 1 & 0 \\ 0 & \widetilde{p}_p^n \end{smallmatrix} \bigr)$ sends $x(\mathfrak{A})=(X(\mathfrak{A}), \eta(\mathfrak{A}))$ to  $x(\mathfrak{p}^{-n}\mathfrak{A})=(X(\mathfrak{p}^{-n}\mathfrak{A}),\eta(\mathfrak{p}^{-n}\mathfrak{A}))$, which is then further moved by $\bigl(\begin{smallmatrix} 1 & 0 \\ 0 & (\widetilde{p}^n)^ {(p)}\end{smallmatrix} \bigr)\in G(\widehat{\mathbb{Z}}^{(p)})$ to some $(X(\mathfrak{p}^{-n}\mathfrak{A}),\eta'(\mathfrak{p}^{-n}\mathfrak{A}))$.  Dealing with the Igusa tower $Ig_{N^{(p)}}$ over $\mathfrak{M}(N^{(p)})$, we can switch from global $\bigl(\begin{smallmatrix} \widetilde{p}^n & u \\ 0  & 1\end{smallmatrix} \bigr)$ to local $\bigl(\begin{smallmatrix} \widetilde{p}_p^n & u \\ 0 & 1 \end{smallmatrix} \bigr)$ as explained above, and the Proposition \ref{adelic} yields that the action of $\bigl(\begin{smallmatrix} \widetilde{p}_p^n & u \\ 0 & 1 \end{smallmatrix} \bigr)$ sends $(X(\mathfrak{p}^{-n}\mathfrak{A}),\eta'(\mathfrak{p}^{-n}\mathfrak{A}))$ to some $x(\mathfrak{a}_u)=(X(\mathfrak{a}_u),\eta(\mathfrak{a}_u))$ on $Sh/\widehat{\Gamma}_1(N^{(p)}p^r)$ as desired.
\end{proof}

Recall that for any function $\phi:\mathbb{Z}/p^n\mathbb{Z}\to\mathbb{C}$, we define its Fourier transform $\phi^*:\mathbb{Z}/p^n\mathbb{Z}\to\mathbb{C}^\times$ by $\phi^*(x)=\sum_{u\in\mathbb{Z}/p^n\mathbb{Z}}\phi(u)\mathbf{e}(xu/p^n)$, where $\mathbf{e}(x)=\mathrm{exp}(2\pi\mathbf{i}x)$. The twist of $f$ by $\phi$ is given by
\[f|\phi=p^{-n}\sum_{u\in\mathbb{Z}/p^n\mathbb{Z}}\phi^*(-u)f|\bigl(\begin{smallmatrix} 1 & up^{-n} \\ 0 & 1 \end{smallmatrix} \bigr)\]
and is a Hecke eigen-cusp form in $S_k(\Gamma_0(Np^{2n}),\psi\phi^2)$ whose $q$-expansion is given by
\[f|\phi(z)=\sum_{j\geq1}\phi(j)a(j,f)q^n\]
where $q=\mathrm{exp}(2\pi\mathbf{i}z)$.

We now prove the key proposition.
\begin{key}\label{key}
For any function $\phi:\mathbb{Z}/p^n\mathbb{Z}\to\mathbb{C}$, $n\geq 1$,  we have 
\[ [\phi]\Phi_{\mu_{f,\mathfrak{A}}}=\Phi_{\mu_{f|\phi^-,\mathfrak{A}}}\]
where $\phi^-:\mathbb{Z}/p^n\mathbb{Z}\to\mathbb{C}$ is defined by $\phi^-(x):=\phi(-x)$.
In particular, if $\phi (\mathbb{Z}/p^n\mathbb{Z})^\times\to\mathbb{C}$ is a primitive Dirichlet character of conductor $p^n$, and we extend $\phi$ and $\phi^{-1}$ to be $0$ outside $(\mathbb{Z}/p^n\mathbb{Z})^\times$, we have
\[ \int_{\mathbb{Z}_p}\phi(x)x^m\mathrm{d}\mu_{f,\mathfrak{A}}(x) = G(\phi)\sum_{u\in (\mathbb{Z}/p^n\mathbb{Z})^\times}\phi^{-1}(-u)d^mf(x(\mathfrak{a}_u),\omega_p(\mathfrak{a}_u))\]
where $G(\phi)$ is the Gauss sum and $\mathfrak{a}_u$'s for $u\in (\mathbb{Z}/p^n\mathbb{Z})^\times$ are proper $R_{N_{ns}p^n}$-ideals as in part (2) of Lemma \ref{shift}.
\end{key}
\begin{proof}
From (\ref{average}) we have
\[ [\phi]\Phi_{\mu_{f,\mathfrak{A}}}(t)=p^{-n}\sum_{b\in\mathbb{Z}/p^n\mathbb{Z}}\phi(b)\sum_{u\in \mathbb{Z}/p^n\mathbb{Z}} \zeta_{p^n}^{-bu} \Phi_{\mu_{f,\mathfrak{A}}}(\zeta_{p^n}^u t) \]
and by Proposition \ref{taylor}, Lemma \ref{shift} and $t$-expansion principle it follows that $[\phi]\Phi_{\mu_{f,\mathfrak{A}}}$ is $t$-expansion of
\[
p^{-n}\sum_{b\in\mathbb{Z}/p^n\mathbb{Z}}\phi(b)\sum_{u\in \mathbb{Z}/p^n\mathbb{Z}} \zeta_{p^n}^{-bu} f|\bigl(\begin{smallmatrix} 1 & -up^{-n} \\ 0 & 1 \end{smallmatrix} \bigr)
\]
with respect to Serre--Tate coordinate $t$ around $x(\mathfrak{A})$. By $q$-expansion principle, this cusp form is nothing but
{\allowdisplaybreaks\begin{align} 
\nonumber & p^{-n}\sum_{b\in\mathbb{Z}/p^n\mathbb{Z}}\phi(b)\sum_{u\in \mathbb{Z}/p^n\mathbb{Z}} \zeta_{p^n}^{-bu} \left(\sum_{j\geq 1}a(j,f)q^j\right)|\bigl(\begin{smallmatrix} 1 & -up^{-n} \\ 0 & 1 \end{smallmatrix} \bigr) \nonumber \\
&= \sum_{j\geq 1}a(j,f)q^j\sum_{b\in\mathbb{Z}/p^n\mathbb{Z}}\phi(b)\left(p^{-n}\sum_{u\in \mathbb{Z}/p^n\mathbb{Z}} \zeta_{p^n}^{-(b+j)u}\right)  \nonumber \\
&= \sum_{j\geq 1}a(j,f)q^j \phi(-j) \nonumber \\
&= f|\phi^- \;. \nonumber
\end{align} }
Another application of Proposition \ref{taylor} now proves the first assertion.

Using Lemma \ref{shift} and (\ref{average}) we have
{\allowdisplaybreaks\begin{align} 
\nonumber ([\phi]\Phi_{\mu_{f,\mathfrak{A}}})|_{t=1} &= p^{-n}\sum_{b\in\mathbb{Z}/p^n\mathbb{Z}}\phi(b)\sum_{u\in \mathbb{Z}/p^n\mathbb{Z}} \zeta_{p^n}^{-bu} f(x(\mathfrak{a}_u),\omega_p(\mathfrak{a}_u)) \nonumber \\
&= \sum_{u\in \mathbb{Z}/p^n\mathbb{Z}} \left( p^{-n}\sum_{b\in\mathbb{Z}/p^n\mathbb{Z}}\phi(b) \zeta_{p^n}^{-bu} \right) f(x(\mathfrak{a}_u),\omega_p(\mathfrak{a}_u)) \nonumber \\
&= \sum_{u\in \mathbb{Z}/p^n\mathbb{Z}} \phi^*(-u) f(x(\mathfrak{a}_u),\omega_p(\mathfrak{a}_u)) \nonumber \\
&= G(\phi)\sum_{u\in \mathbb{Z}/p^n\mathbb{Z}} \phi^{-1}(-u) f(x(\mathfrak{a}_u),\omega_p(\mathfrak{a}_u)) \nonumber 
\end{align} }
To establish the last equality we used a well known fact that for a primitive character $\phi$ we have $\phi^*=\phi^{-1}G(\phi)$ (\cite{LFE} Lemma 2.3.2 on page 45). The second assertion now follows from (\ref{katzdiff}) and (\ref{twmoment}).
\end{proof}
This key proposition immediately yields a very useful
\begin{support}\label{support}
If the Fourier coefficients of $f$ are supported on integers $j$ that are prime to $p$, the measure $\mathrm{d}\mu_{f,\mathfrak{A}}$ is supported on $\mathbb{Z}_p^\times$.
\end{support}
\begin{proof}
This fact is furnished by taking $\phi$ to be characteristic function of $\mathbb{Z}_p^\times$ in the first part of Proposition \ref{key}. 
\end{proof}
Set
\[ f^{(p)}=\sum_{\mathrm{gcd}(j,p)=1}a(j,f)q^j \; .\]
Let $U_p$, $V_p$ and $T_p$ be standard Hecke operators acting on classical and $p$-adic modular forms. The effect of $U_p$ and $V_p$ on $q$-expansion of a modular form $f$ is given by
\[ f|U_p = \sum_{j\geq 1} a(jp,f)q^j \text{ and } f|V_p=\sum_{j\geq 1} a(j,f)q^{jp} \]
and we have
\[ T_p = U_p + \psi(p)p^{k-1}V_p \; .\]
It is easy to verify that
\begin{equation} \label{twist}
f^{(p)}=f|(1-U_pV_p)=f|(1-T_pV_p + \psi(p)p^{k-1}V_p^2 ) \; .
\end{equation}

Let $\{\mathfrak{A}_1,\ldots,\mathfrak{A}_{H^-}\}$ be a complete set of representatives for $\mathrm{Cl}_M^-(N_{ns})$ and let $\widehat{\mathfrak{A}}_j=A_j\widehat{R}_{N_{ns}}$ for $A_j\in M^{\times}_{\mathbb{A}}$, $j=1,\ldots,H^-$. We may assume that $A_{j,p}=1$ for all $j=1,\ldots,H^-$. An explicit coset decomposition 
\begin{equation}\label{coset}
\mathrm{Cl}_M^-(N_{ns}p^\infty)=\bigsqcup_{j=1}^{H^-}\mathfrak{A}_j^{-1}\mathbb{Z}_p^\times
\end{equation}
allows us to construct a $W$-valued measure $\mathrm{d}\mu_f$ on $\mathrm{Cl}_M^-(N_{ns}p^\infty)$ by constructing $H^-$ distinct $W$-valued measures $\mathrm{d}\mu_j$ on $\mathbb{Z}_p^\times$ so that for every continuous function $\xi:\mathrm{Cl}_M^-(N_{ns}p^\infty)\to W$ we have 
\[\int_{\mathrm{Cl}_M^-(N_{ns}p^\infty)} \xi\mathrm{d}\mu_f=\sum_{j=1}^{H^-}\int_{\mathbb{Z}_p^\times}\xi(\mathfrak{A}_j^{-1}x)\mathrm{d}\mu_j(x)\]

Fix once for all an arithmetic Hecke character $\lambda:M^{\times} \backslash M^{\times}_{\mathbb{A}}\to\mathbb{C}^\times$ so that $\lambda(a_\infty)=a_\infty^k$ and $\lambda|_{\mathbb{A}^\times}=\boldsymbol{\psi}^{-1}$, where $\boldsymbol{\psi}=\psi|\cdot|_{\mathbb{A}}^{-k}$ is a central character of $\mathbf{f}$. Due to criticality assumption $\lambda|_{\mathbb{A}^\times}=\boldsymbol{\psi}^{-1}$, the conductor of $\lambda$ is not independent of the conductor $c(\psi)$ of nebentypus $\psi$. Thus we give the following

\begin{recipe} 
Note that regardless of the choice of $\lambda$ we have:
\begin{itemize}
\item at primes $l\nmid c(\psi)$ Hecke character $\lambda$ is unramified and
\item at non-split primes $l|c(\psi)$ the conductor of $\lambda_l$ divides $l^{\mathrm{ord}_l(c(\psi))}$ hence does not exceed $l^{\nu(l)}$.
\end{itemize}
However, if a split prime $l|c(\psi)$ we make a very subtle choice as follows. Note that $M_l^\times= M_{\bar{\mathfrak{l}}}^\times \times M_{\mathfrak{l}}^\times  =\mathbb{Q}_l^\times\times\mathbb{Q}_l^\times$ and consequently $\lambda_l = \lambda_{\bar{\mathfrak{l}}}\lambda_{\mathfrak{l}}$. Then
\begin{itemize}
\item at split primes $l|c(\psi)$ we choose $\lambda_l$ so that its conductor is supported at $\mathfrak{l}$, that is, we choose $\lambda_{\bar{\mathfrak{l}}}$  to be unramified and $\lambda_{\mathfrak{l}}$ to have conductor $\mathfrak{l}^{\mathrm{ord}_l(c(\psi))}$. 
\end{itemize}
\end{recipe}
We take a finite extension of $W$ obtained by adjoining values of arithmetic Hecke character $\lambda$ and, abusing the symbol, we keep denoting it $W$. Let $\widehat{\lambda}:M^{\times} \backslash M^{\times}_{\mathbb{A}}\to W^\times$ defined by $\widehat{\lambda}(x)=\lambda(x)x_p^k$ be the $p$-adic avatar of $\lambda$. Then $W$-valued measures $\mathrm{d}\mu_j$ are given by
\[\mathrm{d}\mu_j=\widehat{\lambda}({\mathfrak{A}_j^{-1}})\mathrm{d}\mu_{f^{(p)},\mathfrak{A}_j} \qquad \text{for } j=1,\ldots,H^-.\]
We emphasize that important point here is that measures $\mathrm{d}\mu_{f^{(p)},\mathfrak{A}_j}$ are indeed supported on $\mathbb{Z}_p^\times$ due to Corollary \ref{support}. Note that since $A_{j,p}=1$ we actually have $\widehat{\lambda}({\mathfrak{A}_j^{-1}})=\lambda({\mathfrak{A}_j^{-1}})$.

\subsection{$p$-adic interpolation} \label{subsec:interpolation}
Let $\varphi:M^{\times} \backslash M^{\times}_{\mathbb{A}}\to\mathbb{C}^\times$ be an anticyclotomic arithmetic Hecke character such that $\varphi(a_\infty)=a_\infty^{m(1-c)}$ for some $m\geq 0$ and of conductor $N_{ns}p^s$ where $s\geq \mathrm{max}(1,\mathrm{ord}_p(N_0))$ is an arbitrary integer. Let $\widehat{\varphi}:M^{\times} \backslash M^{\times}_{\mathbb{A}}\to \mathbb{C}_p^\times$ defined by $\widehat{\varphi}(x)=\varphi(x)x_p^{m(1-c)}$ be its $p$-adic avatar. Recall that $(R\otimes_{\mathbb{Z}}\mathbb{Z}_p)^\times= R_{\bar{\mathfrak{p}}}^\times \oplus R_{\mathfrak{p}}^\times \cong\mathbb{Z}_p^\times\oplus\mathbb{Z}_p^\times$ and if $\mathrm{Cl}_M= M^\times \left\backslash M^{\times}_{\mathbb{A}}\right/ M^\times_\infty$ is the idele class group, the projection $\mathrm{pr}:\mathrm{Cl}_M\rightarrow \mathrm{Cl}_M^-(N_{ns}p^\infty)$ at $p$-th place is given by $\mathrm{pr}_p : x_p \mapsto x_p^{1-c}$. Thus, in the view of coset decomposition (\ref{coset}), $\widehat{\varphi}|_{\mathbb{Z}_p^\times}(z) =\tilde{\varphi}_p(z)z^m$ for $z=x_p^{1-c}$ where 
$\tilde{\varphi}_p=\varphi_p\circ \mathrm{pr}_p$ is a primitive Dirchlet character of conductor $p^s$.

The Mazur--Mellin transform of measure $\mathrm{d}\mu_f$ given by
\[ \mathscr{L}(f,\xi)=\int_{\mathrm{Cl}_M^-(N_{ns}p^\infty)}\xi{d}\mu_f\;, \qquad \xi \in \mathrm{Hom}_{cont}(\mathrm{Cl}_M^-(N_{ns}p^\infty),W^\times)\;,\]
is a $p$-adic analytic Iwasawa function on the $p$-adic Lie group $\mathrm{Hom}_{cont}(\mathrm{Cl}_M^-(N_{ns}p^\infty),W^\times)$ and for the $p$-adic avatar $\widehat{\varphi}$ we have: 
{\allowdisplaybreaks\begin{align} 
\nonumber \frac{\mathscr{L}(f,\widehat{\varphi})}{\mathrm{\Omega}_p^{k+2m}}&=\frac{1}{\mathrm{\Omega}_p^{k+2m}}
\int_{\mathrm{Cl}_M^-(N_{ns}p^\infty)}\widehat{\varphi} \mathrm{d}\mu_f \nonumber \\
&= \frac{1}{\mathrm{\Omega}_p^{k+2m}}\sum_{j=1}^{H^-}\widehat{\lambda}(\mathfrak{A}_j^{-1}) \int_{\mathbb{Z}_p^\times}\widehat{\varphi}(\mathfrak{A}_j^{-1}z)\mathrm{d}\mu_{f^{(p)},\mathfrak{A}_j}(z) \nonumber \\
&=\frac{1}{\mathrm{\Omega}_p^{k+2m}}\sum_{j=1}^{H^-}\widehat{\lambda}(\mathfrak{A}_j^{-1}) \widehat{\varphi}(\mathfrak{A}_j^{-1}) \int_{\mathbb{Z}_p^\times} \tilde{\varphi}_p(z)z^m \mathrm{d}\mu_{f^{(p)},\mathfrak{A}_j}(z) \nonumber \\
&= \frac{ G(\tilde{\varphi}_p)}{\mathrm{\Omega}_p^{k+2m}}\sum_{j=1}^{H^-}\widehat{\lambda}(\mathfrak{A}_j^{-1}) \widehat{\varphi}(\mathfrak{A}_j^{-1}) \sum_{u\in (\mathbb{Z}/p^s\mathbb{Z})^\times} \tilde{\varphi}_p^{-1}(-u) (d^mf^{(p)})(x(\mathfrak{a}_{j,u}),\omega_p(\mathfrak{a}_{j,u})) \label{p-adic-sum}
\end{align}}
Since the measures $\mathrm{d}\mu_{f^{(p)},\mathfrak{A}_j}$ are supported on $\mathbb{Z}_p^\times$ due to Corollary \ref{support}, we can use Proposition \ref{key} to establish the last equality. Thus, for each $j=1,\ldots,H^-$, we have $x(\mathfrak{a}_{j,u})=x(\mathfrak{A}_j)\circ \bigl(\begin{smallmatrix} 1 & up^{-n} \\ 0 & 1 \end{smallmatrix} \bigr)$ on $Sh/\widehat{\Gamma}_1(N^{(p)}p^r)$, $u\in (\mathbb{Z}/p^s\mathbb{Z})^\times$, and $\mathfrak{a}_{j,u}$'s are exactly $p^{s-1}(p-1)$ proper $R_{N_{ns}p^n}$-ideal class representatives such that $\mathfrak{a}_{j,u}R_{N_{ns}}=\mathfrak{A}_jR_{N_{ns}}$ (see Lemma \ref{shift} and Section \ref{subsec:cm}). Let $\widehat{\mathfrak{a}}_{j,u}=a_{j,u}\widehat{R}_{N_{ns}p^s}$ for $a_{j,u}\in M^{\times}_{\mathbb{A}}$.

Note that $A_{j,p}=1$ implies $\widehat{\lambda}(\mathfrak{A}_j^{-1})=\lambda(\mathfrak{A}_j^{-1})$ and  $\widehat{\varphi}(\mathfrak{A}_j^{-1})=\varphi(\mathfrak{A}_j^{-1})$. Since $\mathfrak{a}_{j,u}R_{N_{ns}}=\mathfrak{A}_jR_{N_{ns}}$ and $\widehat{\mathfrak{a}}_{j,u}=\bigl(\begin{smallmatrix} 1 & up^{-s} \\ 0 & 1 \end{smallmatrix} \bigr)^{-1}\widehat{\mathfrak{A}}_j=\bigl(\begin{smallmatrix} 1 & -up^{-s} \\ 0 & 1 \end{smallmatrix} \bigr)\widehat{\mathfrak{A}}_j$, we have \[(a_{j,u})_p=((a_{j,u})_{\bar{\mathfrak{p}}},(a_{j,u})_{\mathfrak{p}})=(-u,1)A_{j,p}=(-u,1)\; \mathrm{mod}\; (\bar{\mathfrak{p}}^s, \mathfrak{p}^s) \,.\] 
We conclude 
\[\lambda(\mathfrak{a}_{j,u})=\lambda_{\bar{\mathfrak{p}}}(-u)\lambda_{\mathfrak{p}}(1)\lambda(\mathfrak{A}_j)= \lambda(\mathfrak{A}_j)\] 
since even if $\lambda$ eventually has a non-trivial conductor at $p$ (precisely in the case $p|c(\psi)$) it is supported at $\mathfrak{p}$ by our choice. On the other hand, $\mathrm{pr}_p((-u,1))=(-u,1)^{1-c}=-u$ and we conclude
\[\varphi(\mathfrak{a}_{j,u})=\varphi_p((-u,1)^{1-c})\varphi(\mathfrak{A}_j)=\tilde{\varphi}_p(-u)\varphi(\mathfrak{A}_j)\,.\]

If we set $\chi_m=\lambda \cdot \varphi \cdot |\cdot|_{M_\mathbb{A}}^m$, its $\infty$-type is $(k+2m,0)$, that is, $\chi_m(a_\infty)=a_\infty^{k+2m}$, and its conductor is precisely $N_{ns}p^s\prod_{l\in C_{sp}}\mathfrak{l}^{\mathrm{ord}_l(c(\psi))}$ by our choice of $\lambda$. Moreover, $\chi_m|_{\mathbb{A}^\times}=\boldsymbol{\psi}^{-1}|\cdot|_{\mathbb{A}}^{2m}=\boldsymbol{\psi}_m^{-1}$ since the anticyclotomic character $\varphi$ is trivial on rational ideles $\mathbb{A}^\times$,  while $\lambda|_{\mathbb{A}^\times}=\boldsymbol{\psi}^{-1}$ by our choice. 

Then continuing from (\ref{p-adic-sum}), we have
{\allowdisplaybreaks\begin{align}\nonumber
\nonumber \frac{\mathscr{L}(f,\widehat{\varphi})}{\mathrm{\Omega}_p^{k+2m}}&= \frac{G(\tilde{\varphi}_p)}{\mathrm{\Omega}_p^{k+2m}} \sum_{j=1}^{H^-}\sum_{u\in (\mathbb{Z}/p^s\mathbb{Z})^\times}\lambda(\mathfrak{a}_{j,u}^{-1})\varphi(\mathfrak{a}_{j,u}^{-1}) (d^mf^{(p)})(x(\mathfrak{a}_{j,u}),\omega_p(\mathfrak{a}_{j,u})) \nonumber \\
&=\frac{G(\tilde{\varphi}_p)}{\mathrm{\Omega}_p^{k+2m}} \sum_{j=1}^{H^-}\sum_{u\in (\mathbb{Z}/p^s\mathbb{Z})^\times} \left(\chi_m(a_{j,u}^{-1})|a_{j,u}|_{M_\mathbb{A}}^m\right) (d^mf^{(p)})(x(\mathfrak{a}_{j,u}),\omega_p(\mathfrak{a}_{j,u})) \label{outside-p}
\end{align}}
Note that proper $R_{N_{ns}}p^s$-ideals $\mathfrak{a}_{j,u}$ in the above double sum form a complete set of class representatives for $\mathrm{Cl}_M^-(N_{ns}p^s)$ and moreover every summand is independent of the choice of proper $R_{N_{ns}}p^s$-ideal class representative. The latter immediately follows from Lemma \ref{improv} and Katz--Shimura rationality result (\ref{K-S}) (alternatively, we can directly quote statement (G3') on page 762 of \cite{HidaDwork}). Thus we can relabel this complete set of representatives of proper $R_{N_{ns}p^s}$-ideal classes in $\mathrm{Cl}_M^-(N_{ns}p^s)$ into $\mathfrak{a}_1,\ldots,\mathfrak{a}_{h^-}$ where $h^-=h^-(N_{ns}p^s):=|\mathrm{Cl}_M^-(N_{ns}p^s)|$ and $\widehat{\mathfrak{a}}_j=a_j\widehat{R}_{N_{ns}p^s}$ for $j=1,\ldots,h^-$. Moreover, we may assume that $a_{j,l}=1$ at all primes $l|N_{ns}p$. Then the double sum in question is
\[\mathcal{L}_{\chi_m}(d^mf^{(p)}):=\sum_{j=1}^{h^-}\left(\chi_m(a_j^{-1})|a_j|_{M_\mathbb{A}}^m\right) (d^mf^{(p)})(x(\mathfrak{a}_j),\omega_p(\mathfrak{a}_j))\;. \]
Now we prove the following
\begin{back} 
\[ \mathcal{L}_{\chi_m}(d^mf^{(p)})=\left(1-a(p,f)\chi_m(\bar{\mathfrak{p}})+\psi(p)p^{k-1}\chi_m(\bar{\mathfrak{p}}^2)\right)\mathcal{L}_{\chi_m}(d^mf)\]
\end{back}
\begin{proof}
We follow the notation from the proof of Lemma \ref{shift}. The key point of the proof is that
\[ (d^mf)|V_p (x(\mathfrak{a}_j),\omega_p(\mathfrak{a}_j))=d^mf\left((x(\mathfrak{a}_j),\omega_p(\mathfrak{a}_j))\circ\bigl(\begin{smallmatrix} \widetilde{p}_p & 0 \\ 0 & 1 \end{smallmatrix} \bigr)^{-1}\right)=d^mf(x(\bar{\mathfrak{p}}\mathfrak{a}_j),\omega_p(\bar{\mathfrak{p}}\mathfrak{a}_j)) \]
The second identity follows immediately from Shimura's reciprocity law while the first one can be established by repeating subtle argument from Lemma \ref{shift}.
Indeed, if we set $\alpha=\bigl(\begin{smallmatrix} \widetilde{p} & 0 \\ 0 & 1 \end{smallmatrix}\bigr)\in G(\mathbb{Q})\subset G(\mathbb{A})$, then $d^mf|V_p=p^{-k/2-m}d^m(f|\alpha_\infty)$ is evident from the $q$-expansion principle, and we have the following incarnation of the identity (\ref{subtle}):
\[\delta_k^mf((x,\omega_\infty(x))\circ \alpha_p^{-1}) = p^{-k/2-m}\delta_k^m (f|\alpha_\infty)(x ,\omega_\infty(x))\;.\]
Its proof is literally the same as the one of (\ref{subtle}), except that the present $\alpha_\infty$ not being unipotent causes appearance of the $p$-power factor coming from its determinant.

Using (\ref{twist}) we conclude
\begin{align}
d^mf^{(p)}(x(\mathfrak{a}_j),\omega_p(\mathfrak{a}_j)) &= \left(d^mf|(1-T_pV_p + \psi(p)p^{k-1}V_p^2)\right)(x(\mathfrak{a}_j),\omega_p(\mathfrak{a}_j))  \nonumber \\
&= d^mf(x(\mathfrak{a}_j),\omega_p(\mathfrak{a}_j))- p^ma(p,f)d^mf(x(\bar{\mathfrak{p}}\mathfrak{a}_j),\omega_p(\bar{\mathfrak{p}}\mathfrak{a}_j)) + \nonumber  \\ & \psi(p)p^{k+2m-1} d^mf(x(\bar{\mathfrak{p}}^{2}\mathfrak{a}_j),\omega_p(\bar{\mathfrak{p}}^2\mathfrak{a}_j))  \nonumber
\end{align}
for $j=1,\ldots,h^-$. After multiplying by $\chi_m(a_j^{-1})|a_j|_{M_\mathbb{A}}^m$ and summing over $j$ we get the desired identity.
\end{proof}
This lemma allows us to piece together (\ref{outside-p}) with (\ref{adelic-sum}) by means of the Katz--Shimura rationality result (\ref{K-S}):
{\allowdisplaybreaks\begin{align}\nonumber
\frac{\mathscr{L}(f,\widehat{\varphi})}{\mathrm{\Omega}_p^{k+2m}}&=  \frac{G(\tilde{\varphi}_p)}{\mathrm{\Omega}_p^{k+2m}} \mathcal{L}_{\chi_m}(d^mf^{(p)}) \nonumber \\
&= \frac{G(\tilde{\varphi}_p)}{\mathrm{\Omega}_p^{k+2m}} \left(1-a(p,f)\chi_m(\bar{\mathfrak{p}})+\psi(p)p^{k-1}\chi_m(\bar{\mathfrak{p}}^2)\right)\mathcal{L}_{\chi_m}(d^mf) \nonumber \\ 
&\stackrel{(\ref{K-S})}{=} \frac{G(\tilde{\varphi}_p)}{\mathrm{\Omega}_\infty^{k+2m}} \left(1-a(p,f)\chi_m(\bar{\mathfrak{p}})+\psi(p)p^{k-1}\chi_m(\bar{\mathfrak{p}}^2)\right) \sum_{j=1}^{h^-} \left(\chi_m(a_j^{-1})|a_j|_{M_\mathbb{A}}^m\right) (\delta_k^mf)(x(\mathfrak{a}_j),\omega_\infty(\mathfrak{a}_j)) \nonumber \\
&\stackrel{(\ref{adelic-sum})}{=} \frac{G(\tilde{\varphi}_p)}{\mathrm{\Omega}_\infty^{k+2m}} \tilde{E}(p) j(g_{1,\infty},\mathbf{i})^{k+2m}|\mathrm{det}(g_1^{(\infty)})|_\mathbb{A}^m\left(\frac{\varphi_{\mathbb{Q}}(N_{ns}p^s)}{2\varphi_M(N_{ns}p^s)}\right)^{-1}
L_{\chi_m}(\mathbf{f}_m) \nonumber \\
&= \tilde{C}\tilde{E}(p)\frac{L_{\chi_m}(\mathbf{f}_m)}{\mathrm{\Omega}_\infty^{k+2m}} \label{main}
\end{align}}
where
\begin{equation}\label{euler}
\tilde{E}(p):=1-a(p,f)\chi_m(\bar{\mathfrak{p}})+\psi(p)p^{k-1}\chi_m(\bar{\mathfrak{p}}^2)
\end{equation}
and
\begin{equation}\label{fudge}
\tilde{C}:=G(\tilde{\varphi}_p)j(g_{1,\infty},\mathbf{i})^{k+2m}|\mathrm{det}(g_1^{(\infty)})|_\mathbb{A}^m\left(\frac{\varphi_{\mathbb{Q}}(N_{ns}p^s)}{2\varphi_M(N_{ns}p^s)}\right)^{-1}\; .
\end{equation}
Note that we have (\ref{main}) only under assumption $s\geq \mathrm{max}(1,\mathrm{ord}_p(N_0))$ as this was running assumption (\ref{dominate}) in  Section \ref{sec:hecke}.

\section{Main theorem}\label{sec:mainthm}
Before invoking the main Theorem 4.1 of \cite{HidaNV}, it remains to explain how starting from a normalized Hecke newform $f_0\in S_k(\Gamma_0(N_0),\psi)$, we choose a suitable normalized Hecke eigen-cusp form $f$ such that its arithmetic lift $\mathbf{f}$ is in the automorphic representation $\pi_{\mathbf{f}_0}$ generated by the unitarization $\mathbf{f}_0^u$. If $f_0|T(n)=a(n,f_0)f_0$ we define Satake parameters $\alpha_l,\beta_l\in\mathbb{C}$ by the equations $\alpha_l+\beta_l=a(l,f_0)/l^{(k-1)/2}$ and $\alpha_l\beta_l=\psi(l)$ when $l\nmid N_0$, while we set $\alpha_l=a(l,f_0)/l^{(k-1)/2}$ and $\beta_l=0$ when $l|N_0$. The form $f$ is a normalized Hecke eigen-cusp form with $f|T(n)=a(n,f)f$ and $a(l,f)=a(l,f_0)$ for all primes $l$ outside $\mathrm{lcm}(N_0,p,d_0(M))$. If $\mathrm{lcm}(N_0,p,d_0(M))=N_0$, we set $f:=f_0$, otherwise it is possible to choose $f$ such that for primes $l|\mathrm{lcm}(N_0,p,d_0(M))$ we have
\[ a(l,f)=
\begin{cases}
a(l,f_0) & \text{if } l|N_0,\\
\alpha_ll^{(k-1)/2} & \text{if } l\nmid N_0.\\
\end{cases}
\]
Note that $\mathbf{f}_0^u$ and $\mathbf{f}^u$ generate the same unitary automorphic representation which we denote $\pi_{\mathbf{f}}$ from now.
We also assume that $a(l,f_0)\neq0$ for split $l\in C_{sp}$. 

Let $f_0$ and $f$ be as above. Recall that if $N_0=\prod_ll^{\nu(l)}$ is the prime factorization we denote by $N_{ns}=\prod_{l \, \text{non-split}}l^{\nu(l)}$ its ``non-split'' part. Let $\lambda:M^{\times} \backslash M^{\times}_{\mathbb{A}}\to\mathbb{C}^\times$ be an arithmetic Hecke character that was fixed once and for all at the end of Section \ref{subsec:measure} so that $\lambda(a_\infty)=a_\infty^k$ and $\lambda|_{\mathbb{A}^\times}=\boldsymbol{\psi}^{-1}$, where $\boldsymbol{\psi}$ is a central character of $\mathbf{f}$. 

Let $\varphi:M^{\times} \backslash M^{\times}_{\mathbb{A}}\to\mathbb{C}^\times$ be an anticyclotomic arithmetic Hecke character such that $\varphi(a_\infty)=a_\infty^{m(1-c)}$ for some $m\geq 0$ (i.e. of infinity type $(m,-m)$) and of conductor 
\begin{enumerate}
\item[1)] $N_{ns}p^s$ where $s\geq \mathrm{max}(1,\mathrm{ord}_p(N_0))$ is an \emph{arbitrary} integer or
\item[2)] $N_{ns}$ 
\end{enumerate}
We refer to the first case as $p$-ramified one, and to the second case as $p$-unramified one.
We divide the set of prime factors of $d_0(M)$, $N_0$ and ideal norm of the conductor of $\varphi$ into disjoint union $A\sqcup C$ as follows. If we are in the $p$-ramified case we set $A=\{p\}$, otherwise we set $A=\emptyset$. Set $C=C_0\sqcup C_1$ where $C_1$ is the set of prime factors of $d_0(M)$ and $C_0=C_i\sqcup C_{sp}\sqcup C_r$ so that $C_i$ consists of primes inert in $M$, $C_r=\{2\}$ if $\mathrm{ord}_2(d(M))=2$ with $\nu(2)>2$ and $C_r=\emptyset$ otherwise. Then $C_{sp}$ consists of primes split in $M$ that are not already placed in $A$. Thus, we have three possibilities for prime $p$: in the $p$-ramified case it is placed in set $A$, whereas in the $p$-unramified case it is placed in set $C_{sp}$ when $p|N_0$ or is completely out of this consideration when $p\nmid N_0$.
In the $p$-ramified case we set $\mathcal{A}=\{\bar{\mathfrak{p}}\}$, and in both cases we choose a prime $\bar{\mathfrak{l}}$ over each $l\in C_{sp}$ denoting the set of them by $\mathcal{C}_{sp}=\{\bar{\mathfrak{l}}|l\in C_{sp}\}$. 

Set $\chi_m=\lambda \cdot \varphi \cdot |\cdot|_{M_\mathbb{A}}^m$ and recall that $\chi_m^-:=(\chi_m \circ c)/|\chi_m|$ denotes the unitary projection. Note that $\chi_m^-=(\lambda\varphi)^-$ as the norm character has trivial unitary projection. Depending on the two cases above, the character $\chi_m$ has conductor:
\begin{enumerate} 
\item[1)]  $N_{ns}p^s\prod_{l\in C_{sp}}\mathfrak{l}^{\mathrm{ord}_l(c(\psi))}
$
\item[2)] $N_{ns}\prod_{l\in C_{sp}}\mathfrak{l}^{\mathrm{ord}_l(c(\psi))}
$
\end{enumerate} 
and satisfies condition (F) of Theorem 4.1 of \cite{HidaNV}. Set $N:=\mathrm{lcm}(N_0,p^s)$, where in the $p$-unramified case we assume $s=0$. Recall that $d=|d(M)|$ and $\hat{\pi}_{\mathbf{f}}$ is the base-change lift of the above ${\pi}_{\mathbf{f}}$ to $\mathrm{Res}_{M/\mathbb{Q}}G$. We write $L^{(Nd)}(\frac{1}{2},\hat{\pi}_{\mathbf{f}}\otimes \chi_m^-)$ for the imprimitive $L$-function obtained by removing Euler factors at primes dividing $Nd$ from the primitive one. 
\begin{main} There exists a $W$-valued $p$-adic measure $\mathrm{d}\mu_f$ on $\mathrm{Cl}_M^-(N_{ns}p^\infty)$ whose Mazur--Mellin transform is a $p$-adic analytic Iwasawa function on the $p$-adic Lie group $\mathrm{Hom}_{cont}(\mathrm{Cl}_M^-(N_{ns}p^\infty),W^\times)$, such that whenever $\varphi:M^{\times} \backslash M^{\times}_{\mathbb{A}}\to\mathbb{C}^\times$ is an anticyclotomic arithmetic Hecke character with $\varphi(a_\infty)=a_\infty^{m(1-c)}$ for some $m\geq 0$ and of conductor $N_{ns}p^s$ where $s\geq \mathrm{ord}_p(N)$ is an arbitrary integer or of conductor $N_{ns}$,
we have $p$-adic interpolation of the ``square root'' of central critical Rankin--Selberg $L$-value  $L(\frac{1}{2},\hat{\pi}_{\mathbf{f}}\otimes (\lambda\varphi)^-)$ as follows:
\[\left(\frac{\mathscr{L}(f,\widehat{\varphi})}{\mathrm{\Omega}_p^{k+2m}}\right)^2= \tilde{C}^2\tilde{E}(p)^2 c\frac{\Gamma(k+m)\Gamma(m+1)}{(2\pi\mathbf{i})^{k+2m+1}}E(1/2)E'(m)\frac{L^{(Nd)}(\frac{1}{2},\hat{\pi}_{\mathbf{f}}\otimes (\lambda\varphi)^-)}{(\mathrm{\Omega}_\infty^{k+2m})^2}\quad .\]
The constants $\tilde{E}(p)$  and $\tilde{C}$ are given by (\ref{euler}) and (\ref{fudge}), respectively. If the notation $[\cdot]^*$ means that the factor inside the brackets appears only in the $p$-ramified case, the constant $c=c_1\cdot G\cdot v$ with \[c_1=[\mathrm{exp}(-\frac{2\pi\mathbf{i}}{p^s})]^*\sqrt{d(M)}(2\mathbf{i})^{-k-2m}N^{k+2m}\] 
is given by
\[v=\frac{\prod_{l\in C_{sp}}l^{\nu(l)}}{c_2[p^s(1-\frac{1}{p})^3]^*\prod_{l\in C_i}l^{2\nu(l)}(1+\frac{1}{l})^2(1-\frac{1}{l})
\prod_{l\in C_r\cup C_1, \nu(l)>0}(1-\frac{1}{l})}\; ,\]
\[
G=\left[\chi_{m,p}^-(p^s)^{-1}G(\chi_{m,\bar{\mathfrak{p}}}^-)\right]^*\prod_{l\in C}\chi_{m,l}^-(l^{\nu(l)})^{-1}\prod_{l\in C_{sp},l|c(\psi)}l^{((k/2)+m)(\nu(l)-\mathrm{ord}_l(c(\psi)))}\chi_{m,\bar{\mathfrak{l}}}^-(l^{\nu(l)-\mathrm{ord}_l(c(\psi))})G(\chi_{m,\bar{\mathfrak{l}}}^-) \;,
\]
where $\chi_{m,l}^-=\chi_m^-|_{\mathbb{Q}_l^\times}$, $\chi_{m,\bar{\mathfrak{l}}}^-=\chi_m^-|_{M_{\bar{\mathfrak{l}}}^\times}$, $G(\chi_{m,\bar{\mathfrak{l}}}^-)$ is Gauss sum of $\chi_{m,\bar{\mathfrak{l}}}^-$,
\[ c_2=
\begin{cases}
1 & \text{if } 2\nmid d(M) \text{or } \nu(2)\geq2,\\
6 & \text{if } 4\parallel d(M) \text{and } \nu(2)=0,\\
4 & \text{if } 8\parallel d(M) \text{and } \nu(2)=0,\\
2 & \text{if } 2|d(M) \text{and } \nu(2)=1,\\
\end{cases}
\]
and the modification Euler factors are given by
\[ E(1/2)=\prod_{\mathfrak{l}|l\in C_{sp}}\left(1-\frac{\chi_m^-(\mathfrak{l})\alpha_l}{N(\mathfrak{l})^{1/2}}\right)^{-1}\prod_{\mathfrak{l}|d(M)}\left(1-\frac{\chi_m^-(\mathfrak{l})\alpha_l}{N(\mathfrak{l})^{1/2}}\right)^{-1}\left(1-\frac{\chi_m^-(\mathfrak{l})\beta_l}{N(\mathfrak{l})^{1/2}}\right)^{-1}\; , \]
\[
E'(m)=\frac{
\prod_{l\in C_{sp},l|c(\psi)} \frac{\alpha_l^{\nu(l)-\mathrm{ord}_l(c(\psi))}}{l^{(\nu(l)-\mathrm{ord}_l(c(\psi)))(m+(k-1)/2)}}
\prod_{l\in C_{sp},l\nmid c(\psi)}\alpha_l^{\nu(l)}l^{\nu(l)/2}\chi_m^-(\bar{\mathfrak{l}}^{\nu(l)})\left( 1-\frac{1}{\alpha_ll^{1/2}\chi_m^-(\bar{\mathfrak{l}})}\right)
           }
{\prod_{l\in C_{sp}}\alpha_l^{\nu(l)}l^{\nu(l)/2}\chi_m^-(\mathfrak{l}^{\nu(l)})\left(1-\frac{1}{\alpha_ll^{1/2}\chi_m^-(\mathfrak{l})}\right)}
\; .\]
\end{main}
\begin{proof}
In the $p$-ramified case the theorem follows from (\ref{main}) and Theorem 4.1 of \cite{HidaNV}. 

In the $p$-unramified case, the proof amounts to a tiny part of the argument of the $p$-ramified case. Namely, recall that in this case for $g_{1,p}$ we use $\bigl(\begin{smallmatrix} p^{\nu(p)} & 0 \\ 0 & 1  \end{smallmatrix} \bigr)$ when $p|N_0$ or identity matrix when $p\nmid N_0$. Note that proof of the Lemma \ref{factoring} yields that when $\mathfrak{C}_p=1$ and this choice of $g_{1,p}$ is used, $a\mapsto \mathbf{f}_m(\rho(a)g_1) \chi_m(a)$ factors through $\mathrm{Cl}_M^-(N_{ns})$ because, simply put, one can just ignore the part regarding prime number $p$ in that proof. Thus, the (\ref{adelic-sum}) becomes
\[ L_{\chi_m}(\mathbf{f}_m) 
= \frac{\varphi_{\mathbb{Q}}(N_{ns})}{2\varphi_M(N_{ns})}\sum_{j=1}^{H^-}\chi_m(A_j)\mathbf{f}_m(\rho(A_j)g_1)\;.\]
Then note that the action of such a choice of $g_{1,p}$ on CM points on $Sh$ is already explained in the argument above using Deligne's treatment of $Sh$, and we conclude that the points $[z_1,\rho(A_j^{-1})g_1^{(\infty)}]$ in $Sh(\mathbb{C})$ correspond to 
\[x(\mathfrak{A}_j)=\xw{\mathfrak{A}_j}{\mathcal{W}}\]
on $Sh_{/\mathbb{Q}}$. Finally
{\allowdisplaybreaks\begin{align} 
\nonumber \frac{\mathscr{L}(f,\widehat{\varphi})}{\mathrm{\Omega}_p^{k+2m}}&=\frac{1}{\mathrm{\Omega}_p^{k+2m}}
\int_{\mathrm{Cl}_M^-(N_{ns}p^\infty)}\widehat{\varphi} \mathrm{d}\mu_f \nonumber \\
&= \frac{1}{\mathrm{\Omega}_p^{k+2m}}\sum_{j=1}^{H^-}\widehat{\lambda}(\mathfrak{A}_j^{-1}) \int_{\mathbb{Z}_p^\times}\widehat{\varphi}(\mathfrak{A}_j^{-1}z)\mathrm{d}\mu_{f^{(p)},\mathfrak{A}_j}(z) \nonumber \\
&=\frac{1}{\mathrm{\Omega}_p^{k+2m}}\sum_{j=1}^{H^-}\widehat{\lambda}(\mathfrak{A}_j^{-1}) \widehat{\varphi}(\mathfrak{A}_j^{-1}) \int_{\mathbb{Z}_p^\times} z^m \mathrm{d}\mu_{f^{(p)},\mathfrak{A}_j}(z) \nonumber \\
&= \frac{ G(\tilde{\varphi}_p)}{\mathrm{\Omega}_p^{k+2m}}\sum_{j=1}^{H^-}\widehat{\lambda}(\mathfrak{A}_j^{-1}) \widehat{\varphi}(\mathfrak{A}_j^{-1}) \sum_{u\in (\mathbb{Z}/p^s\mathbb{Z})^\times}(d^mf^{(p)})(x(\mathfrak{a}_{j,u}),\omega_p(\mathfrak{a}_{j,u})) \nonumber
&=\ldots = \tilde{C}\tilde{E}(p)\frac{L_{\chi_m}(\mathbf{f}_m)}{\mathrm{\Omega}_\infty^{k+2m}}
\end{align}}
and the theorem follows.
\end{proof}
\begin{remark}\label{remark} { \em
\begin{itemize}
\item[(1)] Our result covers the case considered in \cite{BDP}.
\item[(2)] Note that the anticyclotomic character $\varphi$ satisfies $\varphi(\bar{\mathfrak{p}})=\varphi(\mathfrak{p})^{-1}$ so $\varphi$ is either unramified at $p$ or has conductor $p^s$ for some $s\geq 1$. Thus, as far as ramification at $p$ is concerned, the only case which the main theorem does not cover is
when $1\leq s < \mathrm{ord}_p(N_0)$.
However, in \cite{HidaNV} Hida is able to compute $L_{\chi_m}(\mathbf{f}_m)$ in this case, but the outcome turns out to be $0$ (\cite{HidaNV} Remark 4.2(c)), hence making the question of $p$-adic interpolation of $L_{\chi}(\mathbf{f})$ over arithmetic $\chi$'s in this case vacuous. In this sense, our result is optimal.
\item[(3)] Note that we opt to work with arithmetic Hecke characters $\varphi$ of conductor $N_{ns}p^s$, $s\geq \mathrm{ord}_p(N_0)$ or $s=0$, just for the sake of simplicity of our exposition, this being the simplest instance when the condition (F) in Theorem 4.1 of \cite{HidaNV} is fulfilled by $\chi_m:=\lambda \cdot \varphi \cdot |\cdot|_{M_\mathbb{A}}^m$. We could have treated characters such that their conductor is $l^{\tilde{\nu}(l)}$ at non-split primes $l|N_{ns}$, where $\tilde{\nu}(l)$ is \textbf{any positive integer} such that $\tilde{\nu}(l)\geq\nu(l)$. Then $\chi_m$ as above has conductor $\prod_{l\in C_{sp}}\mathfrak{l}^{\mathrm{ord}_l(c(\psi))}\prod_{l|N_{ns}}l^{\tilde{\nu}(l)}p^s$ and satisfies condition (F) in Theorem 4.1 of \cite{HidaNV}. Then by literally replacing in the above argument $l^{\nu(l)}$ with $l^{\tilde{\nu}(l)}$ for non-split primes, we would obtain a $W$-valued $p$-adic measure on  $\mathrm{Cl}_M^-(\prod_{l|N_{ns}}l^{\tilde{\nu}(l)}p^\infty)$ whose Mazur--Mellin transform provides $p$-adic interpolation of the ``square root'' of central critical Rankin--Selberg $L$-value  $L(\frac{1}{2},\hat{\pi}_{\mathbf{f}}\otimes (\varphi\lambda)^-)$. 
\item[(4)] We could also give alternative treatment for split primes $l|N_0$ in the sense that we could place them in set $A$ together with $p$. To specify $g_{1,l}$ in this case, we use $\bigl(\begin{smallmatrix} l^{\nu(l)} & 1 \\ 0 & 1  \end{smallmatrix} \bigr)$ instead of $\bigl(\begin{smallmatrix} l^{\nu(l)} & 0 \\ 0 & 1 \end{smallmatrix} \bigr)$ and note that the action of such a choice of $g_{1,l}$ on CM points on $Sh$ is explained in the same way as the corresponding action of $g_{1,p}$ in the argument above using Deligne's treatment of $Sh$. Then if $\varphi$ has conductor $N_0^{(p)}p^s$, $s\geq \mathrm{max}(1,\mathrm{ord}_p(N_0))$ or $s=0$, where $N_0^{(p)}$ is the conductor of $f_0$ outside $p$, we ultimately construct a $W$-valued $p$-adic measure on $\mathrm{Cl}_M^-(N_0^{(p)}p^\infty)$ whose Mazur--Mellin transform interpolates the ``square root'' of central critical Rankin--Selberg $L$-value  $L(\frac{1}{2},\hat{\pi}_{\mathbf{f}}\otimes (\varphi\lambda)^-)$. Then (3) above allows us to cover $\varphi$'s of conductor $\tilde{N}_0^{(p)}p^s$ where $\tilde{N}_0^{(p)}$ is \textbf{any positive integer} divisible by $N_0^{(p)}$. 
\end{itemize} }
\end{remark}

\end{document}